\documentclass[11pt,a4paper,reqno]{amsart}

\usepackage{amsmath,amssymb,amsthm,graphicx,amsfonts}
\usepackage{mathrsfs}
\usepackage{amsxtra}
\usepackage{float}
\usepackage[colorlinks, linkcolor=blue,anchorcolor=Periwinkle,
citecolor=red,urlcolor=Emerald]{hyperref}
\usepackage{enumitem}
\makeatletter
\def\namedlabel#1#2{\begingroup
    #2%
    \def\@currentlabel{#2}%
    \phantomsection\label{#1}\endgroup
}
\makeatother
\usepackage[usenames,dvipsnames]{xcolor}

\usepackage{multirow}

\usepackage[all]{xy}
\usepackage{geometry} 
\usepackage{url}
\usepackage{tikz}\usetikzlibrary{matrix}
\usepackage{verbatim}
\usetikzlibrary{arrows}
\usetikzlibrary{decorations.markings}
\tikzset{->-/.style={decoration={  markings,  mark=at position #1 with
			{\arrow{>}}},postaction={decorate}}}
\tikzset{-<-/.style={decoration={  markings,  mark=at position #1 with
			{\arrow{<}}},postaction={decorate}}}

\theoremstyle{plain}
\newtheorem{theorem}{Theorem}[section]
\newtheorem{thmx}{Theorem}

\newtheorem{lemma}[theorem]{Lemma}
\newtheorem{corollary}[theorem]{Corollary}
\newtheorem{proposition}[theorem]{Proposition}
\theoremstyle{definition}
\newtheorem{definition}[theorem]{Definition}
\newtheorem{construction}[theorem]{Construction}
\newtheorem{example}[theorem]{Example}
\newtheorem{remark}[theorem]{Remark}
\newtheorem{notations}[theorem]{Notations}
\numberwithin{equation}{section}

\def\surf{\mathbf{S}}                       
\def\TT{\mathbf{T}}
\def\PP{\mathbf{P}}
\def\MM{\mathbf{M}}
\def\PPS{\mathbb{PP}(\surf)}
\def\PS{\mathbb{P}(\surf)}
\def\PTS{\mathbb{P}^{\times}(\surf)}
\def\ad{\mathbf{Ad}}
\def\GPS{\overline{\mathbb{P}}(\surf)}
\def\GPTS{\overline{\mathbb{P}}^{\times}(\surf)}
\def\<{\langle}
\def\>{\rangle}

\renewcommand{\k}{\mathbf{k}}
\newcommand{\m}{\mathfrak{m}}
\newcommand{\n}{\mathfrak{n}}
\renewcommand{\l}{\mathfrak{l}}

\newcommand{\x}{\mathfrak{x}}
\newcommand{\y}{\mathfrak{y}}
\newcommand{\M}{\mathfrak{M}}

\renewcommand{\mod}{\operatorname{mod}}

\newcommand{\Int}{\operatorname{Int}}
\newcommand{\Hom}{\operatorname{Hom}}
\newcommand{\Ext}{\operatorname{Ext}}

\newcommand{\aaa}{\operatorname{(I)}}
\newcommand{\bb}{\operatorname{(II)}}
\newcommand{\cc}{\operatorname{(III)}}
\newcommand{\dd}{\operatorname{(IV)}}

\newcommand{\PTGPD}{\operatorname{PGD}}
\newcommand{\TGPD}{\operatorname{GD}}
\newcommand{\PTPD}{\operatorname{PD}}
\newcommand{\TPD}{\operatorname{D}}

\newcommand{\st}{\mathrm{s}\tau\operatorname{-tilt}}
\newcommand{\stp}{s\tau\operatorname{-tiltp}}
\newcommand{\ir}{\mathrm{ind}\tau\operatorname{-rigid}}
\newcommand{\tr}{\tau\operatorname{-rigid}}
\newcommand{\trp}{\tau\operatorname{-rigidp}}
\newcommand{\ttt}{\tau\operatorname{-tilt}}

\begin{document}
	
\title{A geometric model for the module category of a skew-gentle algebra}
	
\date{}
	
\author{Ping He}
\address{PH:
	YanQi Lake Beijing Institute of Mathematical Sciences and Applications, 101408,
	Beijing, China}
\email{pinghe@bimsa.cn}

\author{Yu Zhou}
\address{YZ:
	Yau Mathematical Sciences Center,
	Tsinghua University, 100084,
	Beijing, China}
\email{yuzhoumath@gmail.com}

\author{Bin Zhu}
\address{BZ:
	Department of Mathematical Sciences,
	Tsinghua University, 100084,
	Beijing, China}
\email{zhu-b@tsinghua.edu.cn}

\thanks{The work was supported by National Natural Science Foundation of China (Grants No. 12271279,  12031007).}

\begin{abstract}
    In this article, we realize skew-gentle algebras as skew-tiling algebras associated to admissible partial triangulations of punctured marked surfaces. Based on this, we establish a bijection between tagged permissible curves and certain indecomposable modules, interpret the Auslander-Reiten translation via the tagged rotation, and show intersection-dimension formulas. As an application, we classify support $\tau$-tilting modules via maximal collections of non-crossing tagged generalized permissible curves.
\end{abstract}


\maketitle

\tableofcontents

\section*{Introduction}

Cluster algebras were introduced by Fomin and Zelevinsky around 2000. During the last decade, the cluster phenomenon was spotted in various areas in mathematics, as well as in physics, including geometric topology and representation theory. On the one hand, the geometric aspect of cluster theory was explored by Fomin, Shapiro and Thurston \cite{FST}, where flips of triangulations model mutations of clusters. They constructed
a quiver (and later, Labardini-Fragoso \cite{LF} gave a corresponding potential) from any  triangulation of a marked surface. On the other hand, the categorification of cluster algebras leads to cluster categories of acyclic quivers due to Buan, Marsh, Reineke, Reiten and Todorov \cite{BMRRT} (see also \cite{CCS} for type $A$) and later to generalized cluster categories of quivers with potential due to Amiot \cite{A}, where mutations of cluster tilting objects model mutations of clusters. In \cite{QZ}, Qiu and Zhou established directly a bijection from triangulations of a punctured marked surface with non-empty boundary to cluster tilting objects in the corresponding generalized cluster category, which is compatible with flips and mutations. The endomorphism algebras of certain cluster tilting objects in such categories are skew-gentle algebras \cite{GLS,QZ}. 

Skew-gentle algebras were introduced by Gei{\ss} and de la Pe\~{n}a \cite{GP}, as an important class of representation-tame finite dimensional algebras. The indecomposable modules of a skew-gentle algebra (or more generally, a clannish algebra) were classified by Crawley-Boevey \cite{CB}, Deng \cite{De} and Bondarenko \cite{B}. A basis of the space of morphisms between (certain) indecomposable modules was given by Gei{\ss} \cite{G}. 

Not every skew-gentle algebra is the endomorphism algebra of a cluster tilting object in a generalized cluster category (or more generally, a 2-Calabi-Yau triangulated category). However, a skew-gentle algebra, or more generally, an arbitrary finite dimensional algebra admits a cluster phenomenon via considering the so-called support $\tau$-tilting modules, which were introduced by Adachi, Iyama and Reiten \cite{AIR}. The aim of this paper is to give a geometric model of the module categories of skew-gentle algebras via punctured marked surfaces with non-empty boundary, which is applied to get a complete classification of support $\tau$-tilting modules of skew-gentle algebras.

The geometric interpretations of gentle algebras, as a special class of skew-gentle algebras, have been fully investigated (not only for cluster theory but also for module categories and topological/derived Fukaya categories), cf. \cite{APS,ABCP,BS,BrS,BZ,CCS,CS,DRS,HKK,MP,O,OPS,QQ,QQ2,QZ2,S,ZhZZ}. However, skew-gentle algebras are much more complicated. There are two different approaches to cluster categories from punctured marked surfaces with non-empty boundary: one is using the tagging technology after \cite{FST}, e.g. \cite{BQ,QZ,Sch}; the other is using orbifold model, e.g. \cite{AP}. The orbifold approach was generalized later to give a geometric model of derived categories of skew-gentle algebras \cite{AB,LSV}. In contrast, in this paper, we follow the tagging technology to give a geometric model for the module categories of skew-gentle algebras. The realization of gentle algebras as tiling algebras and the classification of indecomposable objects in the module category of a gentle algebra via permissible curves have been already given in \cite{BS}, which we generalize in the paper to the case of skew-gentle algebras. One important feature in our approach is that we further get an intersection-dimension formula (precisely, an equality between the tagged intersection number of two tagged permissible curves and the sum of dimensions of homomorphism spaces from one corresponding object to the Auslander-Reiten translation of the other), which is applied to classify the support $\tau$-tilting modules of the skew-gentle algebra. We note that it is claimed in \cite{LSV} that a formula connecting morphisms between two indecomposable objects in the derived category of a skew-gentle algebra and well-graded intersections of the corresponding curves is currently in preparation.

A punctured marked surface is a compact oriented surface with a finite set of marked points on its non-empty boundary and with a finite set of punctures in its interior. To any admissible partial triangulation $\TT$ (see Definition~\ref{def:part tri}) of a punctured surface $\surf$, we construct an associated algebra $\Lambda^\TT$, called a skew-tiling algebra (see Definition~\ref{def:st}). The main result in the paper is the following.

\begin{thmx}[Theorems~\ref{prop:st to sg}, \ref{thm:curve and mod}, \ref{thm:tau} and \ref{thm:int-dim}]\label{thm:A}
	Any skew-tiling algebra is a skew-gentle algebra. Conversely, for any skew-gentle algebra $A$, there is a punctured marked surface $\surf$ with an admissible partial triangulation $\TT$ such that $\Lambda^\TT\cong A$ and the following hold.
	\begin{enumerate}
		\item There is a bijection 
		$$\begin{array}{cccc}
			M:&\PTS&\to&\mathcal{S}\\
			&(\gamma,\kappa)&\mapsto&M(\gamma,\kappa)
		\end{array}$$
		where $\PTS$ is the set of tagged permissible curves (up to inverse) on $\surf$ (see Definition~\ref{def:permissible}), and $\mathcal{S}$ is a certain class of indecomposable $A$-modules.
		\item For any pair of tagged permissible curve  $(\gamma,\kappa)$, if $M(\gamma,\kappa)$ is not projective, we have $$\tau M(\gamma,\kappa)=M(\rho(\gamma,\kappa)),$$
		where $\rho$ is the tagged rotation (see Definition~\ref{def:ro}) and $\tau$ is the Auslander-Reiten translation.
		\item For any tagged permissible curves $(\gamma_1,\kappa_1)$ and $(\gamma_2,\kappa_2)$ (not necessarily distinct), we have
		$$\Int((\gamma_1,\kappa_1),(\gamma_2,\kappa_2))=\dim_\k\Hom(M_1,\tau M_2)+\dim_\k\Hom(M_2,\tau M_1)$$
		where $M_i=M(\gamma_i,\kappa_i)$, $i=1,2$, 
		and $\Int$ stands for the intersection number (see Definition~\ref{def:tag int}).
	\end{enumerate}
\end{thmx}

Moreover, the Auslander-Reiten sequence ending at a non-projective string module is described via geometric methods (see Theorem~\ref{thm:tau}), and the dimension of $\Hom(M_1,\tau M_2)$ is also interpreted via the so-called black intersection number (see Definition~\ref{def:black} and Theorem~\ref{thm:bint-dim}). In the case that $A$ is gentle (or equivalently, $\surf$ has no punctures), Theorem~\ref{thm:A} recovers \cite[Theorem~2]{BS} and moreover gives an int-dim formula for their model, where the intersection number $\Int$ is the usual one (see Corollary~\ref{cor:gentle}).

As an application of our main result, we classify the support $\tau$-tilting modules for skew-gentle algebras.

\begin{thmx}[Corollary~\ref{cor:supp}]\label{thm:tau-tilting}
	Under the notion and notations in Theorem~\ref{thm:A}, there is a bijection
	$$\begin{array}{cccc}
		&\TGPD(\surf)&\to&\st A\\
		&R&\mapsto& \bigoplus\limits_{(\gamma,\kappa)\in R}M(\gamma,\kappa)
	\end{array}
	$$
	from the set $\TGPD(\surf)$ of generalized dissections on $\surf$ (see  Definition~\ref{def:t.p.d}), i.e. maximal collections of non-crossing tagged generalized permissible curves (see Definition~\ref{def:gen-perm}), to the set $\st A$ of isoclasses of basic support $\tau$-tilting $A$-modules.
\end{thmx}

The paper is organized as follows. In $\S 1$ we recall the definition of and set the notation for skew-gentle algebras, introduce the notion of skew-tiling algebras associated to admissible partial triangulations of punctured marked surfaces and show that a finite dimensional algebra is skew-gentle if and only if it is skew-tiling. In $\S 2$, we introduce the notion of tagged permissible curves and show that they are in bijection with isoclasses of certain indecomposable modules. In $\S 3$, we recall a description of the Auslander-Reiten translation for skew-gentle algebras and realize it via the tagged rotation. In $\S 4$, we interpret the dimension of the Hom space between two indecomposable modules in $\mathcal{S}$ via the intersection numbers of tagged permissible curves. In $\S 5$, we apply our results to classify support $\tau$-tilting modules of skew-gentle algebras. In $\S 6$, we give an example to illustrate our results.

Throughout this paper, we assume that $\k$ is an algebraically closed field. For any $\k$-algebra $A$, an $A$-module is always a finitely generated right $A$-module. We denote by $\mod A$ the category of finitely generated right $A$-modules.

In Tables~\ref{table1} and \ref{table2}, we list some notations used throughout the paper.
\begin{table}[t]
    \caption{Notations of geometric items}\label{table1}
    \begin{tabular}{|c|cp{3in}|c|}
    \hline
    $\surf=(S,\MM,\PP)$ && a punctured marked surface & Section~\ref{subsec:st} \\ \hline
    $\gamma^{-1}$ && the inverse of a curve $\gamma$ & Definition~\ref{def:curve}\\ \hline
    $\TT$ && an admissible partial triangulation of $\surf$ & Definition~\ref{def:part tri} \\ \hline
		$\Lambda^\TT$ && the skew-tiling algebra arising from $\TT$ & Definition~\ref{def:st}\\ \hline
		$\overline{\gamma}$ && the completion of a curve $\gamma$ & Definition~\ref{def:completion} \\ \hline
		$\PS$ && the set of permissible curves on $\surf$ & Definition~\ref{def:permissible} \\ \hline
		PAS && permissible curve segment & Definition~\ref{def:PAS}\\ \hline
		PPC && prepermissible curve & Definition~\ref{def:prepermissible} \\ \hline
		$\PPS$ && the set of prepermissible curves on $\surf$ & Definition~\ref{def:prepermissible} \\ \hline
		PPCS && prepermissible curve segment & Definition~\ref{def:prepermissible} \\ \hline
        $\GPS$ && the set of generalized permissible curves on $\surf$ & Definition~\ref{def:gen-perm}\\ \hline
        \multirow{2}{*}{$\GPTS$} && the set of tagged generalized permissible curves on $\surf$ & \multirow{2}{*}{Definition~\ref{def:gen-perm}}\\ \hline
		\multirow{3}{*}{$[1]\gamma$ (resp. $\gamma[1]$)} && the curve obtained from $\gamma$ by moving $\gamma(1)$ (resp. $\gamma(0)$) along the boundary clockwise to the next marked point & \multirow{3}{*}{Definition~\ref{def:ro}} \\ \hline
		$\rho(\gamma,\kappa)$ && the tagged rotation of a tagged curve $(\gamma,\kappa)$ & Definition~\ref{def:ro} \\ \hline
        $\gamma_1\sim\gamma_2$ && $\gamma_1$ is homotopic to $\gamma_2$ & Section~\ref{sec:int-dim}\\ \hline
        \multirow{2}{*}{$\gamma_1\cap^\circ\gamma_2$} && the set of interior intersections between $\gamma_1$ and $\gamma_2$ & \multirow{2}{*}{Definition~\ref{def:tag int}} \\ \hline
        \multirow{2}{*}{$\Int^\bullet(\gamma_1,\gamma_2)$} && the number of black interior intersections from $\gamma_1$ to $\gamma_2$ & \multirow{2}{*}{Definition~\ref{def:black}}\\ \hline
		\multirow{2}{*}{$\mathfrak{T}((\gamma_1,\kappa_1),(\gamma_2,\kappa_2))$} && the set of tagged intersections between $(\gamma_1,\kappa_1)$ and $(\gamma_2,\kappa_2)$ & \multirow{2}{*}{Definition~\ref{def:tag int}}\\ \hline
        \multirow{2}{*}{$\mathfrak{T}^\bullet((\gamma_1,\kappa_1),(\gamma_2,\kappa_2))$} && the set of black tagged intersections from $(\gamma_1,\kappa_1)$ to $(\gamma_2,\kappa_2)$  & \multirow{2}{*}{Definition~\ref{def:black}}\\ \hline
		\multirow{2}{*}{$\Int((\gamma_1,\kappa_1),(\gamma_2,\kappa_2))$} && the number of intersections between $(\gamma_1,\kappa_1)$ and $(\gamma_2,\kappa_2)$ & \multirow{2}{*}{Definition~\ref{def:tag int}}\\ \hline
        \multirow{2}{*}{$\Int^\bullet((\gamma_1,\kappa_1),(\gamma_2,\kappa_2))$} && the number of black intersections from $(\gamma_1,\kappa_1)$ to $(\gamma_2,\kappa_2)$ & \multirow{2}{*}{Definition~\ref{def:black}}\\ \hline
		\multirow{2}{*}{$\mathfrak{P}(\gamma_1,\gamma_2)$} && the set of punctured intersections between $\gamma_1$ and $\gamma_2$ & \multirow{2}{*}{Section~\ref{subsec:4.2}}\\\hline
        \multirow{2}{*}{$\mathfrak{P}^\bullet(\gamma_1,\gamma_2)$} && the set of black punctured intersections from $\gamma_1$ to $\gamma_2$ & \multirow{2}{*}{Definition~\ref{def:black}}\\\hline
		GD$(\surf)$ && the set of generalized dissections on $\surf$ & Definition~\ref{def:t.p.d}\\\hline
        \end{tabular}
\end{table}
\begin{table}[t]
    \caption{Notations of quiver representations}\label{table2}
    \begin{tabular}{|c|cp{3in}|c|}
    \hline
    $(Q,Sp,I)$ && a skew-gentle triple & Definition~\ref{def:sg} \\ \hline
    $\sigma$, $\tau$ && functions from $Q_1$ to $\{\pm 1\}$ & Section~\ref{subsec:words}\\ \hline
    $\hat{Q}$ && an extended quiver obtained from $Q$ & Section~\ref{subsec:words} \\ \hline
		\multirow{2}{*}{$\l(-)$} && a bijection from the set of PASs to the set of letters in $\hat{Q}_1$ & \multirow{2}{*}{Section~\ref{subsec:words}} \\ \hline
		$\mathbf{IE}$ && the set of inextensible words & Section~\ref{subsec:preper} \\ \hline
		$\M(-)$ && a map from PPCSs to the set of words & Section~\ref{subsec:preper}\\ \hline
		$\Gamma(-)$ && a map from the set of words to PPCSs & Section~\ref{subsec:preper}\\ \hline
		\multirow{2}{*}{$\hat{Q}(i,\theta)$} && the set of letters in $\hat{Q}_1$ that starts at $i$ with the sign $\theta$ on its tail & \multirow{2}{*}{Section~\ref{subsec:ad}} \\ \hline
		\multirow{2}{*}{$\m(i,\theta)$} && the set of inextensible words that starts at $i$ with the sign $\theta$ on its tail & \multirow{2}{*}{Section~\ref{subsec:ad}}\\ \hline
		$\m_{(i,j)}$ && the subword of a word $\m$ between $i$ and $j$ & Notation~\ref{not:()}\\ \hline
		$F(\m)$ && the completion of a word $\m$ & Section~\ref{subsec:ad}\\ \hline
		$\ad$ && the set of admissible words & Definition~\ref{def:admissible word} \\ \hline
		$A_\m$ && a $\k$-algebra associated to a word $\m\in\ad$ & Section~\ref{subsec:mod and curve} \\ \hline
		\multirow{2}{*}{$M(\m,N)$} && a module associated to an admissible word $\m$ with a one-dimensional $A_\m$-module $N$ & \multirow{2}{*}{Construction~\ref{construct:M(m,N)}}\\ \hline
		$\mathcal{S}$ && the set of all modules $M(\m,N)$ & Section~\ref{subsec:mod and curve}\\ \hline
		$[1]\m$ && the successor of a word $\m$ in $\m(i,\theta)$ & Section~\ref{subsec:ar} \\ \hline 
		$H^{\m,\n}$ && the set of common pairs from $\m$ to $\n$ & Definition~\ref{def:common pair}\\\hline	
	\end{tabular}
\end{table}

\subsection*{Acknowledgments}
We would like to thank Yu Qiu for helpful discussions.  The authors would also like to thank the anonymous referees for a very careful reading and for many comments and suggestions that improved the presentation of the paper.

\section{Skew-gentle algebras and their geometric realization}
\subsection{Skew-gentle algebras}\label{subsec:sg}

In this subsection, we recall the notion of skew-gentle algebras from \cite{G,GP,QZ}.

A \emph{quiver} $Q=(Q_0,Q_1,s,t)$ is a directed graph, which we always assume to be finite, with $Q_0$ the set of vertices, $Q_1$ the set of arrows, and $s,t:Q_1\to Q_0$ two functions, sending an arrow to its start and target respectively. A \emph{loop} (at a vertex $v\in Q_0$) is an arrow $\epsilon\in Q_1$ with $s(\epsilon)=t(\epsilon)$($=v$). A \emph{path} in $Q$ of \emph{length} $n\geq 1$ is a sequence $\alpha_n\cdots\alpha_1$ of arrows such that $t(\alpha_i)=s(\alpha_{i+1})$ for $1\leq i\leq n-1$. In addition, to each vertex $v\in Q_0$, we associate a \emph{trivial} path $e_v$ of length 0 with $s(e_v)=v=t(e_v)$. By $\k Q$ we denote the \emph{path algebra} of $Q$. An ideal $\mathcal{I}$ of $\k Q$ is called \emph{admissible} if $R_Q^m\subseteq \mathcal{I}\subseteq R_Q^2$ for some $m\geq 2$, where $R_Q$ denotes the ideal of $\k Q$ generated by $Q_1$. The quotient algebra $\k Q/\mathcal{I}$ is finite dimensional whenever $\mathcal{I}$ is admissible.

\begin{definition}\label{def:gentle}
	A pair $(Q,I)$ where $Q$ is a quiver and $I$ is a set of paths in $Q$ is called a \emph{gentle pair} if the following hold.
	\begin{enumerate}
		\item[(G1)] Any vertex of $Q$ is the start of at most two arrows and the target of at most two arrows.
		\item[(G2)] The length of any path in $I$ is 2.
		\item[(G3)] For any arrow $\alpha$, there is at most one arrow $\beta$ (resp. $\gamma$) such that $\beta\alpha\in I$ (resp. $\gamma\alpha\notin I$).
		\item[(G4)] For any arrow $\alpha$, there is at most one arrow $\beta$ (resp. $\gamma$) such that $\alpha\beta\in I$ (resp. $\alpha\gamma\notin I$).
	\end{enumerate}
	A finite dimensional basic algebra $A$ is called a \emph{gentle algebra} if it is isomorphic to $\k Q/\langle I\rangle$ with $(Q,I)$ a gentle pair, where $\langle I\rangle$ denotes the ideal of $\k Q$ generated by $I$.
\end{definition}

Let $(Q,I)$  be a gentle pair. It is known that $\k Q/\langle I\rangle$ is finite dimensional if and only if $\langle I\rangle$ is admissible. We have the following observation.

\begin{lemma}\label{lem:gen fin}
	Let $(Q,I)$ be a gentle pair with $\langle I\rangle$ an admissible ideal of $\k Q$. Then $\epsilon^2\in I$ for any loop $\epsilon$, and there is at most one loop at each vertex. 
\end{lemma}

\begin{proof}
	Since $\langle I\rangle$ is admissible, for any loop $\epsilon$ in $Q_1$, there is an integer $m\geq 2$ such that $\epsilon^m\in\langle I\rangle$. The ideal $\langle I\rangle$ consists of the linear combinations of paths which contain a sub-path in $I$. So we have $\epsilon^2\in I$ by (G2) in Definition~\ref{def:gentle}. 
	
	If there are two distinct loops $\epsilon_1,\epsilon_2$ at a vertex $v$, then by (G3) and (G4) in Definition~\ref{def:gentle}, either $\epsilon_1^2,\epsilon_1^2\notin I$ or $\epsilon_1\epsilon_2,\epsilon_2\epsilon_1\notin I$. In the former case, we have $\epsilon_1^m\notin \langle I\rangle$ for any $m\geq 1$, and in the latter case, we have $(\epsilon_1\epsilon_2)^m\notin \langle I\rangle$ for any $m\geq 1$, both of which contradict that $\langle I\rangle$ is admissible.
\end{proof}

\begin{definition}\label{def:sg}
    A triple $(Q,Sp,I)$ consisting of a quiver $Q$, a subset $Sp$ of $Q_0$ and a set $I$ of paths in $Q$ is called \emph{skew-gentle} if $(Q^{sp},I^{sp})$ is a gentle pair, where $Q_0^{sp}=Q_0$, $Q_1^{sp}=Q_1\cup\{\epsilon_i\mid i\in Sp \}$ with $\epsilon_i$ a loop at $i$, and $I^{sp}=I\cup\{\epsilon_i^2\mid i\in Sp \}$. We call  $\epsilon_i$, where $i\in Sp$, a \emph{special loop} and the elements in $Q_1$ \emph{ordinary arrows}.
	
    A finite dimensional basic algebra $A$ is called \emph{skew-gentle} if it is isomorphic to $\k Q^{sp}/\<I^{sg}\>$ for a skew-gentle triple $(Q,Sp,I)$, where $I^{sg}=I\cup \{\epsilon_i^2-\epsilon_i\mid i\in Sp\}$. 
\end{definition}

By definition, for a skew-gentle triple $(Q,Sp,I)$, the algebra $\k Q^{sp}/\<I^{sg}\>$ is obtained from the algebra $\k Q^{sp}/\<I^{sp}\>$ by specializing the nilpotent loops $\epsilon_i, i\in Sp$, to be idempotents, i.e. $\epsilon_i^2=\epsilon_i$.

\begin{remark}\label{rmk:QZ}
    Let $A=\k Q^{sp}/\<I^{sg}\>$ be a skew-gentle algebra for a skew-gentle triple $(Q,Sp,I)$. Then the set $\{e_j\mid j\in Q^\TT_0\setminus Sp\}\cup\{e_i-\epsilon_i,\epsilon_i|i\in Sp\}$ is a complete set of primitive orthogonal idempotents, where $\epsilon_i,i\in Sp$ is the special loop at $i$.
\end{remark}

\begin{lemma}\label{lem:f.d.}
    Let $(Q,Sp,I)$ be a skew-gentle triple. Then there is an isomorphism of vector spaces $\k Q^{sp}/\<I^{sg}\>\cong\k Q^{sp}/\<I^{sp}\>$. In particular, the quotient algebra $\k Q^{sp}/\<I^{sg}\>$ is finite dimensional if and only if $\k Q^{sp}/\<I^{sp}\>$ is finite dimensional.
\end{lemma}

\begin{proof}
	The vector space $\k Q^{sp}/\<I^{sp}\>$ has a basis consisting of the residue classes (modulo $I^{sp}$) of paths in $Q^{sp}$ which do not contain any sub-path in $I^{sp}$. By the construction of $I^{sg}$, the residue classes (modulo $I^{sg}$) of the same paths form a basis of $\k Q^{sp}/\<I^{sg}\>$. Thus, we get an isomorphism of the vector spaces.
\end{proof}

Combining Lemma~\ref{lem:gen fin} and Lemma~\ref{lem:f.d.}, we have the following.

\begin{lemma}\label{lem:sg fin}
	Let $(Q,Sp,I)$ be a skew-gentle triple with $\k Q^{sp}/\<I^{sg}\>$ finite dimensional. Then there is at most one loop at each vertex, and for each loop $\epsilon$, either $\epsilon^2$ or $\epsilon^2-\epsilon$ is in $I^{sg}$, or equivalently, $\epsilon^2$ is in $I^{sp}$.
\end{lemma}

\subsection{Skew-tiling algebras}\label{subsec:st}

A \emph{punctured marked surface} in the sense of \cite{FST} is a triple $\surf=(S,\MM,\PP)$, where $S$ is a compact oriented surface with nonempty boundary $\partial S$, $\MM\subset \partial S$ is a finite set of marked points on the boundary and $\PP\subset S\setminus\partial S$ is a finite set of punctures in the interior of $S$.

A connected component of $\partial S$ is called a \emph{boundary component} of $\surf$. A boundary component $B$ of $\surf$ is called \emph{unmarked} if $\MM\cap B=\emptyset$. A \emph{boundary segment} is the closure of a component of $\partial S\setminus \MM$. 

\begin{definition}\label{def:curve}
    A \emph{curve} on a punctured marked surface $\surf$ is a continuous map $\gamma:[0,1]\longrightarrow S$ such that
    \begin{enumerate}
        \item $\gamma(0), \gamma(1)\in \MM\cup \PP$ and $\gamma(t)\in S\setminus(\partial S\cup \PP)$ for $0<t<1$; and
        \item $\gamma$ is neither null-homotopic nor homotopic to a boundary segment. 
    \end{enumerate}
    The \emph{inverse} $\gamma^{-1}$ of a curve $\gamma$ is defined as $\gamma^{-1}(t)=\gamma(1-t), t\in[0,1]$. We always consider curves on $\surf$ up to homotopy relative to their endpoints and up to inverse.
\end{definition}


\begin{definition}\label{def:part tri}
	A \emph{partial triangulation} $\TT$ of $\surf$ is a collection of curves such that they do not intersect themselves or each other in $S\setminus(\MM\cup \PP)$. An \emph{admissible partial triangulation} $\TT$ is a partial triangulation where each puncture lies in the interior of a monogon of $\TT$ such as in Figure~\ref{fig:self-fold}. Note in particular that in an admissible partial triangulation, the endpoints of each curve lie in $\MM$.
	\begin{figure}[htpb]\centering
		\begin{tikzpicture}[xscale=1,yscale=1]
			\draw[ultra thick]plot[smooth,tension=2](-.5,1)to(.5,1);
			\draw[red,thick](0,1)to[out=-135,in=60](-1,0)to[out=-120,in=180](0,-1)to[out=0,in=-60](1,0)to[out=120,in=-45](0,1);
			\draw(0,1)node{$\bullet$};
			\draw[thick](0,0)node{$\bullet$};
		\end{tikzpicture}
		\caption{A once-punctured monogon}
		\label{fig:self-fold}
	\end{figure}
\end{definition}

Let $\TT$ be an admissible partial triangulation of a punctured marked surface $\surf$. Then $\surf$ is divided by $\TT$ into a collection of regions. The following types of regions are specially important to us:
\begin{enumerate}
	\item[$\aaa$] $m$-gons ($m\geq 3$) without any unmarked boundary components nor punctures in their interiors;
	\item[$\bb$] monogons with exactly one unmarked boundary component but no punctures in their interiors;
	\item[$\cc$] digons with exactly one unmarked boundary component but no punctures in their interiors;
	\item[$\dd$] once-punctured monogons, that is, monogons with exactly one puncture but no unmarked boundary components in their interiors;
\end{enumerate}
Note that not every region is necessarily of one of the types above (e.g. when $\surf$ is an unpunctured torus and $\TT$ contains exactly one curve, see Figure~\ref{fig:other region}) and that the number of regions of type (IV) is $|\PP|$.

\begin{figure}[htpb]\centering
	\begin{tikzpicture}[scale=.5]
		\draw(0,0) ellipse (5 and 3);
		\draw[bend left=50](-2.5,-.2)to(2.5,-.2);
		\draw[bend right=50](-2.2,.05)to(2.2,.05);
		\draw[fill=gray!20](1,-2)to[out=80,in=140](2,-1)to[out=-40,in=10](1,-2);
		
		\draw[red,thick,bend left=50](.45,-.9)to(.55,-3);
		\draw[red,thick,dashed,bend right=40](.3,-.95)to(.3,-3);
		\draw(1,-2)node[black]{$\bullet$}(1.3,-2.5)node[red]{$\TT$}(2.3,-.7)node{$\partial S$};
	\end{tikzpicture}
	\caption{A region of an admissible partial triangulation, which is not of one of the types (I)-(IV)}
	\label{fig:other region}
\end{figure}   

\begin{definition}\label{def:st}
	Let $\TT$ be an admissible partial triangulation of a punctured marked surface $\surf$. The \emph{skew-tiling algebra} of $\TT$ is defined to be $\Lambda^\TT=\k Q^\TT/\langle R^\TT\rangle$, where the quiver $Q^\TT$ and the relation set $R^\TT$ are given by the following.
	\begin{itemize}
		\item The vertices in $Q^\TT_0$ are (indexed by) the curves in $\TT$.
		\item There is an arrow $\alpha\in Q^\TT_1$ from $i$ to $j$ whenever the corresponding curves $i$ and $j$ share an endpoint $p_{\alpha}\in \MM$ such that $j$ follows $i$ when going around $p_\alpha$ in the anticlockwise direction. Note that by this construction, each loop in $Q^\TT$ is at (the vertex indexed by) a curve in $\TT$ whose endpoints coincide. 
		\item The relation set $R^\TT$ consists of
		\begin{itemize}
			\item[(R1)] $\epsilon^2-\epsilon$ (resp. $\epsilon^2$) if $\epsilon$ is a loop such that the curve (corresponding to) $s(\epsilon)=t(\epsilon)$ does (resp. does not) cut out a region of type $\dd$; and
			\item[(R2)] $\alpha\beta$ if $p_{\beta}\neq p_{\alpha}$ (see Figure~\ref{fig:diff vertex}), or the endpoints of the curve (corresponding to) $t(\beta)=s(\alpha)$ coincide and we are in one of the situations in Figure~\ref{fig:R2 loop}.
			\begin{figure}[htbp]\centering
				\begin{tikzpicture}[xscale=1,yscale=.5]
					\draw[ultra thick]plot[smooth,tension=1](-2,2)to(2,2);
					\draw[ultra thick]plot[smooth,tension=1](-2,-2)to(2,-2);
					\draw[red,thick,bend left=10](0,2)to(0,-2);
					\draw[red,thick,bend right=10](0,-2)to(1,-.5);
					\draw[red,thick,bend left=10](0,2)to(1,.5);
					\draw[thick,bend right=20,->-=.7,>=stealth](.75,-1)to(.15,-1);
					\draw[thick,bend right=20,->-=.7,>=stealth](.15,1)to(.75,1);
					\draw(0,2)node[above]{$p_{\alpha}$}(0,-2)node[below]{$p_{\beta}$}(.5,1)node[below]{$\alpha$}(.5,-1)node[above]{$\beta$};
					\draw (0,2)node{$\bullet$} (0,-2)node{$\bullet$};
				\end{tikzpicture}
				\caption{Relations in (R2), the case $p_{\beta}\neq p_{\alpha}$}
				\label{fig:diff vertex}
			\end{figure}
			
			\begin{figure}[htpb]\centering
				\begin{tikzpicture}[xscale=1,yscale=.8]
					\draw[ultra thick,bend right=20](-.5,2.1)to(.5,2.1);
					\draw[red,thick](0,2)to[out=-140,in=90](-2,0);
					\draw[red,thick](2,0)to[out=90,in=-40](0,2);
					\draw[red,dashed](-2,0)to[out=-90,in=180](-.5,-1);
					\draw[red,dashed](.5,-1)to[out=0,in=-90](2,0); 
					\draw(0,2)node{$\bullet$}(0,.8);
					\draw[red,thick,bend right=5](0,2)to(-.8,0);
					\draw[red,thick,bend left=5](0,2)to(.8,0);
					\draw[thick,->-=.6,>=stealth,bend right=5](-.78,1.5)to(-.21,1.5);
					\draw[thick,-<-=.6,>=stealth,bend left=5](.78,1.5)to(.21,1.5);
					\draw(-.5,1.3)node{$\alpha$}(.5,1.3)node{$\beta$}; 
					\draw(0,2)node{$\bullet$};
					\draw(0,2)node[above]{$p_\alpha=p_\beta$};
				\end{tikzpicture}
				\qquad
				\begin{tikzpicture}[xscale=1,yscale=1]
					\draw[ultra thick,bend right=20](-.5,2.1)to(.5,2.1);
					\draw[red,thick](0,2)to[out=-120,in=80](-.8,0);
					\draw[red,dashed](-.8,0)to[out=-70,in=180](-.5,-.5);
					\draw[red,dashed](.5,-.5)to[out=0,in=-110](.8,0);
					\draw[red,thick](.8,0)to[out=100,in=-60](0,2);
					\draw[red,thick,bend right=20](0,2)to(-2,0);
					\draw[red,thick,bend left=20](0,2)to(2,0);
					\draw[thick,->-=.6,>=stealth,bend right=5](-.81,1.5)to(-.29,1.5);
					\draw[thick,-<-=.6,>=stealth,bend left=5](.81,1.5)to(.29,1.5);
					\draw(-.8,1.3)node{$\beta$}(.8,1.3)node{$\alpha$};
					\draw(0,2)node{$\bullet$};
					\draw(0,2)node[above]{$p_\alpha=p_\beta$};
				\end{tikzpicture}
				\caption{Relations in (R2), the case $p_\beta=p_\alpha$ and $s(\alpha)=t(\beta)$ is a loop}
				\label{fig:R2 loop}
			\end{figure}
		\end{itemize}
	\end{itemize}
\end{definition}

Note that by the above construction, each region of type (II) or (IV) gives rise to a loop while any region of type (I) or (III) does not. In particular, we have the following.

\begin{lemma}\label{lem:loop}
	If each region of $\TT$ is of one of the types $\aaa$-$\dd$, then there is a bijection between the loops in $Q^\TT$ and the regions of type $\bb$ or $\dd$.
\end{lemma}

For a curve $i\in\TT$ which cuts out a region of type (II) or (IV), we denote by $\epsilon_i$ the loop in $Q^\TT$ at $i$. See Section~\ref{sec:example} for an example of a skew-tiling algebra.

Note that, for the case when $\PP$ is empty, any partial triangulation is admissible. In this case, skew-tiling algebras are tiling algebras introduced in \cite{BS}. So we have the following result.

\begin{theorem}[{\cite{BS}}]\label{prop:bau}
	Let $\TT$ be an admissible partial triangulation of a punctured marked surface $(S,\MM,\PP)$ with $\PP=\emptyset$. Then the algebra $\Lambda^\TT$ is a gentle algebra. Conversely, for any gentle algebra $A$, there is an admissible partial triangulation $\TT$ of a punctured marked surface $(S,\MM,\PP)$ with $\PP=\emptyset$ such that $\Lambda^\TT\cong A$ and $\TT$ divides $(S,\MM,\PP)$ into a collection of regions of types $\aaa$-$\cc$.
\end{theorem}

We generalize this to the case when $\PP$ might not be empty.

\begin{theorem}\label{prop:st to sg}
	Let $\TT$ be an admissible partial triangulation of a punctured marked surface $\surf$. Then the skew-tiling algebra $\Lambda^\TT$ is skew-gentle. Conversely, for any skew-gentle algebra $A$, there is a punctured marked surface $\surf$ with an admissible partial triangulation $\TT$ such that $\Lambda^{\TT}\cong A$ and $\surf$ is divided by $\TT$ into a collection of regions of types $\aaa$-$\dd$.	
\end{theorem}

\begin{proof}
	Let $\TT$ be an admissible partial triangulation of a punctured marked surface $\surf=(S,\MM,\PP)$. Let $\widetilde{S}$ be the surface obtained from $S$ by replacing each puncture $p\in \PP$ with an unmarked boundary component $\partial_p$ and denote by $\widetilde{\surf}=(\widetilde{S},M,\emptyset)$. Then the set $\TT$ becomes an admissible partial triangulation of $(\widetilde{S},M,\emptyset)$, which is denoted by $\TT'$ to avoid misunderstanding. By Theorem~\ref{prop:bau}, $\Lambda^{\TT'}$ is a gentle algebra, i.e. $(Q^{\TT'},R^{\TT'})$ is a gentle pair and $\Lambda^{\TT'}$ is finite dimensional. Let $\TT^{sp}$ be the subset of $\TT$ consisting of the curves that cut out regions of type $\dd$ in $\surf$. Then each curve in $\TT^{sp}$ cuts out a region of type $\bb$ in $\widetilde{\surf}$. So we have $Q^{\TT}=Q^{\TT'}$ and $R^\TT=(R^{\TT'}\setminus\{\varepsilon_i^2\mid i\in \TT^{sp}\})\cup\{\varepsilon_i^2-\varepsilon_i\mid i\in \TT^{sp}\}$. By Definition~\ref{def:sg} and Lemma~\ref{lem:f.d.}, $\Lambda^\TT$ is a skew-gentle algebra.
	
	Conversely, let $A=\k Q^{sp}/I^{sg}$ be a skew-gentle algebra for a skew-gentle triple $(Q,Sp,I)$. By Definition~\ref{def:sg} and Lemma~\ref{lem:f.d.}, the corresponding algebra $\k Q^{sp}/I^{sp}$ is a gentle algebra. So by Theorem~\ref{prop:bau}, there is a punctured marked surface $\surf'=(S',\MM',\PP')$ with $\PP'=\emptyset$ and an admissible partial triangulation $\TT'$ which divides $\surf'$ into a collection of regions of types (I)-(III) such that $Q^{\TT'}=Q^{sp}$ and $R^{\TT'}=I^{sp}$. Then by Lemma~\ref{lem:loop}, the loops in $Q^{sp}$ are in bijection with the regions of type $\bb$ on $\surf'$. Denote by $\mathcal{B}^{sp}$ the set of unmarked boundary components in regions of type $\bb$ whose corresponding loops are $\epsilon_i$, with $i\in Sp$.  Let $\surf$ be the punctured marked surface obtained from $\surf'$ by replacing each unmarked boundary component in $\mathcal{B}^{sp}$ with a puncture. Then $\TT'$ becomes an admissible partial triangulation of $\surf$, denoted by $\TT$, which divides $\surf$ into a collection of regions of types (I)-(IV). So we have $Q^{\TT}=Q^{\TT'}=Q^{sp}$ and $R^{\TT}=R^{\TT'}\setminus\{\varepsilon_i^2\mid i\in Sp \}\cup\{\varepsilon_i^2-\varepsilon_i\mid i\in Sp \}=I^{sg}$. Hence the skew-tiling algebra $\Lambda^{\TT}$ is isomorphic to $A$.
	
\end{proof}

An immediate consequence is the following.

\begin{corollary}
	Let $A$ be a finite dimensional algebra. Then $A$ is skew-gentle if and only if it is skew-tiling.
\end{corollary}

\section{Geometric interpretation of modules}\label{sec:module}

Throughout the rest of the paper, let $A=\k Q^{sp}/I^{sg}$ be a skew-gentle algebra for a skew-gentle triple $(Q,Sp,I)$. By Theorem~\ref{prop:st to sg}, there is a punctured marked surface $\surf=(S,\MM,\PP)$ with an admissible partial triangulation $\TT$ satisfying $Q^\TT=Q^{sp}$ and $R^\TT=I^{sg}$ (in particular, the skew-tiling algebra $\Lambda^{\TT}$ is isomorphic to $A$), and such that $\surf$ is divided by $\TT$ into a collection of regions of types (I)-(IV). For any curve $\gamma$ on $\surf$, we always assume that $\gamma$ has a minimal intersection number with $\TT$. 

\begin{definition}[Completions of curves, {\cite[Definition~3.1]{QZ}}]\label{def:completion}
	For a curve $\gamma$ on $\surf$, we define its \emph{completion} $\overline{\gamma}$ as shown in Figure~\ref{fig:completion of curve}, where $\overline{\gamma}=\gamma$ if and only if both of the endpoints of $\gamma$ are in $\MM$ (i.e. the first case), and where the endpoints of $\gamma$ could coincide (for the first/third case).
	\begin{figure}[h]\centering
		\begin{tikzpicture}[xscale=2, yscale=.5,rotate=90]
			\draw[ultra thick]plot [smooth, tension=1] coordinates {(170:4.5) (180:4) (190:4.5)}; 
			\draw[ultra thick]plot [smooth, tension=1] coordinates {(20:1.5) (0:1) (-20:1.5)};
			\draw[blue, thick,->-=.5,>=stealth](0:1)to(180:4);
			\draw[blue](-.5,.2)node{$\gamma$} (-.5,-.2)node{$\overline{\gamma}$};
			\draw[thick](180:4)node{$\bullet$}(0:1)node{$\bullet$};
		\end{tikzpicture}\qquad
		\begin{tikzpicture}[xscale=2, yscale=.5,rotate=90]
			\draw[ultra thick]plot [smooth, tension=1] coordinates {(170:4.5) (180:4) (190:4.5)};  
			\draw[blue, thick,->-=.5,>=stealth]plot [smooth, tension=1] coordinates {(180:4) (90:.5) (0:2) (-90:.5) (180:4)};
			\draw[blue, thick,->-=.5,>=stealth](0:1)to(180:4);
			\draw[blue](-.5,.2)node{$\gamma$};
			\draw[blue](-.5,-.5)node[right]{$\overline{\gamma}$};
			\draw[thick](180:4)node{$\bullet$}(0:1)node{$\bullet$};
		\end{tikzpicture}\qquad
		\begin{tikzpicture}[xscale=.7, yscale=.3]
			\draw[blue, thick,-<-=.5,>=stealth] (0,0) ellipse (1.5 and 5.5);
			\draw[blue, thick,->-=.5,>=stealth](0,2.5)to(0,-2.5);
			\draw[blue](-.5,.2)node{$\gamma$};
			\draw[blue](1.6,0)node[right]{$\overline{\gamma}$};
			\draw[thick](0,2.5)node{$\bullet$}(0,-2.5)node{$\bullet$};
		\end{tikzpicture}
		\caption{Completions of curves}
		\label{fig:completion of curve}
	\end{figure}
\end{definition}

The notion of completion of curves will be used to define permissible curves (Definition~\ref{def:permissible}) and to describe the middle terms of certain Auslander-Reiten sequences (Theorem~\ref{thm:tau} and Remark~\ref{rmk:AR}).

\begin{definition}[Tagged permissible curves]\label{def:permissible}
	A curve $\gamma$ on $\surf$ is called \emph{permissible} with respect to $\TT$ if the following conditions hold.
	\begin{enumerate}
		\item[\namedlabel{itm:t1}{(T1)}] The starting/ending segment of $\gamma$ has one of the forms shown in Figure~\ref{fig:ro}, where the red curves lie in $\TT$ or are boundary segments.
		\item[\namedlabel{itm:t2}{(T2)}] If $\gamma$ crosses consecutively two sides $x,y$ of a region $\Delta$ of $\TT$, then $x,y$ are neighboring sides and $\gamma$ cuts out an angle of $\Delta$ as shown in Figure~\ref{fig:local tri}.
		\item[\namedlabel{itm:p1}{(P1)}] The curve $\gamma$ does not cut out a once-punctured monogon by its self-intersection, see Figure~\ref{fig:permissible};
		\item[\namedlabel{itm:p2}{(P2)}] If $\gamma(0),\gamma(1)\in\PP$ then the completion $\overline{\gamma}$ is not a proper power of a closed curve in the sense of the multiplication in the quotient of the fundamental group of $\surf$ by the squares of simple closed curves that encloses a puncture.	
	\end{enumerate}
	We denote by $\PS$ the set of permissible curves.
	\begin{figure}[htpb]
		\begin{tikzpicture}[scale=1.25]
			\clip(-1.2,-1.2) rectangle (1.2,1.2);
			\draw[red, thick, bend left=10](0,1)to(0,-1);
			\draw[red, thick](0,1)to(1,0);
			\draw[red,thick,dashed,bend right=10](0,-1)to(1,0);
			\draw[blue, thick, bend right=10](-1,0)to(1,0);
			\draw[blue,thick](-1,0)node[below]{$\gamma$};
			\draw[thick](0,1)node{$\bullet$}(0,-1)node{$\bullet$}(1,0)node{$\bullet$};
		\end{tikzpicture}\qquad
		\begin{tikzpicture}[scale=1.25]
			\clip(-1.2,-1.2) rectangle (1.2,1.2);
			\draw[red,thick,bend right=60](0,1)to(0,-1);
			\draw[red,thick,bend left=60](0,1)to(0,-1);
			\draw[ultra thick, fill=gray!20](0,0)circle(.1);
			\draw[blue,thick,smooth](-1,.4)to[out=-5,in=170](0,.3)to[out=-10,in=90](.3,0)to[out=-90,in=70](0,-1);
			\draw[blue](-1,0.1)node{$\gamma$};
			\draw[thick](0,1)node{$\bullet$}(0,-1)node{$\bullet$};
		\end{tikzpicture}\qquad
		\begin{tikzpicture}[scale=1.25]
			\clip(-1.2,-1.2) rectangle (1.2,1.2);
			\draw[ultra thick, fill=gray!20](0,0) circle (.1);
			\draw[red,thick](0,1)to[out=-135,in=80](-.6,0)to[out=-100,in=180](0,-1)to[out=0,in=-80](.6,0)to[out=100,in=-45](0,1);
			\draw[blue,thick](-1.2,.3)to[bend left=5](0,.3)to[out=-10,in=90](.3,0)to[out=-90,in=0](0,-.3)to[out=180,in=-90](-.3,0)to[out=90,in=-100](0,1);
			\draw[blue] (-1,0)node{$\gamma$};
			\draw(0,1)node{$\bullet$};
		\end{tikzpicture}\qquad
		\begin{tikzpicture}[scale=1.25]
			\clip(-1.2,-1.2) rectangle (1.2,1.2);
			\draw[red,thick](0,1)to[out=-135,in=80](-.6,0)to[out=-100,in=180](0,-1)to[out=0,in=-80](.6,0)to[out=100,in=-45](0,1);
			\draw[blue,thick](0,0)to(0,-1.2);
			\draw[blue] (-.1,-.4)node[right]{$\gamma$};
			\draw(0,1)node{$\bullet$};
			\draw(0,0)node{$\bullet$};
		\end{tikzpicture}
		\caption{Condition \ref{itm:t1}}
		\label{fig:ro}
	\end{figure}
	\begin{figure}[htpb]\centering
		\begin{tikzpicture}[xscale=1,yscale=1]
			\draw[ultra thick]plot[smooth,tension=2](-1,1)to(1,1);
			\draw[red,thick,bend right=10](0,1)to(-1,0);
			\draw[red,thick,bend left=10](0,1)to(1,0);
			\draw[blue,thick,bend right](-1,.3)to(1,.3);
			\draw[red,thick](-1,0)node[below]{$x$}(1,0)node[below]{$y$};
			\draw[blue,thick](0,0)[below]node{$\gamma$};
		\end{tikzpicture}
		\caption{Condition \ref{itm:t2}}
		\label{fig:local tri}
	\end{figure}
	\begin{figure}[htpb]\centering
		\begin{tikzpicture}[xscale=2, yscale=.4,rotate=90]
			\draw[blue, thick]plot [smooth, tension=1] coordinates {(180:4) (90:.5) (0:2) (-90:.5) (180:4)};
			\draw[thick](180:4)node{$\bullet$}(0:1)node{$\bullet$};
		\end{tikzpicture}\qquad
		\begin{tikzpicture}[xscale=2, yscale=.4,rotate=90]
			\draw[blue, thick]plot [smooth, tension=1] coordinates {(185:4) (90:.5) (0:2) (-90:.5) (175:4)};
			\draw[thick](0:1)node{$\bullet$};
		\end{tikzpicture}
		\caption{Curves not satisfying condition \ref{itm:p1}}\label{fig:permissible}
	\end{figure}
	
	A \emph{tagged permissible curve} on $\surf$ is a pair $(\gamma,\kappa)$, where $\gamma\in\PS$ and $$\kappa:\{t|\gamma(t)\in \PP\}\to\{0,1\}$$ is a map. The \emph{inverse} of a tagged permissible curve $(\gamma,\kappa)$ is $(\gamma^{-1},\kappa^{-1})$, where $\gamma^{-1}(t)=\gamma(1-t)$ and $\kappa^{-1}(t)=1-\kappa(t)$. Let $\PTS$ denote the set of (representatives of) equivalence classes of tagged permissible curves on $\surf$ under taking inverse.
\end{definition}

The conditions \ref{itm:t1} and \ref{itm:t2} are basically from \cite{BS} (although we modified \ref{itm:t1} a little in order to pick up a suitable representative in each equivalence class) while the conditions \ref{itm:p1} and \ref{itm:p2} are from \cite{QZ} (but note that we fixed \ref{itm:p2}).

\begin{example}\label{example:P2}
    Let $c$ be a simple closed curve winding around two punctures $p,q$ clockwise, and let $a$ (resp. $b$) be the simple closed curve winding around the puncture $p$ (resp. $q$) clockwise. See the left picture of Figure~\ref{fig:P2}. Then we have $c=ab$ in the fundamental group of $\surf$. Take a curve $\gamma$ connecting $p$ and $q$ in the way shown in the left picture. The completion $\overline{\gamma}$ of $\gamma$ is shown in the right picture. We have $\overline{\gamma}=a^{-1}b^{-1}abab=(ab)^3=c^3$ in the quotient group of the fundamental group of $\surf$ by $a^2$ and $b^2$. So $\gamma$ does not satisfy \ref{itm:p2}. 
	\begin{figure}[htpb]\centering
		\begin{tikzpicture}[scale=1.5]
			\draw(-.8,.2)node{$\bullet$}(.8,.2)node{$\bullet$};
			\draw[red,thick,-<-=.5,>=stealth](-.8,.2) circle (.4);
			\draw[red,thick,-<-=.0,>=stealth](.8,.2) circle (.4);
			\draw[blue,thick,->-=.5,>=stealth](.8,.2)to[out=170,in=0](-.8,.4)to[out=180,in=90](-1,.2)to[out=-90,in=175](-.8,0)to[out=-5,in=-175](.8,0)to[out=5,in=-90](1,.2)to[out=90,in=0](.8,.4)to[out=180,in=10](-.8,.2);
			
			\draw[blue,thick,-<-=.5,>=stealth](0,.1)ellipse(1.5 and .8);
			\draw[blue](0,-.2)node{$\gamma$}(1.7,.1)node{$c$}(-.55,.12)node[black]{$p$}(-.6,-.3)node[red]{$a$}(.55,.12)node[black]{$q$}(.6,-.3)node[red]{$b$}(0,-.6)node{};
		\end{tikzpicture}
		\qquad
		\begin{tikzpicture}[scale=1.5]
			\draw(-.8,.2)node{$\bullet$}(.8,.2)node{$\bullet$}(1.65,.2)node[cyan]{$\overline{\gamma}$}(.65,.3)node{$q$}(-.65,.3)node{$p$};
			
			\draw[cyan,thick,-<-=.62,>=stealth](.8,.71)to[out=0,in=90](1.21,.35)to[out=-90,in=0](0,-.23)to[out=180,in=-90](-1.21,.35)to[out=90,in=180](-.8,.71)to[out=0,in=153](0,.45)to[out=-27,in=153](.65,.09)to[out=-27,in=-140](.92,.07)to[out=40,in=-30](.91,.34)to[out=150,in=-24](0,.8)to[out=156,in=0](-.7,1)to[out=180,in=90](-1.5,.35)to[out=-90,in=180](0,-.5)to[out=0,in=-90](1.5,.35)to[out=90,in=0](.7,1)to[out=180,in=24](0,.8)to[out=-156,in=30](-.91,.35)to[out=-150,in=140](-.92,.07)to[out=-40,in=-153](-.65,.09)to[out=27,in=-153](0,.45)to[out=27,in=180](.8,.71);
		\end{tikzpicture}
		\caption{A curve not satisfying \ref{itm:p2}}
		\label{fig:P2}
	\end{figure}
\end{example}

In this section, we show that a certain class $\mathcal{S}$ of indecomposable $A$-modules can be realized as tagged permissible curves on the surface $\surf$.

\subsection{Letters and permissible arc segments}\label{subsec:words}

As in \cite{G,QZ}, let $\sigma,\tau: Q^{sp}_1\rightarrow \{\pm1\}$ be two maps defined in such a way that for any path $\alpha\beta$ in $Q^{sp}$ of length 2, we have that $\alpha\beta\in I^{sp}$ if and only if $\sigma(\alpha)=\tau(\beta)$. These two maps can be regarded as labeling signs on the head and tail of an arrow, respectively. Note that if $s(\alpha)=s(\alpha')$ (resp. $t(\beta)=t(\beta')$) and $\alpha\neq\alpha'$ (resp. $\beta\neq\beta'$) then $\sigma(\alpha)=-\sigma(\alpha')$ (resp. $\tau(\beta)=-\tau(\beta')$), and also note that for any loop $\epsilon$ in $Q^{sp}$, we always have $\sigma(\epsilon)=\tau(\epsilon)$, since $\epsilon^2\in I^{sp}$ by Lemma~\ref{lem:sg fin}.
\begin{example}
    In the following quiver $Q^{sp}$ with functions $\sigma$ and $\tau$, we have $\sigma(\alpha)=1$, $\tau(\alpha)=1$, $\sigma(\alpha')=-1$, $\tau(\alpha')=1$, $\sigma(\beta)=-1$, $\tau(\beta)=1$, $\sigma(\epsilon)=\tau(\epsilon)=1$, which imply $I^{sp}=\{\alpha\beta,\epsilon^2\}$.
    $$\xymatrix@R=1.5em@C=4em{
		&&\cdot\\
        \cdot\ar@(ul,dl)_{\epsilon}_(.2){+}_(.8){+}\ar[r]^{\beta}^(.2){-}^(.8){+}&\cdot\ar[ur]^{\alpha}^(.2){+}^(.8){+}\ar[dr]_{\alpha'}_(.2){-}_(.8){+}&\\
		&&\cdot
	}$$
    Note that different signs may give the same relation set, e.g. if we take $\tau(\alpha)=-1$, $\sigma(\alpha)=-1$, $\sigma(\alpha')=1$ and $\tau(\beta)=-1$ instead, the relations set $I^{sp}$ does not change.
\end{example}

We construct a new quiver $\hat{Q}=(\hat{Q}_0, \hat{Q}_1)$, which is obtained from $Q^{sp}$ by adding, for each $i\in Q_0^{sp}$ and $\theta\in\{+,-\}$, a new vertex $i_\theta$ and a new arrow $z_{i,{\theta}}:i\rightarrow i_{\theta}$, and by adding a formal inverse $\alpha^{-1}:t(\alpha)\to s(\alpha)$ for each arrow $\alpha$ in $Q_1^{sp}\cup\{z_{i,{\theta}}|i\in Q_0^{sp},\theta\in\{+,-\} \}$. That is, we have
\begin{itemize}
	\item $\hat{Q}_0=Q^{sp}_0\cup \{i_{\theta}|i\in Q_0^{sp},\theta\in\{+,-\}\}$;
	\item $\hat{Q}_1= \{\alpha^{\pm1}|\alpha\in Q^{sp}_1\}\cup \{z_{i,{\theta}}^{\pm 1}|i\in Q^{sp}_0,\theta\in\{+,-\}\}$.
\end{itemize}
The maps $\sigma,\tau$ can be extended to $\hat{Q}$ so that
$$\sigma(z_{i,{\theta}})=\theta 1,\ \tau(z_{i,{\theta}})=1$$
for any $i\in Q_0^{sp}$ and $\theta\in\{+,-\}$, and 
$$\sigma(\alpha^{-1})=\tau(\alpha),\  \tau(\alpha^{-1})=\sigma(\alpha)$$ 
for any $\alpha\in Q_1^{sp}\cup\{z_{i,{\theta}}|i\in Q_0^{sp},\theta\in\{+,-\}\}$. 

A \emph{letter} is an arrow $\omega\in\hat{Q}_1$. A letter in $ \{z_{i,{\theta}}|i\in Q_0^{sp},\theta\in\{+,-\}\}$ is called a \emph{left end} letter; a letter in $ \{z_{i,{\theta}}^{-1}|i\in Q_0^{sp},\theta\in\{+,-\}\}$ is called a \emph{right end} letter. Both left end letters and right end letters are called \emph{end letters}. An end letter $z_{i,\rho}^{\pm1}$ is called \emph{special} provided $i\in Sp$ and $\rho=\sigma(\epsilon_i)=\tau(\epsilon_i)$.

For any letter $l\in\hat{Q}_1$, its inverse is defined to be
$$l^{-1}=\begin{cases}
	\alpha^{-1}&\text{if }l=\alpha\in Q_1^{sp}\cup\{z_{i,{\theta}}|i\in Q_0^{sp},\theta\in\{+,-\}\};\\
	\alpha&\text{if $l=\alpha^{-1}$ for $\alpha\in Q_1^{sp}\cup\{z_{i,{\theta}}|i\in Q_0^{sp},\theta\in\{+,-\}\}$}.
\end{cases}$$

\begin{definition}\label{def:PAS}
	A segment of a permissible curve in a region of $\TT$ is called a \emph{permissible arc segment} (=PAS). A PAS whose endpoints are not in $\MM\cup\PP$ is called \emph{interior} (cf. Figure~\ref{fig:local tri} for interior PASs and Figure~\ref{fig:ro} for non-interior PASs).
\end{definition}

By the construction of $Q^\TT$, each interior PAS $\eta$ gives rise to an arrow $\alpha$ in $Q^\TT$ or its formal inverse $\alpha^{-1}$ (depending on the orientation of $\eta$), which is defined to be the letter $\l(\eta)$ associated to $\eta$. The function $\sigma$ has the following interpretation.

\begin{lemma}\label{lem:wd}
	For any two interior PASs $\eta_1$ and $\eta_2$, if $\eta_1(0)$ and $\eta_2(0)$ are in the same curve $i$ in $\TT$, then $\eta_1$ and $\eta_2$ are on the same side of $i$ if and only if $\sigma(\l(\eta_1))=\sigma(\l(\eta_2))$.
\end{lemma}

\begin{proof}
	By the construction of $Q^\TT$ and $R^\TT$, $\eta_1$ and $\eta_2$ are on the same side of $i$ if and only if either $\l(\eta_1)$ and $\l(\eta_2)$ are the same, or one of them is an arrow $\alpha\in Q_1^{sp}$, the other is $\beta^{-1}$ with $\beta\in Q_1^{sp}$, and $\alpha\beta\in I^{sp}$. So by the definition of the function $\sigma$, this is equivalent to $\sigma(\l(\eta_1))=\sigma(\l(\eta_2))$.
\end{proof}

For any interior PAS $\eta$, we define $\theta(\eta)=\sigma(\l(\eta))$. By Lemma~\ref{lem:wd}, the map $\theta$ can be extended to the set of all PASs, such that for any PASs $\eta_1,\eta_2$ satisfying $\eta_1(0)$ and $\eta_2(0)$ in the same curve $i$ in $\TT$, $\eta_1$ and $\eta_2$ are on the same side of $i$ if and only if $\theta(\eta_1)=\theta(\eta_2)$. So the map $\theta$ can be regarded as labeling the sign at the start of each PAS, such that the starts of PASs starting from the same side of a curve in $\TT$ have the same sign, see Figure~\ref{fig: add}.

\begin{figure}[htpb]
    \centering
    \begin{tikzpicture}[xscale=2.5,yscale=2]
		\draw[red,thick](-.8,1.4)node[black]{$\bullet$}to(0,1)node[black]{$\bullet$}to(.8,1.4)node[black]{$\bullet$} (0,1)to(0,0) (-.8,-.4)node[black]{$\bullet$}to(0,0)node[black]{$\bullet$}to(.8,-.4)node[black]{$\bullet$};
		\draw[blue,thick,bend right=30,-<-=.5,>=stealth](-.4,1.2)to(0,.8);
		\draw[blue,thick,bend right=30,->-=.5,>=stealth](0,.8)to(.4,1.2);
		\draw[blue,thick,bend left=30,-<-=.5,>=stealth](-.4,-.2)to(0,.2);
		\draw[blue,thick,bend left=30,->-=.5,>=stealth](0,.2)to(.4,-.2);
		\draw[blue,thick,-<-=.5,>=stealth,bend left=30](.8,1.4)to(0,.5);
		\draw[blue,thick,-<-=.5,>=stealth,bend left=30](-.8,-.4)to(0,.5);
		\draw[blue](-.1,.75)node{$+$}(.1,.75)node{$-$}(-.1,.25)node{$+$}(.1,.25)node{$-$}(-.2,.5)node{$+$}(.2,.5)node{$-$}(.5,.6)node{$z_{i,-}$}(-.5,.35)node{$z_{i,+}$}(.05,.4)node[red]{$i$};
	\end{tikzpicture}\qquad
    \begin{tikzpicture}[xscale=2,yscale=2.5]
    	\draw[red,thick](0,1)node[black]{$\bullet$}to[out=-140,in=90](-.4,.1)to[out=-90,in=180](0,-.4)to[out=0,in=-90](.4,0.1)to[out=90,in=-40](0,1);
    	\draw[red](0,0)node[black]{$\bullet$};
    	\draw[red,thick,bend right=20](0,1)to(-1,-.2)node[black]{$\bullet$};
    	\draw[red,thick,bend left=20](0,1)to(1,-.2)node[black]{$\bullet$};
    	
    	\draw[blue,thick](-.3,.6)to(.3,.6);
        \draw[blue,thick,-<-=.5,>=stealth](0,0)to(0,-.4);
        \draw[blue,thick,->-=.5,>=stealth](-.4,.3)to[bend left=10](-.7,.4);
        \draw[blue,thick,->-=.5,>=stealth](.4,.3)to[bend right=10](.7,.4);
        \draw[blue,thick,-<-=.5,>=stealth](1,-.2)to[bend left=10](.37,-.1);
    	\draw[blue](-.5,.25)node{$-$}(.5,.25)node{$-$}(-.2,.5)node{$+$}(.2,.5)node{$+$}(-.1,-.3)node{$+$}(-.17,-.1)node{$z_{i,+}$}(.5,-.05)node{$-$}(.7,-.3)node{$z_{i,-}$}(0,-.5)node[red]{$i$};
	\end{tikzpicture}

    \caption{The sign $\theta$ on PAS.}
    \label{fig: add}
\end{figure}

To each non-interior PAS $\eta$, denote by $i$ the curve in $\TT$ whose interior contains an endpoint of $\eta$. Define the associated letter
$$\l(\eta)=\begin{cases}z_{i,\theta(\eta)}&\text{if $\eta(1)\in \PP\cup \MM$,}\\	z_{i,\theta(\eta^{-1})}^{-1}&\text{if $\eta(0)\in \PP\cup \MM$.}\end{cases}$$

\begin{proposition}\label{prop:bi}
	The map $\eta\mapsto\l(\eta)$ is a bijection from the set of PASs to the set $\hat{Q}_1$. 	Moreover, the following hold.
	\begin{itemize}
		\item[(1)] Any PAS $\eta$ is non-interior if and only if $\l(\eta)$ is an end letter. In this case, 
		\begin{itemize}
			\item[(i)] $\eta(0)\in\PP\cup\MM$ if and only if $\l(\eta)$ is a right end letter;
			\item[(ii)] $\eta(1)\in\PP\cup\MM$ if and only if $\l(\eta)$ is a left end letter;
			\item[(iii)] $\eta$ has an endpoint in $\PP$ if and only if $\l(\eta)$ is special.
		\end{itemize}
		\item[(2)] For any PAS $\eta$,  $\l(\eta)^{-1}=\l(\eta^{-1})$.
		\item[(3)] For any PASs $\eta_1$ and $\eta_2$, $s(\l(\eta_1))=s(\l(\eta_2))$ if and only if $\eta_1(0)$ and $\eta_2(0)$ are in the same curve $i$ of $\TT$, and in this case, $\sigma(\l(\eta_1))=\sigma(\l(\eta_2))$ if and only if $\eta_1$ and $\eta_2$ are on the same side of $i$.
	\end{itemize}
\end{proposition}

\begin{proof}
	By the construction of the map $\l(-)$,  two PASs $\eta_1$ and $\eta_2$ are the same if and only if $\l(\eta_1)=\l(\eta_2)$. So the map $\l(-)$ is injective. Any letter in $\{\alpha^{\pm1}\mid\in Q_1^{sp}\}$ corresponds to an interior PAS; for any vertex $i\in Q_0^{sp}$, the letters $z_{i,\pm 1}$ (resp. $z_{i,\pm 1}^{-1}$) respectively correspond to the two non-interior PASs which start (resp. end) at a point in the curve $i$. Hence $\l$ is a bijection and a PAS $\eta$ is non-interior if and only if $\l(\eta)$ is an end letter.
	
	For a non-interior PAS $\eta$, $\eta(0)\in\PP\cup\MM$ if and only if $\l(\eta)=z^{-1}_{i,\theta}$ is a right end letter; $\eta(1)\in\PP\cup\MM$ if and only if $\l(\eta)=z_{i,\theta}$ is a left end letter. Note that a non-interior PAS $\eta$ has an endpoint in $\PP$ if and only if it is in a once-punctured monogon. Let $\eta'$ be an interior PAS in this once-punctured monogon. Then $\l(\eta')$ is a special loop $\epsilon_i$ or its inverse $\epsilon_i^{-1}$, and hence $\sigma(\l(\eta'))=\sigma(\epsilon_i)=\sigma(\epsilon_i^{-1})$ which is independent of the orientation of $\eta'$. Thus $\l(\eta)\in\{z_{i,\sigma(\epsilon_i)}^{\pm 1}\}$, which implies that
	$\l(\eta)$ is special. Thus, we have the assertion (1). The assertions (2) and (3) follows from the construction of $\l(-)$ and Lemma~\ref{lem:wd}.
\end{proof}

\subsection{Words and prepermissible curves}\label{subsec:preper}

A \emph{word} of \emph{length} $m$ is a sequence $\m=\omega_m\cdots\omega_1$ of letters in $\hat{Q}_1$ satisfying $t(\omega_{i})=s(\omega_{i+1})$ and $\tau(\omega_{i})=-\sigma(\omega_{i+1})$ for any $1\leq i\leq m-1$. We define $s(\m)=s(\omega_1), \sigma(\m)=\sigma(\omega_1), t(\m)=t(\omega_m)$ and $\tau(\m)=\tau(\omega_m)$. Define the \emph{inverse} of $\m$ to be $\m^{-1}=\omega_1^{-1}\dots\omega_m^{-1}$. The \emph{product} of two words $\m=\omega_m\cdots\omega_1$ and $\n=\nu_n\cdots\nu_1$ is $\m\n=\omega_m\cdots\omega_1\nu_n\cdots\nu_1$ if it is again a word, that is, $t(\n)=s(\m)$ and $\tau(\n)=-\sigma(\m)$. For a word $\m$ with $t(\m)=s(\m)$ and $\tau(\m)=-\sigma(\m)$, define $\m^1=\m$ and $\m^r=\m\m^{r-1}$ for any $r\geq2$. We call $\m^r$ a proper power of $\m$ when $r\geq 2$. By Proposition~\ref{prop:bi}, we have the following.

\begin{lemma}\label{lem:word}
	For any PASs $\eta_1$ and $\eta_2$, we have that $\l(\eta_2)\l(\eta_1)$ is a word if and only if $\eta_1(1)$ and $\eta_2(0)$ are in the same curve $i$ in $\TT$ with $\eta_1$ and $\eta_2$ on the different sides of $i$.
\end{lemma}

A word $\m$ is said to be \emph{left} (resp. \emph{right}) \emph{inextensible} if there is no letter $\l\in\hat{Q}_1$ such that $\l\m$ (resp. $\m \l$) is again a word. A word is said to be \emph{inextensible} if it is both left inextensible and right inextensible. We denote by $\mathbf{IE}$ the set of inextensible words.

\begin{lemma}\label{lem:li+ri}
	Let $\m=\omega_m\cdots\omega_1$ be a word. The following are equivalent.
	\begin{itemize}
		\item[(1)] $\m$ is right (resp. left) inextensible.
		\item[(2)] $s(\m)$ (resp. $t(\m)$) is in $\hat{Q}_0\setminus Q_0^{sp}$.
		\item[(3)] $\omega_1$ (resp. $\omega_m$) is a right (resp. left) end letter.
	\end{itemize}
\end{lemma}

\begin{proof}
	The equivalence between (2) and (3) follows from the definition of left/right end letters. To show the equivalence between (1) and (2), suppose that $\m$ is left inextensible and $t(\m)\in Q_0^{sp}$. Then we have $z_{t(\m),-\tau(\m)}\m$ is a word, a contradiction. So we have that (1) implies (2). Conversely, suppose that $t(\m)\in\hat{Q}_0\setminus Q_0^{sp}$. Then $\tau(\m)=1$. But for any letter $\l$ with $s(\l)\in\hat{Q}_0\setminus Q_0^{sp}$, we have $\sigma(\l)=1$. So $\l\m$ is not a word. This implies that $\m$ is left inextensible. 
\end{proof}

\begin{definition}[Prepermissible curves]\label{def:prepermissible}
	Let $\gamma$ be a curve on $\surf$. We call $\gamma$ a \emph{prepermissible curve} (=PPC) if it satisfies the conditions \ref{itm:t1} and \ref{itm:t2} in Definition~\ref{def:permissible}. Denote by $\PPS$ the set of PPCs on $\surf$. A segment of a PPC with both endpoints in the interior of curves in $\TT$ is called a \emph{prepermissible curve segment} (PPCS).
\end{definition}

For any PPCS $\gamma$, denote by $\gamma^1,\cdots,\gamma^m$ the arc segments of $\gamma$ divided by $\TT$, which are ordered with respect to the orientation of $\gamma$. Here, $\gamma^1$ and $\gamma^m$ are called \emph{end segments} of $\gamma$. Define $\M(\gamma)=\l(\gamma^m)\cdots\l(\gamma^1)$. Conversely, for a word $\m=\omega_m\cdots\omega_1$, define $\Gamma(\m)$ to be the class of compositions of representatives in $\l^{-1}(\omega_1),\cdots,\l^{-1}(\omega_m)$ in order.

By Lemma~\ref{lem:word}, $\M(\gamma)$ is a word for any PPCS $\gamma$, and $\Gamma(\m)$ is a PPCS for any word $\m$. Hence by Proposition~\ref{prop:bi}, we have the following bijections.

\begin{lemma}\label{lem:bi}
	The maps $\M(-)$ and $\Gamma(-)$ are mutually inverse bijections between the set of PPCSs and the set of words.
\end{lemma}

Note that by Proposition~\ref{prop:bi}, the right (resp. left) end letters correspond to the non-interior PASs which start (resp. end) at $\PP\cup\MM$. So by Lemma~\ref{lem:li+ri}, we have the following.

\begin{lemma}\label{lem:li+ri2}
	Let $\gamma$ be a PPCS. Then $\M(\gamma)$ is left (resp. right) inextensible if and only if $\gamma(1)$ (resp. $\gamma(0)$) is in $\PP\cup \MM$. In particular, $\M(\gamma)$ is  inextensible if and only if $\gamma$ is a PPC. 
\end{lemma}

Combining Lemmas~\ref{lem:bi}, \ref{lem:li+ri2} and Proposition~\ref{prop:bi}, we have the following result.

\begin{proposition}\label{prop:bis}
	The maps $\M(-):\PPS\to\mathbf{IE}$ and $\Gamma(-):\mathbf{IE}\to \PPS$ are mutually inverse. Moreover, for any PPC $\eta$, we have that
	\begin{enumerate}
		\item the right end letter in $\M(\gamma)$ is special if and only if $\gamma(0)\in \PP$,
		\item the left end letter in $\M(\gamma)$ is special if and only if $\gamma(1)\in \PP$, and
		\item $\M(\gamma^{-1})=\M(\gamma)^{-1}$.
	\end{enumerate}
\end{proposition}

\subsection{Admissible words and permissible curves}\label{subsec:ad}
For any $i\in Q_0^{sp}$ and $\theta\in\{+,-\}$, let 
$$\hat{Q}(i, \theta):=\{\l\in\hat{Q}_1|s(\l)=i, \sigma(\l)=\theta\},$$
be the set of letters in $\hat{Q}_1$ that starts at $i$ with the sign $\theta$ on its tail. There is a linear order on the set $\hat{Q}(i, \theta)$ where the formal inverse of an arrow in $Q_1$ is greater than an end letter, and an end letter is greater than an arrow in $Q_1$. So $\hat{Q}(i, \theta)$ has one of the following forms, i.e.,
\begin{itemize}
	\item $\hat{Q}(i, \theta)=\{z_{i,{\theta}}\}$,
	\item $\hat{Q}(i, \theta)=\{z_{i,\theta}>\alpha\}$, where $\alpha\in Q_1$ with $s(\alpha)=i$ and $\sigma(\alpha)=\theta$,
	\item $\hat{Q}(i, \theta)=\{\beta^{-1}>z_{i,{\theta}}\}$, where $\beta\in Q_1$ with $t(\beta)=i$ and $\tau(\beta)=\theta$,
	\item $\hat{Q}(i, \theta)=\{\beta^{-1}>z_{i,\theta}>\alpha\}$, where $\alpha\in Q_1$ with $s(\alpha)=i$ and $\sigma(\alpha)=\theta$, $\beta\in Q_1$ with $t(\beta)=i$ and $\tau(\beta)=\theta$, and $\alpha\beta\in I^{sp}$.
\end{itemize}

Using directly the correspondence between letters and PASs in Proposition~\ref{prop:bi}, we have a geometric interpretation of the orders via the relative positions of PASs as follows.

\begin{lemma}\label{lem:order}
	Let $\eta_1,\eta_2$ be PASs such that they are in the same region $\Delta$ of $\TT$ and $\eta_1(0)$ and $\eta_2(0)$ are in the same edge of $\Delta$. Then $\l(\eta_1)> \l(\eta_2)$ if and only if $\eta_1$ is to right of $\eta_2$, i.e. we are in one of the situations presented in  Figure~\ref{fig:order}, where $\eta_i$, $i=1,2$, are respectively the arc segments of $\gamma_i$ in the region. 
\end{lemma}
\begin{figure}[htpb]\centering
	\begin{tikzpicture}[scale=1.25]
		\clip(-1.1,-1.5) rectangle (1.9,1.5);
		\draw[red, thick, bend left=10](0,1)to(0,-1);
		\draw[red, thick](0,-1)to(.8,-.2);
		\draw[red, thick](0,1)to(.8,.2);
		\draw[red,thick,dashed,bend left=5](.8,.2)node[black]{$\bullet$}to(.8,-.2)node[black]{$\bullet$};
		\draw[blue, thick, bend right=10,->-=.5,>=stealth](-1,.5)to(.8,.2);
		\draw[blue, thick, bend left=10,->-=.5, >=stealth](-1,-.5)to(1,-.5);
		\draw[blue,thick](-.8,.5)node[above]{$\gamma_2$};
		\draw[blue,thick](-.8,-.5)node[below]{$\gamma_1$};
		\draw[bend right,->-=.6,>=stealth](.2,-.8)to(0,-.8);
		\draw(.2,-.6)node{$\beta$};
		\draw[thick](0,1)node{$\bullet$}(0,-1)node{$\bullet$};
	\end{tikzpicture}
	\qquad
	\begin{tikzpicture}[scale=1.25]
		\clip(-1.1,-1.5) rectangle (1.9,1.5);
		\draw[red, thick, bend left=10](0,1)to(0,-1);
		\draw[red, thick](0,1)to(1,0);
		\draw[red,thick,dashed,bend right=10](0,-1)to(1,0);
		\draw[blue, thick, bend right=10,->-=.5,>=stealth](-1,0)to(1,0);
		\draw[blue, thick, bend right=10,->-=.5, >=stealth](-1,.5)to(1,.5);
		\draw[blue,thick](-.8,.5)node[above]{$\gamma_2$};
		\draw[blue,thick](-.8,0)node[below]{$\gamma_1$};
		\draw[bend left,-<-=.7,>=stealth](.2,.8)to(0,.8);
		\draw(.2,.6)node{$\alpha$};
		\draw[thick](0,1)node{$\bullet$}(0,-1)node{$\bullet$}(1,0)node{$\bullet$};
	\end{tikzpicture}
	\qquad
	\begin{tikzpicture}[scale=1.25]
		\clip(-1.1,-1.5) rectangle (1.9,1.5);
		\draw[red, thick, bend left=10](0,1)to(0,-1);
		\draw[red, thick](0,-1)to(.8,-.2);
		\draw[red, thick](0,1)to(.8,.2);
		\draw[red,thick,dashed,bend left=5](.8,.2)to(.8,-.2);
		\draw[blue, thick, bend right=10,->-=.5,>=stealth](-1,.5)to(1,.5);
		\draw[blue, thick, bend left=10,->-=.5, >=stealth](-1,-.5)to(1,-.5);
		\draw[blue,thick](-.8,.5)node[above]{$\gamma_2$};
		\draw[blue,thick](-.8,-.5)node[below]{$\gamma_1$};
		\draw[bend right,->-=.6,>=stealth](.2,-.8)to(0,-.8);
		\draw(.2,-.6)node{$\beta$};
		\draw[bend left,-<-=.7,>=stealth](.2,.8)to(0,.8);
		\draw(.2,.6)node{$\alpha$};
		\draw[thick](0,1)node{$\bullet$}(0,-1)node{$\bullet$};
	\end{tikzpicture}\\
	\begin{tikzpicture}[scale=1.25]
		\clip(-1.5,-1.5) rectangle (1.5,1.5);
		\draw[ultra thick]plot[smooth,tension=1](-.5,1)to(.5,1);
		\draw[ultra thick]plot[smooth,tension=1](-.5,-1)to(.5,-1);
		\draw[red,thick,bend right=60](0,1)to(0,-1);
		\draw[red,thick,bend left=60](0,1)to(0,-1);
		\draw[ultra thick,fill=gray!20](0,0)circle(.1);
		\draw[blue,thick,bend right=10,->-=.5,>=stealth](-1,.5)to(1,.5);
		\draw[blue,thick,bend left=10,->-=.5,>=stealth](-1,-.5)to(1,-.5);
		\draw[thick,->-=.6,>=stealth](-.2,.8)to(.2,.8);
		\draw[thick,-<-=.6,>=stealth](-.2,-.8)to(.2,-.8);
		\draw(0,.6)node{$\alpha$}(0,-.6)node{$\beta$};
		\draw[blue](-1,.5)node[left]{$\gamma_2$}(-1,-.5)node[left]{$\gamma_1$};
		\draw[thick](0,1)node{$\bullet$}(0,-1)node{$\bullet$};
	\end{tikzpicture}
	\qquad
	\begin{tikzpicture}[scale=1.25]
		\clip(-1.5,-1.5) rectangle (1.5,1.5);
		\draw[ultra thick]plot[smooth,tension=1](-.5,1)to(.5,1);
		\draw[ultra thick]plot[smooth,tension=1](-.5,-1)to(.5,-1);
		\draw[red,thick,bend right=60](0,1)to(0,-1);
		\draw[red,thick,bend left=60](0,1)to(0,-1);
		\draw[ultra thick, fill=gray!20](0,0)circle(.1);
		\draw[blue,thick,bend right=10,->-=.5,>=stealth](-1,.5)to(1,.5);
		\draw[blue,thick,smooth,->-=.5,>=stealth](-1,.3)to[out=0,in=170](0,.2)to[out=-10,in=110](.3,0)to[out=-110,in=70](0,-1);
		\draw[thick,->-=.6,>=stealth](-.2,.8)to(.2,.8);
		\draw[thick,-<-=.6,>=stealth](-.2,-.8)to(.2,-.8);
		\draw(0,.6)node{$\alpha$}(0,-.6)node{$\beta$};
		\draw[blue](-1,.6)node{$\gamma_2$}(-1,.2)node{$\gamma_1$};
		\draw[thick](0,1)node{$\bullet$}(0,-1)node{$\bullet$};
	\end{tikzpicture}
	\qquad
	\begin{tikzpicture}[scale=1.25]
		\clip(-1.5,-1.5) rectangle (1.5,1.5);
		\draw[ultra thick]plot[smooth,tension=1](-.5,1)to(.5,1);
		\draw[ultra thick]plot[smooth,tension=1](-.5,-1)to(.5,-1);
		\draw[red,thick,bend right=60](0,1)to(0,-1);
		\draw[red,thick,bend left=60](0,1)to(0,-1);
		\draw[ultra thick, fill=gray!20](0,0)circle(.1);
		\draw[blue,thick,bend left=10,->-=.5,>=stealth](-1,-.5)to(1,-.5);
		\draw[blue,thick,smooth,->-=.5,>=stealth](-1,.3)to[out=0,in=170](0,.2)to[out=-10,in=110](.3,0)to[out=-110,in=70](0,-1);
		\draw[thick,->-=.6,>=stealth](-.2,.8)to(.2,.8);
		\draw[thick,-<-=.6,>=stealth](-.2,-.8)to(.2,-.8);
		\draw(0,.6)node{$\alpha$}(0,-.6)node{$\beta$};
		\draw[blue](-1,-.6)node[below]{$\gamma_1$}(-1,-.2)node[above]{$\gamma_2$};
		\draw[thick](0,1)node{$\bullet$}(0,-1)node{$\bullet$};
	\end{tikzpicture}\\
	\begin{tikzpicture}[scale=1.25]
		\clip(-1.5,-1.5) rectangle (1.5,1.5);
		\draw[ultra thick]plot[smooth,tension=2](-.5,1)to(.5,1);
		\draw[ultra thick,fill=gray!20](0,0) circle (.1);
		\draw[thick,->-=.7,>=stealth](-.2,.8)to(.2,.8);
		\draw(0,.8)node[below]{$\epsilon$};
		\draw[red,thick](0,1)to[out=-135,in=60](-.8,0)to[out=-120,in=180](0,-1)to[out=0,in=-60](.8,0)to[out=120,in=-45](0,1);
		\draw[blue,thick,->-=.2,>=stealth](-1.2,-.5)to[bend right=5](0,-.5)to[out=10,in=-90](.45,0)to[out=90,in=-10](0,.4)to[bend right=5](-1,.4);
		\draw[blue,thick,->-=.25,>=stealth](-1.2,.3)to[bend left=5](0,.3)to[out=-10,in=90](.3,0)to[out=-90,in=0](0,-.3)to[out=180,in=-90](-.3,0)to[out=90,in=-100](0,1);
		\draw[blue](-1.2,-.5)node[below]{$\gamma_1$}(-1.2,-.1)node[above]{$\gamma_2$};
		\draw(0,1)node{$\bullet$};
	\end{tikzpicture}
	\qquad
	\begin{tikzpicture}[scale=1.25]
		\clip(-1.5,-1.5) rectangle (1.5,1.5);
		\draw[ultra thick]plot[smooth,tension=2](-.5,1)to(.5,1);
		\draw[ultra thick, fill=gray!20](0,0) circle (.1);
		\draw[thick,->-=.7,>=stealth](-.2,.8)to(.2,.8);
		\draw(0,.8)node[below]{$\epsilon$};
		\draw[red,thick](0,1)to[out=-135,in=60](-.8,0)to[out=-120,in=180](0,-1)to[out=0,in=-60](.8,0)to[out=120,in=-45](0,1);
		\draw[blue,thick,bend right=10,->-=.7,>=stealth](-1,.5)to(1,.5);
		\draw[blue,thick,->-=.25,>=stealth](-1.2,.3)to[bend left=5](0,.3)to[out=-10,in=90](.3,0)to[out=-90,in=0](0,-.3)to[out=180,in=-90](-.3,0)to[out=90,in=-100](0,1);
		\draw[blue](-1,.6)node[left]{$\gamma_2$}(-1,-.1)node[left]{$\gamma_1$};
		\draw(0,1)node{$\bullet$};
	\end{tikzpicture}
	\qquad
	\begin{tikzpicture}[scale=1.25]
		\clip(-1.5,-1.5) rectangle (1.5,1.5);
		\draw[ultra thick]plot[smooth,tension=2](-.5,1)to(.5,1);
		\draw[ultra thick, fill=gray!20](0,0) circle (.1);
		\draw[thick,->-=.7,>=stealth](-.2,.8)to(.2,.8);
		\draw(0,.8)node[below]{$\epsilon$};
		\draw[red,thick](0,1)to[out=-135,in=60](-.8,0)to[out=-120,in=180](0,-1)to[out=0,in=-60](.8,0)to[out=120,in=-45](0,1);
		\draw[blue,thick,->-=.5,>=stealth](-1.2,-.5)to[bend right=5](0,-.5)to[out=10,in=-90](.45,0)to[out=90,in=-10](0,.4)to[bend right=5](-1,.45);
		\draw[blue,thick,bend right=10,->-=.5,>=stealth](-1,.3)to(1,.3)node[right]{$\gamma_2$};
		\draw[blue](-1,-.5)node[left]{$\gamma_1$};
		\draw(0,1)node{$\bullet$};
	\end{tikzpicture}\\
	\begin{tikzpicture}[scale=1.25]
		\clip(-1.5,-1.5) rectangle (1.5,1.5);
		\draw[ultra thick]plot[smooth,tension=2](-.5,1)to(.5,1);
		\draw[thick,->-=.7,>=stealth](-.2,.8)to(.2,.8);
		\draw(0,.8)node[below]{$\epsilon$};
		\draw[red,thick](0,1)to[out=-135,in=60](-.8,0)to[out=-120,in=180](0,-1)to[out=0,in=-60](.8,0)to[out=120,in=-45](0,1);
		\draw[blue,thick,->-=.3,>=stealth](-1.2,-.5)to[bend right=5](0,-.5)to[out=10,in=-90](.45,0)to[out=90,in=-10](0,.4)to[bend right=5](-1,.45);
		\draw[blue,thick,bend right=5,->-=.5,>=stealth](-1.2,0)to(0,0);
		\draw[blue](-1.2,-.5)node[below]{$\gamma_1$}(-1.2,-.1)node[above]{$\gamma_2$};
		\draw(0,1)node{$\bullet$};
		\draw(0,0)node{$\bullet$};
	\end{tikzpicture}
	\qquad
	\begin{tikzpicture}[scale=1.25]
		\clip(-1.5,-1.5) rectangle (1.5,1.5);
		\draw[ultra thick]plot[smooth,tension=2](-.5,1)to(.5,1);
		\draw[thick,->-=.7,>=stealth](-.2,.8)to(.2,.8);
		\draw(0,.8)node[below]{$\epsilon$};
		\draw[red,thick](0,1)to[out=-135,in=60](-.8,0)to[out=-120,in=180](0,-1)to[out=0,in=-60](.8,0)to[out=120,in=-45](0,1);
		\draw[blue,thick,bend right=10,->-=.5,>=stealth](-1,.5)to(1,.5);
		\draw[blue,thick,bend left=5,->-=.5,>=stealth](-1.2,0)to(0,0);
		\draw[blue](-1,.5)node[left]{$\gamma_2$}(-1,-.3)node[left]{$\gamma_1$};
		\draw(0,1)node{$\bullet$};
		\draw(0,0)node{$\bullet$};
	\end{tikzpicture}
	\qquad
	\begin{tikzpicture}[scale=1.25]
		\clip(-1.5,-1.5) rectangle (1.5,1.5);
		\draw[ultra thick]plot[smooth,tension=2](-.5,1)to(.5,1);
		\draw[thick,->-=.5,>=stealth](-.2,.8)to(.2,.8);
		\draw(0,.8)node[below]{$\epsilon$};
		\draw[red,thick](0,1)to[out=-135,in=60](-.8,0)to[out=-120,in=180](0,-1)to[out=0,in=-60](.8,0)to[out=120,in=-45](0,1);
		\draw[blue,thick,->-=.3,>=stealth](-1.2,-.5)to[bend right=5](0,-.5)to[out=10,in=-90](.45,0)to[out=90,in=-10](0,.4)to[bend right=5](-1,.45);
		\draw[blue,thick,bend right=10,->-=.7,>=stealth](-1,.3)node[left]{$\gamma_2$}to(1,.3);
		\draw[blue](-1,-.5)node[left]{$\gamma_1$};
		\draw(0,1)node{$\bullet$};
		\draw(0,0)node{$\bullet$};
	\end{tikzpicture}
	\caption{Orders via relative positions of PASs}
	\label{fig:order}
\end{figure}

For any $i\in Q_0^{sp}$ and $\theta\in\{+,-\}$, let 
\[\mathfrak{w}(i, \theta)=\{\m\mbox{ left inextensible word}|s(\m)=i, \sigma(\m)=\theta\},\] 
be the set of left inextensible words that starts at $i$ with the sign $\theta$ on the tail of its first word (from the right). Then the order on $\hat{Q}(i,\theta)$ can be extended to the set $\m(i,\theta)$ by the dictionary order, i.e., $\m>\n$ if $\m=\mathfrak{y}u\mathfrak{x}, \n=\mathfrak{z}v\mathfrak{x}$ with $\mathfrak{x},\mathfrak{y},\mathfrak{z}$ words, $u,v$ letters and $u>v$ in $\hat{Q}(t(\mathfrak{x}),-\tau(\mathfrak{x}))$. Note that this is a linear order. By Lemma~\ref{lem:order}, we have the following.

\begin{lemma}\label{lem:order2}
	Let $\gamma_1,\gamma_2$ be two PPCSs such that $\M(\gamma_1),\M(\gamma_2)\in\mathfrak{w}(i,\theta)$ for some $i\in Q^{sp}_0$ and $\theta\in\{+,-\}$. Then 
	$\M(\gamma_1)>\M(\gamma_2)$ if and only if $\gamma_1$ and $\gamma_2$ separate as in one of the situations presented in  Figure~\ref{fig:order}, after sharing a common starting part.
\end{lemma}

For technical reasons, we also consider a trivial word $1_i$ associated to each vertex $i\in \hat{Q}_0$ and define $1_i\m=\m$ and $\n 1_i=\n$ for any word $\m$ with $t(\m)=i$ and any word $\n$ with $s(\n)=i$. 

\begin{notations}\label{not:()}
	Let $\m=\omega_m\cdots\omega_1$ be a  word. For any integers $i,j$ with $0\leq i<j\leq m+1$, define the \emph{subword} of $\m$ between $i$ and $j$ as
	$$\m_{(i,j)}=\begin{cases}
		\omega_{j-1}\cdots\omega_{i+1}&\text{if $i<j-1$,}\\
		1_{t(\omega_i)}&\text{if $i=j-1$.}
	\end{cases}$$
\end{notations}

For a word $\m=\omega_m\cdots\omega_1$ with $\m\m$ a word, a \emph{rotation} of $\m$ is a word of the form $\m_{(0,i+1)}\m_{(i,m+1)}$. For an inextensible word $\m=\omega_m\cdots\omega_1$, define its \emph{completion} to be
\[F(\m)=\begin{cases}
	\m&\text{if neither $\omega_1$ nor $\omega_m$ is special;}\\
	(\m_{(0,m)})^{-1}\epsilon_{s(\omega_m)}\m_{(0,m)}&\text{if $\omega_m$ is special and $\omega_1$ is not special;}\\
	\m_{(1,m+1)}\epsilon_{t(\omega_1)}(\m_{(1,m+1)})^{-1}&\text{if $\omega_1$ is special and $\omega_m$ is not special;}\\
	\epsilon_{s(\omega_m)}\m_{(1,m)}\epsilon_{t(\omega_1)}(\m_{(1,m)})^{-1}&\text{if both $\omega_1$ and $\omega_m$ are special.}
\end{cases}\]
Note that in the case that both $\omega_1$ and $\omega_m$ are special, the completion $F(\m)$ does not contain any end letters, and $F(\m)F(\m)$ is a word.

\begin{definition}\label{def:admissible word}
	An inextensible word $\m=\omega_m\cdots\omega_1$ is called \emph{admissible} if the following conditions hold.
	\begin{enumerate}
		\item[\namedlabel{itm:a1}{(A1)}] If $\omega_i=\epsilon$ for $\epsilon\in Q_1$ a special loop, then we have $(\m_{(0,i)})^{-1}>\m_{(i,m+1)}$; and if $\omega_i=\epsilon^{-1}$ for $\epsilon\in Q_1$ a special loop, then we have $(\m_{(0,i)})^{-1}<\m_{(i,m+1)}$.
		\item[\namedlabel{itm:a2}{(A2)}] If both $\omega_1$ and $\omega_m$ are special then its completion $F(\m)$ is not equivalent to a rotation of a proper power of any word, where two words $\omega_m\cdots\omega_1$ and $\nu_n\cdots\nu_1$ are called equivalent if $m=n$ and for any $1\leq i\leq m$, either $\omega_i=\nu_i$ or one of $\omega_i$ and $\nu_i$ is a special loop $\varepsilon$ and the other is $\varepsilon^{-1}$.
	\end{enumerate}
	We denote by $\ad$ the set of admissible words.
\end{definition}

Note that \ref{itm:a1} is well-defined since the words $\m_{(0,i)}$ and $\m_{(i,m+1)}$ are not trivial words as $\m$ is inextensible.

Recall that a PPC is a permissible curve if and only if the conditions \ref{itm:p1} and \ref{itm:p2} in Definition~\ref{def:permissible} hold. Under the bijections in Proposition~\ref{prop:bis}, the following lemmas give the equivalences between condition \ref{itm:a1} and condition \ref{itm:p1} and between condition \ref{itm:a2} and condition \ref{itm:p2}, respectively.

\begin{lemma}\label{lem:A1T1}
	An inextensible word $\m$ does not satisfy condition \ref{itm:a1} if and only if the curve $\Gamma(\m)$ cuts out a once-punctured monogon by its self-intersection. 
\end{lemma}

\begin{proof} 
	Write $\m=\omega_m\cdots\omega_1$. Then the failure of \ref{itm:a1} is that there is (after taking the inverse of $\m$ if necessary) $\omega_i=\epsilon$ such that either $\m^{-1}_{(0,i)}=\m_{(i,m+1)}$ or $\m^{-1}_{(0,i)}<\m_{(i,m+1)}$. Denote $\gamma_1=\Gamma(\m_{(i,m+1)})$, $\gamma_2=\Gamma(\m^{-1}_{(0,i)})$ and $\gamma=\Gamma(\omega_i)$. Then we are in the situation presented in Figure~\ref{fig:order and punc-monogon}, with the right part being one of the forms in  Figure~\ref{fig:order} for the case $\m^{-1}_{(0,i)}<\m_{(i,m+1)}$ by Lemma~\ref{lem:order2}, or being the form that the last arc segments of $\gamma_1$ and $\gamma_2$ coincide for the case $\m^{-1}_{(0,i)}=\m_{(i,m+1)}$. 
	\begin{figure}[htpb]\centering
		\begin{tikzpicture}[scale=0.56]
			\draw(-6,0)node{$\bullet$};
			\draw[red,thick]plot[smooth,tension=1.5] coordinates
			{(-8,0) (-6,1) (-6,-1)  (-8,0)};
			\draw[blue,thick,bend right=75,-<-=.5,>=stealth](-6,1)to(-6,-1);
			\draw[dashed] (8,0) circle (2)node{right part}
			(-3,1)node[above,blue]{$\gamma_{1}$}(-3,-1)node[below,blue]{$\gamma_{2}$};
			\draw[dashed] (5,2)to(-5,2) (5,-2)to(-5,-2);
			\draw[red, thick] (5,2)to(5,-2) (-5,-2)to(-5,2);
			\draw[red, thick] (4,2)to(4,-2) (-4,-2)to(-4,2);
			\draw[blue, thick,->-=.4,>=stealth] (-5,1).. controls +(0:3.9) and +(180:3.9) .. (5,-1);
			\draw[blue, thick,->-=.4,>=stealth] (-5,-1).. controls +(0:3.9) and +(180:3.9) .. (5,1);
			\draw(-8,0)node{$\bullet$};
			\draw[blue](-7,0)node{$\gamma$};
		\end{tikzpicture}
		\caption{A curve whose corresponding word does not satisfy condition \ref{itm:a1}}
		\label{fig:order and punc-monogon}
	\end{figure}
	In both cases, the curve $\Gamma(\m)$ cuts out a once-punctured monogon by its self-intersection. All arguments above are invertible. Hence we complete the proof.
\end{proof}

\begin{lemma}\label{lem:A2T2}
	An inextensible word $\m$ satisfies condition \ref{itm:a2} if and only if $\Gamma(\m)$ satisfies condition \ref{itm:p2}. 
\end{lemma}

\begin{proof}
	Both of the endpoints of $\Gamma(\m)$ are punctures if and only if $\m$ has two special end letters. In this case, $\Gamma(F(\m))$ (after gluing its endpoints which must be in the same arc in $\TT$) is the completion of $\Gamma(\m)$. Thus, $\Gamma(F(\m))$ is a proper power. Hence $\Gamma(\m)$ satisfies \ref{itm:p2} if and only if $\m$ satisfies \ref{itm:a2}.
\end{proof}

The following correspondence follows directly from the above two lemmas, together with Proposition~\ref{prop:bis}.

\begin{proposition}\label{prop:bi3}
	The maps $\M(-):\PS\to\ad$ and $\Gamma(-):\ad\to \PS$ are mutually inverse. Moreover, for any permissible curve $\gamma$, we have that
	\begin{enumerate}
		\item the right end letter in $\M(\gamma)$ is special if and only if $\gamma(0)\in \PP$,
		\item the left end letter in $\M(\gamma)$ is special if and only if $\gamma(1)\in \PP$, 
		\item $\M(\gamma^{-1})=\M(\gamma)^{-1}$, and
		\item $\M(\overline{\gamma})=F(\M(\gamma))$.
	\end{enumerate}
\end{proposition}

\subsection{Modules and permissible tagged curves}\label{subsec:mod and curve}

For any admissible word $\m\in\ad$, we associate an indeterminate $x$ to each special end letter in $\m$ and let $A_\m$ be the $\k$-algebra generated by these indeterminates modulo relations $x^2-x$. Let $N$ be a one-dimensional $A_\m$-module. Then the module $N$ is determined (up to isomorphism) by the values $N(x)\in\{0, 1\}$. More specifically,
\begin{itemize}
	\item if $\m$ contains no special end letters, then $A_\m=\k$, which has only one one-dimensional module $N=\k$;
	\item if $\m$ has exactly one special end letter, then $A_\m=\k[x]/(x-x^2)$, which has exactly two one-dimensional modules $\k_a, a\in\{0,1\}$, with $\k_a(x)=a$. We have
	\begin{equation}\label{eq:1}
		\dim_\k\mbox{Hom}_{A_\m}(\k_a, \k_b)=\delta_{a,b}, \ a, b\in\{0,1\},
	\end{equation}
	where $\delta_{a,b}$ is the Kronecker symbol, i.e. $\delta_{a,b}=1$ when $a=b$ and $\delta_{a,b}=0$ when $a\neq b$;
	\item if $\m$ has two special end letters, then $A_\m=\k\langle x,y\rangle/(x-x^2,y-y^2)$, which has four one-dimensional modules $\k_{a,b}$, $a, b\in\{0,1\}$, with $\k_{a,b}(x)=a$ and $\k_{a,b}(y)=b$. We have
	\begin{equation}\label{eq:2}
		\dim_\k\mbox{Hom}_{A_\m}(\k_{a,b}, \k_{a',b'})=\delta_{a,a'}\delta_{b,b'}, \ a,a',b,b'\in\{0,1\}.
	\end{equation}
\end{itemize}

\begin{construction}\label{construct:M(m,N)}
	For each pair $(\m,N)$ with $\m=\omega_m\cdots\omega_1\in\ad$ and $N$ one-dimensional $A_\m$-module, a representation $M(\m, N)=(\{M_i\}_{i\in Q^{sp}_0},\{M_a\}_{a\in Q^{sp}_1})$ of $Q^{sp}$ bounded by $I^{sg}$ (so $M(\m,N)$ is also an $A$-module) is constructed as follows.
	\begin{itemize}
		\item For each vertex $i\in Q^{sp}_0$, let $I_i=\{1\leq j\leq m-1|t(\omega_j)=i\}$;
		\item Let $M_i$ be a vector space of dimension $|I_i|$, say with base vectors $z_j, j\in I_i$;
		\item For $\alpha\in Q^{sp}_1$ an ordinary arrow, define $$M_{\alpha}(z_j)=\begin{cases}
			z_{j+1}&\text{if $\omega_{j+1}=\alpha$,}\\
			z_{j-1}&\text{if $\omega_j=\alpha^{-1}$,}\\
			0&\text{otherwise;}
		\end{cases}$$
		\item For $\epsilon\in Q_1^{sp}$ a special loop, define $$M_{\epsilon}(z_j)=\begin{cases}
			z_{j+1}&\text{if $\omega_{j+1}=\epsilon$,}\\
			z_{j-1}&\text{if $\omega_{j}=\epsilon^{-1}$,}\\
			z_j&\text{if $\omega_{j+1}=\epsilon^{-1}$ or $\omega_{j}=\epsilon$,}\\
			N(x_1)z_1&\text{if $j=1$,  $\omega_1$ is special and $t(\epsilon)=t(\omega_1)$,}\\
			N(x_m)z_{m-1}&\text{if $j=m-1$,  $\omega_m$ is special and $s(\epsilon)=s(\omega_m)$,}
		\end{cases}$$
		where $x_1$ (resp. $x_m$) is the indeterminate corresponding to $\omega_1$ (resp. $\omega_m$).
	\end{itemize}
\end{construction}

\begin{remark}\label{rmk:proj}
    A word is said to be \emph{direct} if each of its letters is in $Q^{sp}_1\cup\{z_{j,\theta}|j\in Q^{sp}_0,\theta\in\{+,-\}\}$. For any $i\in Q^{sp}_0$ and $\theta\in\{+,-\}$, the least element in the linearly ordered set $\mathfrak{w}(i,\theta)$ is the direct word whose length is maximal among those of direct words in $\mathfrak{w}(i,\theta)$. Note that the maximality implies that the left end letter of the least element is not special.

    Recall from Remark~\ref{rmk:QZ} that $\{e_j\mid j\in Q^{sp}_0\setminus Sp\}\cup\{e_i-\epsilon_i,\epsilon_i|i\in Sp\}$ is a complete set of primitive orthogonal idempotents of $A$. We have the following description of indecomposable projective modules.
    \begin{enumerate}
        \item For any $j\in Q^{sp}_0\setminus Sp$, the indecomposable projective $A$-module $e_jA=M(\m,\k)$, where $\m=\x\y^{-1}$ with $\x$ and $\y$ the least elements in $\mathfrak{w}(j,+)$ and $\mathfrak{w}(j,-)$, respectively.
        \item For any $i\in Sp$, the indecomposable projective $A$-modules $(e_i-\epsilon_i)A=M(\m,\k_0)$ and $\epsilon_iA=M(\m,\k_1)$, where $\m=\x z^{-1}_{i,\sigma(\epsilon_i)}$ with $\x$ the least element in $\mathfrak{w}(i,-\sigma(\epsilon_i))$.
    \end{enumerate}
\end{remark}

Denote by $\overline{\ad}$ the equivalence class of admissible words under taking inverse. An element in $\overline{\ad}$ is always presented by a representative in it. For any $\m=\omega_m\cdots\omega_1\in\overline{\ad}$ which contains exactly one special end letter, we always assume that $\omega_1$ is the special end letter. 

We shall use the following classification of indecomposable modules of a skew-gentle algebra.

\begin{theorem}[\cite{B,CB,De}]\label{thm:Deng}
	Let $A=\k Q^{sp}/I^{sg}$ be a skew-gentle algebra for a skew-gentle triple $(Q,Sp,I)$. Then the map $(\m,N)\mapsto M(\m,N)$ is an injective map from the set of pairs $(\m, N)$ with $\m\in\overline{\ad}$ and $N$ one-dimensional $A_\m$-module (up to  isomorphism), to the set of indecomposable $A$-modules (up to isomorphism).
\end{theorem}

\begin{remark}
	Indeed, Bondarenko \cite{B}, Crawley-Boevey \cite{CB}, and Deng \cite{De} proved the result above for general clannish algebras, where the map can be upgraded to a bijection by enlarging the set $\overline{\ad}$ and taking $N$ to be an arbitrary indecomposable $A_\m$-module. The modules $M(\m,N)$ with $\m$ containing at most one special end letter (in this case, any indecomposable $A_\m$-module is one-dimensional) are called \emph{string} modules. Any other indecomposable module is called a \emph{band} modules. Note that not all band modules are of the form $M(\m,N)$ with $\m$ containing two special end letters. 
\end{remark}

The description in \cite{G} of the Auslander-Reiten components consisting of band modules gives the following.

\begin{lemma}[\cite{G}]\label{lem:band}
	Any band module is either in a homogeneous tube or in a tube of rank 2. The band modules which are on the bottom of tubes of rank 2 are exactly the modules $M(\m,N)$ with $\m$ containing two special end letters and $N$ one-dimensional.
\end{lemma}

\begin{remark}
	More generally, the band modules $M(\m,N)$ for an $\m$ containing two special end letters and for all indecomposable $A_\m$-modules $N$ form two tubes of rank 2 and a family of tubes of rank 1. In this case, if $M(\m,N)$ is on the $r$-th level of a tube of rank 2 (resp. 1), then $\dim N=r$ ($\dim N=2r$).
\end{remark}

Denote by $\mathcal{S}$ the set of (isoclasses of) indecomposable modules $M(\m,N)$ with $\m\in\overline{\ad}$ and $N$ one-dimensional $A_\m$-module (so $\mathcal{S}$ is the image of the injective map in Theorem~\ref{thm:Deng}). 

Recall that a tagged permissible curve is a pair $(\gamma,\kappa)$, where $\gamma$ is a permissible curve and $\kappa:\{t\mid\gamma(t)\in\PP\}\to\{0,1\}$ is a map, and that $\PTS$ denotes the set of tagged permissible curves up to inverse. For any tagged permissible curve $(\gamma,\kappa)\in\PTS$, if exactly one of the endpoints of $\gamma$ is in $\PP$, we always assume $\gamma(0)\in \PP$.

\begin{remark}\label{rmk:underlying}
	For any $\gamma\in\PS$, there are the following three cases for possible tagged permissible curves $(\gamma,\kappa)\in\PTS$.
	\begin{itemize}
		\item If both $\gamma(0)$ and $\gamma(1)$ are in $\MM$, then by definition (Definition~\ref{def:permissible}), the domain of $\kappa$ is empty. So in this case, there is only one tagged permissible curve $(\gamma,\emptyset)$.
		\item If $\gamma(0)\in \PP$ and $\gamma(1)\in \MM$, then there are two tagged permissible curves $(\gamma,\kappa_a), a\in\{0,1\}$ with $\kappa_a(0)=a$.
		\item If both $\gamma(0)$ and $\gamma(1)$ are in $\PP$, then there are four tagged permissible curves $(\gamma,\kappa_{a,b}), a,b\in\{0,1\}$ with $\kappa_{a,b}(0)=a$ and $\kappa_{a,b}(1)=b$.
	\end{itemize}
\end{remark}

\begin{construction}\label{cons:mod}
	Let  $(\gamma,\kappa)$  be a tagged permissible curve. To any  $t\in\{0,1\}$ with $\gamma(t)\in \PP$, by Proposition~\ref{prop:bi3}, there is an associated special end letter of $\M(\gamma)$. Let $x_t$ denote the corresponding indeterminate. Define $N_{\kappa}$ to be the one-dimensional $A_{\M(\gamma)}$-module satisfying $N_{\kappa}(x_t)=\kappa(t)$.  We denote $M(\gamma,\kappa)=M(\M(\gamma),N_{\kappa})$.
\end{construction}

We have the following main result in this section.

\begin{theorem}\label{thm:curve and mod}
	The map $(\gamma,\kappa)\mapsto M(\gamma,\kappa)$ is a bijection from the set  $\PTS$ of tagged permissible curves to the set $\mathcal{S}$.
\end{theorem}

\begin{proof}
    For a fixed permissible curve $\gamma$, the map $\kappa\mapsto N_{\kappa}$ is a bijection from the set of tagged curves whose underlying permissible curve is $\gamma$ (cf. Remark~\ref{rmk:underlying}) to the set of one-dimensional modules of $A_{\M(\gamma)}$. Hence by Proposition~\ref{prop:bi3} and Theorem~\ref{thm:Deng}, we have this theorem.
\end{proof}

\section{Auslander-Reiten translation via tagged rotation}

In this section, we introduce the notion of tagged rotation of tagged generalized permissible curves, and show that this gives a geometric realization of the Auslander-Reiten translation.

\subsection{The tagged rotation}\label{subsec:ro}

In this subsection, we introduce the notion of tagged rotation, which is a generalization of the pivotal elementary move in \cite{BS} and the tagged rotation in \cite{BQ,QZ}.

\begin{definition}[Tagged rotation]\label{def:ro}
	The tagged rotation $\rho(\gamma,\kappa)$ of a tagged permissible curve $(\gamma,\kappa)$ is $(\rho(\gamma),\kappa')$ where
	\begin{itemize}
		\item $\rho(\gamma)$ is the rotation of $\gamma$ obtained from $\gamma$ by moving every endpoint of $\gamma$ that is in $\MM$ along the boundary clockwise to the next marked point, and
		\item $\kappa'(t)=1-\kappa(t)$ for any $t$ with $\gamma(t)\in\PP$.
	\end{itemize}
\end{definition}

For a curve $\gamma$ on $\surf$ with $\gamma(1)\in \MM$, denote by $[1]\gamma$ the curve obtained from $\gamma$ by moving $\gamma(1)$ along the boundary clockwise to the next marked point, see Figure~\ref{fig:rot}; dually, for a curve $\gamma$ on $\surf$ with $\gamma(0)\in \MM$, denote by $\gamma[1]$ the curve obtained from $\gamma$ by moving $\gamma(0)$ along the boundary clockwise to the next marked point. Then we have
$$\rho(\gamma)=\begin{cases}
	[1]\gamma[1]&\text{if $\gamma(0),\gamma(1)\in \MM$,}\\
	[1]\gamma&\text{if $\gamma(0)\in \PP$ and $\gamma(1)\in \MM$,}\\
	\gamma&\text{if $\gamma(0),\gamma(1)\in \PP$.}
\end{cases}$$

\begin{figure}[htpb]
	\begin{tikzpicture}
		\draw[ultra thick] (0,0)circle(1);
		\draw[blue,thick,->-=.5,>=stealth] (0,3)to(0,1);
		\draw[blue,thick,->-=.5,>=stealth] (0.2,3)to[out=-90,in=90](1,0);
		\draw[blue,thick,->-=.5,>=stealth] (-0.2,3)to[out=-90,in=90](-1,0);
		\draw[blue] (0.1,1.5)node{$\gamma$};
		\draw[blue] (1.1,1.5)node{$[1]\gamma$};
		\draw[blue] (-1.2,1.5)node{$[-1]\gamma$};
		\draw(0,1)node{$\bullet$} (1,0)node{$\bullet$} (-1,0)node{$\bullet$};
	\end{tikzpicture}
	\caption{Rotations of curves}\label{fig:rot}
\end{figure}

Note that for a permissible curve $\gamma$, it is possible that none of $\gamma[1],[1]\gamma$ and $\rho(\gamma)$ crosses any curve in $\TT$ in the interior of $\surf$. So we need the following notion.

\begin{definition}[Trivial curves]
	A curve $\gamma$ on $\surf$ which does not transversally cross any curve in $\TT$ in the interior of $\surf$ is called \emph{trivial} (with respect to $\TT$). For any trivial curve $\gamma$, we define $M(\gamma,\kappa)$ to be a zero $A$-module for any map $\kappa:\{t|\gamma(t)\in \PP\}\to\{0,1\}$.
\end{definition}

Note also that, it is possible that one of $\gamma[1],[1]\gamma$ and $\rho(\gamma)$ is not permissible even if it is non-trivial. So we introduce the following notion.

\begin{definition}[Almost permissible curves]
    A curve $\gamma$ on $\surf$ is called an \emph{almost permissible curve} (resp. \emph{almost PPC}) provided that its starting/ending segment has one of the forms in Figure~\ref{fig:ex} and after moving these segments to equivalent segments in condition \ref{itm:t1}, it becomes a permissible curve (resp. PPC).
\end{definition}

Although a PPC is a permissible curve if and only if it satisfies \ref{itm:p1} and \ref{itm:p2}, an almost PPC satisfying \ref{itm:p1} and \ref{itm:p2} is possibly not an almost permissible curve because after changing starting/ending segments, \ref{itm:p1} may fail, see Figure~\ref{fig:curve and punc monogon}.

\begin{figure}[htpb]\centering
	\begin{tikzpicture}[xscale=1,yscale=1]
		\draw[red,thick,dashed](0,-1)to(-1,-1);
		\draw[red,thick,bend left=50,dashed](-1,-1)to(-2,1);
		\draw[red,thick,dashed](-2,1)to(-1,1.5);
		\draw[red,thick](0,-1)to(-1,1.5);
		\draw[blue,thick,bend left=10](-1,-1)to(-.4,0);
		\draw[thick] (1,0.2)node{$\simeq$};
		\draw(0,-1)node{$\bullet$}(-1,-1)node{$\bullet$}(-2,1)node{$\bullet$}(-1,1.5)node{$\bullet$};
	\end{tikzpicture}
	\qquad
	\begin{tikzpicture}[xscale=1,yscale=1]
		\draw[red,thick,dashed](0,-1)to(-1,-1);
		\draw[red,thick,bend left=50,dashed](-1,-1)to(-2,1);
		\draw[red,thick,dashed](-2,1)to(-1,1.5);
		\draw[red,thick](0,-1)to(-1,1.5);	\draw[blue,thick,bend right=10](-2,1)to(-.6,.5);
		\draw(0,-1)node{$\bullet$}(-1,-1)node{$\bullet$}(-2,1)node{$\bullet$}(-1,1.5)node{$\bullet$};
	\end{tikzpicture}\\
	\begin{tikzpicture}[scale=1.25]
		\draw[red,thick,bend right=60](0,1)to(0,-1);
		\draw[red,thick,bend left=60](0,1)to(0,-1);
		\draw[ultra thick, fill=gray!20](0,0)circle(.1);
		\draw[blue,thick,smooth](0,1)to[out=-60,in=90] (.3,0)to[out=-90,in=10] (-.45,-.4);
		\draw[thick] (1.5,0)node{$\simeq$};
		\draw[thick](0,1)node{$\bullet$}(0,-1)node{$\bullet$};
	\end{tikzpicture}
	\qquad
	\begin{tikzpicture}[scale=1.25]
		\draw[red,thick,bend right=60](0,1)to(0,-1);
		\draw[red,thick,bend left=60](0,1)to(0,-1);
		\draw[ultra thick, fill=gray!20](0,0)circle(.1);
		\draw[blue,thick,smooth](0,-1)to[out=60,in=-90] (.3,0)to[out=90,in=-10] (-.45,.4);
		\draw[thick] (1.5,0)node{$\simeq$};
		\draw[thick](0,1)node{$\bullet$}(0,-1)node{$\bullet$};
	\end{tikzpicture}\qquad
	\begin{tikzpicture}[scale=1.25]
		\draw[red,thick,bend right=60](0,1)to(0,-1);
		\draw[red,thick,bend left=60](0,1)to(0,-1);
		\draw[ultra thick, fill=gray!20](0,0)circle(.1);
		\draw[blue,thick,smooth](0,1)to[out=-115,in=90](-.3,0)to[out=-90,in=180](0,-.3)to[out=0,in=-90](.3,0)to[out=90,in=-10] (-.45,.4);
		\draw[thick] (1.5,0)node{$\simeq$};
		\draw[thick](0,1)node{$\bullet$}(0,-1)node{$\bullet$};
	\end{tikzpicture}\qquad
	\begin{tikzpicture}[scale=1.25]
		\draw[red,thick,bend right=60](0,1)to(0,-1);
		\draw[red,thick,bend left=60](0,1)to(0,-1);
		\draw[ultra thick, fill=gray!20](0,0)circle(.1);
		\draw[blue,thick,smooth](0,-1)to[out=115,in=-90](-.3,0)to[out=90,in=180](0,.3)to[out=0,in=90](.3,0)to[out=-90,in=10] (-.45,-.4);
		\draw[thick](0,1)node{$\bullet$}(0,-1)node{$\bullet$};
	\end{tikzpicture}
	\\
	\begin{tikzpicture}[scale=1.25]
		\draw[ultra thick, fill=gray!20](0,0) circle (.1);
		\draw[red,thick](0,1)to[out=-135,in=60](-.8,0)to[out=-120,in=180](0,-1)to[out=0,in=-60](.8,0)to[out=120,in=-45](0,1);
		\draw[blue,thick,smooth](0,1)to[out=-75,in=90](.3,0)to[out=-90,in=0](0,-.3)to[out=180,in=-90](-.3,0)to[out=90,in=-170] (.54,.4);
		\draw[thick] (1.7,0)node{$\simeq$};
		\draw(0,1)node{$\bullet$};
	\end{tikzpicture}
	\begin{tikzpicture}[scale=1.25]
		\draw[ultra thick, fill=gray!20](0,0) circle (.1);
		\draw[red,thick](0,1)to[out=-135,in=60](-.8,0)to[out=-120,in=180](0,-1)to[out=0,in=-60](.8,0)to[out=120,in=-45](0,1);
		\draw[blue,thick,smooth](0,1)to[out=-115,in=90](-.3,0)to[out=-90,in=180](0,-.3)to[out=0,in=-90](.3,0)to[out=90,in=-10] (-.54,.4);
		\draw(0,1)node{$\bullet$};
	\end{tikzpicture}\\
	\begin{tikzpicture}[scale=1.25]
		\clip(-1.2,-1.2) rectangle (1.2,1.2);
		\draw[red,thick](0,1)to[out=-135,in=80](-.6,0)to[out=-100,in=180](0,-1)to[out=0,in=-80](.6,0)to[out=100,in=-45](0,1);
		\draw[blue,thick](0,0)to(0,-1.2);
		\draw(0,1)node{$\bullet$};
		\draw(0,0)node{$\bullet$};
	\end{tikzpicture}
	\caption{Equivalences of starting/ending segments}\label{fig:ex}
\end{figure}

\begin{notations}\label{not:per}
	Each almost permissible curve (resp. almost PPC) $\gamma$ gives rise to a unique permissible curve (resp. PPC), which is denoted by $[\gamma]$. Since starting and ending segments contribute nothing to the corresponding word, we can define $\M(\gamma)=\M([\gamma])$. So for any tagged almost permissible curve $(\gamma,\kappa)$, we can define $M(\gamma,\kappa)=M([\gamma],\kappa)$.
\end{notations}

By construction (cf. also \cite{BS}), for any PPC $\gamma$, any of $\gamma[1],[1]\gamma$ and $\rho(\gamma)$ is an almost PPC if it is not trivial. So we have the following.

\begin{lemma}\label{lem:ro}
	For any permissible curve $\gamma$, its rotation $\rho(\gamma)$ is almost permissible if it is not trivial.
\end{lemma}

However, this may not be true for $\gamma[1]$ and $[1]\gamma$. This is because the permissible version $[\gamma[1]]$ (resp. $[[1]\gamma]$) does possibly cut out a once-punctured monogon by a self-intersection at $\MM$, see Figure~\ref{fig:curve and punc monogon}. In such case, $[\gamma[1]]$ (resp. $[[1]\gamma]$) is the completion $\overline{\eta}$ of a permissible curve $\eta$ with $\eta(0)\in \PP$ and $\eta(1)\in \MM$. 

\begin{figure}[htpb]\centering
	\begin{tikzpicture}[xscale=1.8,yscale=1.5]
		\draw[ultra thick](-.5,-1)node{$\bullet$}to(1,-1)node{$\bullet$};
		\draw[ultra thick](-1,-.5)to(-1,1);
		\draw(0,-1)node[black]{$\bullet$}(.5,-1)node[black]{$\bullet$}(-1,0)node[black]{$\bullet$}(-1,.5)node[black]{$\bullet$};
		\draw[orange,thick,-<-=.5,>=stealth](.5,-1)to(.5,0)node[black]{$\bullet$};
		
		\draw[green,thick,->-=.4,>=stealth](.5,-1)to[out=60,in=-90](.8,0)to[out=90,in=0](.5,.5)to[out=180,in=90](.2,0)to[out=-90,in=120](.5,-1);
		\draw[green,thick,->-=.5,>=stealth](.5,-1)to[out=45,in=-90](1,0.2)to[out=90,in=0](.6,.9)to[out=180,in=90](0,-1);
		
		\draw[blue,thick,->-=.5,>=stealth](.5,-1)to[out=30,in=-90](1.2,.5)to[out=90,in=0](.5,1.3)to[out=180,in=90](-.5,-1);
		
		\draw[red,thick](1,-1)to[out=60,in=-90](1.5,.5)to[out=90,in=0](.5,1.7)to[out=180,in=60](-1,.5);
		
		\draw[red,thick](-1,0)to(-.5,-1) (-1,0)to(0,-1) (-1,.5)to[bend right=5](1,-1) (0,-1)to(-1,.5);
		\draw[red,thick](-1,.5)to[out=-5,in=75](.68,-.06)to[out=-115,in=-30](-1,.5);
		
		\draw(.6,-.5)node[orange]{$\eta$}(.5,1.4)node[blue]{$\gamma$}(.5,1.05)node[green]{$[1]\gamma$}(.55,.65)node[green]{$[[1]\gamma]$};
	\end{tikzpicture}
	\caption{The completion of a permissible curve as the permissible version of the rotation of a permissible curve at one end}
	\label{fig:curve and punc monogon}
\end{figure}

\begin{notations}\label{not:sum}
	For any completion $\overline{\eta}$ of a permissible curve $\eta$ with $\eta(0)\in \PP$ and $\eta(1)\in \MM$, we denote by $M(\overline{\eta},\emptyset)$ the decomposable module $M(\eta,\kappa_0)\oplus M(\eta,\kappa_1)$, where $\kappa_a(0)=a, a\in\{0,1\}$.
\end{notations}

\begin{definition}[Generalized tagged permissible curves]\label{def:gen-perm}
	A generalized permissible curve is either a permissible curve or a curve whose completion is in $\TT$. Denote by $\GPS$ the set of generalized permissible curves and $\TT^\ast=\GPS\setminus\PS$.
	
	A tagged generalized permissible curve is a pair $(\gamma,\kappa)$ with $\gamma$ a generalized curve and $\kappa:\{t\mid \gamma(t)\in\PP \}\to \{0,1\}$ a map. Denote by $\GPTS$ the set of tagged generalized permissible curves up to inverse and denote $\TT^\times=\GPTS\setminus\PTS$.
\end{definition}

The set $\TT^\ast$ can be obtained from $\TT$ by replacing the edge $\gamma$ of each once-punctured monogon by the only curve $\eta$ in the monogon whose completion is $\gamma$, see Figure~\ref{fig:rep}. 

\begin{figure}[htpb]\centering
	\begin{tikzpicture}[xscale=1,yscale=1.5]
		\draw[ultra thick]plot[smooth,tension=2](-.5,1)to(.5,1);
		\draw[red,thick](0,1)to[out=-155,in=70](-1,.2)to[out=-110,in=180](0,-.5)to[out=0,in=-70](1,.2)node[right]{$\gamma$}to[out=110,in=-25](0,1);
		\draw[blue,thick] (0,1)to(0,0) (0,.3)node[right]{$\eta$};
		\draw(0,1)node{$\bullet$};
		\draw[thick](0,0)node{$\bullet$};
	\end{tikzpicture}
	\caption{Replacing the edge $\gamma$ of a once-punctured monogon with $\eta$}
	\label{fig:rep}
\end{figure}

\begin{remark}\label{rmk:times}
	The set $\TT^\times$ is the tagged version of the admissible partial triangulation $\TT$ (cf. \cite{FST,ILF,QZ}). Recall from Remark~\ref{rmk:QZ} that the set $\{e_j\mid j\in Q^\TT_0\setminus Sp\}\cup\{e_i-\epsilon_i,\epsilon_i|i\in Sp\}$ is a complete set of primitive orthogonal idempotents. There is a bijection from this set to $\TT^\times$ which 
	\begin{itemize}
	\item sends $e_j$ to the tagged generalized permissible curve $(\gamma,\emptyset)$ such that $j$ is the vertex indexed by $\gamma$, and
	\item sends $e_i-\epsilon_i$ (resp. $\epsilon_i$) to the tagged generalized permissible curve $(\gamma,\kappa_a)$ (see Remark~\ref{rmk:underlying} for the notation $\kappa_a$) such that $i$ is the vertex indexed by the completion $\overline{\gamma}$ of $\gamma$ and $a=1$ (resp. $a=0$). 
	\end{itemize}
\end{remark}

\begin{lemma}\label{lem:suc2}
	Let $\gamma$ be a permissible curve. Then $\rho(\gamma)$ is trivial if and only if $\rho(\gamma)\in\TT^\ast$.
\end{lemma}

\begin{proof}
    The ``if" part is trivial. We show the ``only if'' part. If $\gamma(0),\gamma(1)\in \PP$, then $\rho(\gamma)=\gamma$ is not trivial, a contradiction. So at least one of $\gamma(0)$ and $\gamma(1)$ is in $\MM$. If $\gamma$ has exactly one endpoint in $\MM$, then so does $\rho(\gamma)$. Since $\rho(\gamma)$ is trivial, its completion is the edge of a once-punctured monogon of $\TT$, which implies $\rho(\gamma)\in\TT^\ast$. Now we consider the case $\gamma(0),\gamma(1)\in \MM$. Let $\Delta$ be the region of $\TT$ where $\rho(\gamma)$ lives. There are the following cases.
	\begin{itemize}
		\item $\Delta$ is a once-punctured monogon. Because $\gamma$ is not trivial, it must cut out a once-punctured monogon by its endpoints, which contradicts \ref{itm:p1} in Definition~\ref{def:permissible}. So this case does not occur.
		\item $\Delta$ is a monogon or digon with an unmarked boundary component in its interior. If $\rho(\gamma)$ has self-intersections in the interior of the surface, then one of the arc segments of $\gamma$ which are not end segments has self-intersections. This contradicts the condition \ref{itm:t2} in Definition~\ref{def:permissible}. So $\rho(\gamma)$ has no self-intersections in the interior of the surface. Hence if $\Delta$ is a monogon, $\rho(\gamma)$ has to be the edge of $\Delta$, which is in $\TT^\ast$. If $\Delta$ is a digon  and if $\rho(\gamma)$ has the following form
		
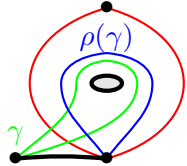
\begin{center}
			\begin{tikzpicture}[xscale=2,yscale=1]
				\draw[red,thick,bend right=60](0,1)to(0,-1);
				\draw[red,thick,bend left=60](0,1)to(0,-1);
				\draw[ultra thick, fill=gray!20](0,0)circle(.1);
				\draw[green,thick,smooth](-.6,-1)to[out=60,in=-90](-.2,0)to[out=90,in=190](0,.3)to[out=-10,in=90](.2,0)to[out=-90,in=30](-.6,-1);
				\draw[blue,thick,smooth](0,-1)to[out=110,in=-90](-.3,0)to[out=90,in=190](0,.4)to[out=-10,in=90](.3,0)to[out=-90,in=70](0,-1);
				\draw[green,thick] (-.6,-.7)node{$\gamma$};
				\draw[blue](0,.6)node{$\rho(\gamma)$};
				\draw[ultra thick,bend left=10](-.6,-1)node{$\bullet$}to(0,-1);
				\draw[thick](0,1)node{$\bullet$}(0,-1)node{$\bullet$};
			\end{tikzpicture}	
		\end{center}
		then the curve $\gamma$ does not cut out an angle of $\Delta$, which contradicts the condition \ref{itm:t2} in Definition~\ref{def:permissible}. So $\rho(\gamma)$ has to be one edge of $\Delta$, which is in $\TT^\ast$.
		\item $\Delta$ is a polygon without a puncture nor an unmarked boundary component in its interior. Suppose that $\rho(\gamma)$ is not in $\TT^\ast$, then we are in one of the following situations
		
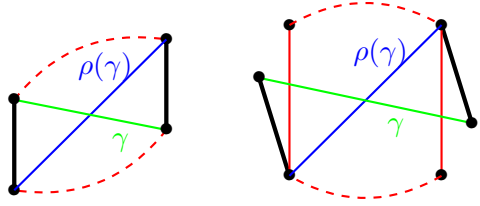
\begin{center}
			\begin{tikzpicture}[xscale=1,yscale=1]
				\draw[ultra thick] (-1,.2)to(-1,-1) (1,1)to(1,-.2);
				\draw[red, thick, dashed, bend left=30] (-1,.2)to (1,1) (1,-.2)to(-1,-1);
				\draw (1,1)node{$\bullet$} (-1,-1)node{$\bullet$}; 
				\draw[blue,thick] (-1,-1)to(1,1) (.2,.3)node[above]{$\rho(\gamma)$};
				\draw[green,thick] (-1,.2)node[black]{$\bullet$}to (1,-.2)node[black]{$\bullet$} (.4,-.1)node[below]{$\gamma$};
			\end{tikzpicture}\qquad
			\begin{tikzpicture}[xscale=1,yscale=1]
				\draw[ultra thick] (-1,-1) to (-1.4,.3) (1,1) to (1.4,-.3);
				\draw[red, thick] (-1,1)node[black]{$\bullet$}to(-1,-1) (1,1)to(1,-1)node[black]{$\bullet$};
				\draw[red, thick, dashed, bend left=30] (-1,1)to (1,1) (1,-1)to(-1,-1);
				\draw (1,1)node{$\bullet$} (-1,-1)node{$\bullet$}; 
				\draw[blue,thick] (-1,-1)to(1,1) (.2,.3)node[above]{$\rho(\gamma)$};
				\draw[green,thick] (-1.4,.3)node[black]{$\bullet$} to (1.4,-.3)node[black]{$\bullet$} (.4,-.1)node[below]{$\gamma$};
			\end{tikzpicture}\qquad
		\end{center}
		where in the first situation, $\gamma$ is trivial while in the last situation, $\gamma$ is not permissible since it does not cut out an angle of $\Delta$ as required in \ref{itm:t2} in Definition~\ref{def:permissible}. So there is always a contradiction. Hence $\rho(\gamma)$ is in $\TT^\ast$.
	\end{itemize}
\end{proof}

\subsection{The Auslander-Reiten translation}\label{subsec:ar}

In this subsection, we show that the Auslander-Reiten translation can be realized as the tagged rotation.

Recall from Section~\ref{subsec:ad} that we have a linear order $>$ on the set 
\[\mathfrak{w}(i, \theta)=\{\m\mbox{ left inextensible word}|s(\m)=i, \sigma(\m)=\theta\},\ i\in Q_0^{sp},\theta\in\{+,-\}.\]
Denote by $[1]\m$ (resp. $[-1]\m$) the successor (resp. predecessor) of $\m$ with respect to this linear order, if it exists. So we have $[-1]\m>\m>[1]\m$. For a right inextensible word $\m$, noting that $\m^{-1}$ is left inextensible, define $\m[1]=([1]\m^{-1})^{-1}$ and $\m[-1]=([-1]\m^{-1})^{-1}$. Let $\m$ be an inextensible word. By \cite{G}, $[1](\m[1])=([1]\m)[1]$ (resp. $[-1](\m[-1])=([-1]\m)[-1]$) if both sides of the equality exist; denote by $[1]\m[1]$ (resp. $[-1]\m[-1]$) one (or both) of them. 

Let $\m$ be an admissible word. By \cite{QZ}, if the word $[1]\m$ exists but is not admissible, then $[1]\m=F(\n)$ or $F(\n)^{-1}$, where $\n$ is an admissible word which contains exactly one special end letter. In this case, we denote $M([1]\m,\k)=M(\n,\k_0)\oplus M(\n,\k_{1})$. 

We shall use the following description of irreducible morphisms and the Auslander-Reiten translation $\tau$ for skew-gentle algebras. Recall that for any admissible word $\m=\omega_m\cdots\omega_1\in\overline{\ad}$ with exactly one special end letter, we always assume that $\omega_1$ is the special end letter.

\begin{theorem}[\cite{G}]\label{thm:ref}
	For any $\m=\omega_m\cdots\omega_1\in\overline{\ad}$ and any one-dimensional $A_\m$-module $N$, the $A$-module $M(\m,N)$ is projective (resp. injective) if and only if one of the following holds.
	\begin{enumerate}
		\item Both $\omega_1$ and $\omega_m$ are not special, and $[1]\m[1]$ (resp. $[-1]\m[-1]$) does not exist.
		\item Only $\omega_1$ is special and $[1]\m$ (resp. $[-1]\m$) does not exist.
	\end{enumerate}
	When the $A$-module $M(\m,N)$ is not projective, we have	
	\[\tau M(\m,N)=\begin{cases}
		M([1]\m[1],\k)&\text{if both $\omega_1$ and $\omega_n$ are not special  and $N=\k$;}\\
		M([1]\m,\k_{1-a})&\text{if $\omega_1$ is special, $\omega_m$ is not special and $N=\k_a$;}\\
		M(\m,\k_{1-a,1-b})&\text{if both $\omega_1$ and $\omega_n$ are special and $N=\k_{a,b}$.}\\
	\end{cases}\] 
	Moreover, the middle term of the Auslander-Reiten sequence ending at $M(\m,N)$ is
	$$\begin{cases}
		M(\m[1],\k)\oplus M([1]\m,\k)&\text{if both $\omega_1$ and $\omega_n$ are not special,}\\
		M([1]F(\m),\k)&\text{if $\omega_1$ is special and $\omega_m$ is not special,}
	\end{cases}$$
	where, $M(\m[1],\k)$ (resp. $M([1]\m,\k)$) is taken to be a zero module if $\m[1]$ (resp. $[1]\m$) does not exist.
\end{theorem}

The rotation (at one end) of PPCs gives a realization of the successors of inextensible words.

\begin{lemma}\label{lem:suc}
	Let $\gamma$ be a PPC.
	\begin{enumerate}
		\item[(1)] If $\gamma(1)\in \MM$ and $[1]\gamma$ is not trivial, then $\M([1]\gamma)=[1]\M(\gamma)$. 
		\item[(2)] If $\gamma(0)\in \MM$ and $\gamma[1]$ is not trivial, then $\M(\gamma[1])=\M(\gamma)[1]$.
	\end{enumerate} 
\end{lemma}

\begin{proof}
	We only prove the first assertion since the second one can be proved dually. By construction, $\gamma$ and $[1]\gamma$ start at the same point, go through the same way at the beginning, and then separate in a region as one of the forms in Figure~\ref{fig:forms} (except the last two, where either $\gamma$ or $[1]\gamma$ is trivial), where $\delta$ is the boundary segment between $[1]\gamma(1)$ and $\gamma(1)$. 
	\begin{figure}[htpb]\centering
		\begin{tikzpicture}[xscale=1,yscale=1]
			\clip(-1.5,-1.5) rectangle (1.4,1.5);
			\draw[red, thick, bend left=10](0,1)to(0,-1);
			\draw[red, thick](0,-1)to(1,0);
			\draw[red,thick,dashed,bend left=10](0,1)to(1,0);
			\draw[green, thick, bend left=10,->-=.5,>=stealth](-1,0)to(1,0);
			\draw[blue, thick, bend left=10,->-=.5, >=stealth](-1,-.5)to(1,-.5);
			\draw[ultra thick,bend right=10] (1,0) to (1,-.5)node{$\bullet$};
			\draw[thick] (1,-.25)node[right]{$\delta$};
			\draw[green,thick](-1,0)node[above]{$[1]\gamma$};
			\draw[blue,thick](-1,-.5)node[below]{$\gamma$};
			\draw[red] (.6,-.6)to (1,0)to(.8,-.6);
			\draw[thick](0,1)node{$\bullet$}(0,-1)node{$\bullet$}(1,0)node{$\bullet$};
		\end{tikzpicture}
		\quad\quad
		\begin{tikzpicture}[xscale=1,yscale=1]
			\clip(-1.5,-1.5) rectangle (1.4,1.5);
			\draw[red, thick, bend left=10](0,1)to(0,-1);
			\draw[red, thick](0,1)to(1,0);
			\draw[red,thick,dashed,bend right=10](0,-1)to(1,0);
			\draw[blue, thick, bend right=10,->-=.5,>=stealth](-1,0)to(1,0);
			\draw[green, thick, bend right=10,->-=.5, >=stealth](-1,.5)to(1,.5);
			\draw[blue,thick](-1,0)node[below]{$\gamma$};
			\draw[green] (-1,.5)node[above]{$[1]\gamma$};
			\draw[ultra thick,bend left=10] (1,0) to (1,.5);
			\draw[thick] (1,.25)node[right]{$\delta$};
			\draw[red] (.6,.6)to (1,0)to(.8,.6);
			\draw[thick](0,1)node{$\bullet$}(0,-1)node{$\bullet$}(1,0)node{$\bullet$} (1,.5)node{$\bullet$};
		\end{tikzpicture}
		\qquad
		\begin{tikzpicture}[xscale=1,yscale=1]
			\clip(-1.5,-1.5) rectangle (1.4,1.5);
			\draw[ultra thick]plot[smooth,tension=1](-.5,1)to(.5,1);
			\draw[red,thick,bend right=60](0,1)to(0,-1);
			\draw[red,thick,bend left=60](0,1)to(0,-1);
			\draw[ultra thick, fill=gray!20](0,0)circle(.1);
			\draw[green,thick,->-=.5,>=stealth]plot[smooth,tension=1] coordinates {(-1,.5) (.2,.4) (.6,-1)};
			\draw[blue,thick,smooth,->-=.5,>=stealth](-1,.4)to[out=-5,in=170](0,.3)to[out=-10,in=90](.3,0)to[out=-90,in=70](0,-1);
			\draw[green,thick] (-1,.7)node{$[1]\gamma$};
			\draw[blue](-1,0)node{$\gamma$};
			\draw[thick] (.3,-1)node[below]{$\delta$};
			\draw[red] (.8,-.4)to (0,-1)to(.8,-.7);
			\draw[ultra thick,bend left=10](0,-1)to(.6,-1)node{$\bullet$};
			\draw[thick](0,1)node{$\bullet$}(0,-1)node{$\bullet$};
		\end{tikzpicture}
		\qquad
		\begin{tikzpicture}[xscale=1,yscale=1]
			\clip(-1.5,-1.5) rectangle (1.4,1.5);
			\draw (.3,1)node[above]{$\delta$};
			\draw[red,thick,bend right=60](0,1)to(0,-1);
			\draw[red,thick,bend left=60](0,1)to(0,-1);
			\draw[ultra thick, fill=gray!20](0,0)circle(.1);
			\draw[green,thick,smooth,->-=.5,>=stealth](-1,-.3)to[out=0,in=-170](0,-.2)to[out=10,in=-110](.3,0)to[out=80,in=-80](0,1);
			\draw[blue] (-1,-.6)node[below]{$\gamma$};
			\draw[green] (-1,-.2)node[above]{$[1]\gamma$};
			\draw[blue,thick,->-=.5,>=stealth]plot[smooth,tension=1] coordinates {(-1,-.5) (.2,-.4) (.6,1)};
			\draw[red] (.8,.4)to (0,1)to(.8,.7);
			\draw[ultra thick,bend right=10](0,1)to(.6,1)node{$\bullet$};
			\draw[ultra thick]plot[smooth,tension=1](-.5,-1)to(.5,-1);
			\draw[thick](0,1)node{$\bullet$}(0,-1)node{$\bullet$};
		\end{tikzpicture}\qquad
		\begin{tikzpicture}[xscale=1,yscale=1]
			\clip(-1.5,-1.5) rectangle (1.4,1.5);
			\draw (.3,1)node[above]{$\delta$};
			\draw[ultra thick, fill=gray!20](0,-.2) circle (.1);
			\draw[red,thick](0,1)to[out=-135,in=80](-.45,.1)to[out=-100,in=180](0,-.8)to[out=0,in=-80](.45,.1)to[out=100,in=-45](0,1);
			\draw[blue,thick,->-=.5,>=stealth](-1.2,-1)to[bend right=5](0,-1)to[out=10,in=-90](.6,0)to[out=90,in=0](0,.5)to[out=180,in=90](-.6,0)to[out=-90,in=180](.2,-1.1)to[out=0,in=-90](.8,-.2)to[out=90,in=-80](.6,1);
			\draw[green,thick,->-=.5,>=stealth](-1.2,-.5)to[bend right=5](0,-.5)to[out=10,in=-90](.3,-.2)to[out=90,in=0](0,.1)to[out=180,in=90](-.3,-.2)to[out=-90,in=180](0,-.4)to[out=0,in=-90](.2,-.2)to[out=90,in=-80](0,1);
			\draw[blue] (-1.2,-.5)node[below]{$\gamma$};
			\draw[green](-1.15,-.55)node[above]{$[1]\gamma$};
			\draw[red] (.9,.4)to (0,1)to(.9,.7);
			\draw[ultra thick,bend right=10](0,1)to(.6,1)node{$\bullet$};
			\draw(0,1)node{$\bullet$};
		\end{tikzpicture}
		\qquad
		\begin{tikzpicture}[xscale=1,yscale=1]
			\clip(-1.5,-1.5) rectangle (1.4,1.5);
			\draw (-.3,1)node[above]{$\delta$};
			\draw[ultra thick, fill=gray!20](0,-.2) circle (.1);
			\draw[red,thick](0,1)to[out=-135,in=80](-.45,.1)to[out=-100,in=180](0,-.8)to[out=0,in=-80](.45,.1)to[out=100,in=-45](0,1);
			\draw[green,thick,->-=.5,>=stealth](-1.2,.4)to[bend left=5](0,.4)to[out=-10,in=90](.6,0)to[out=-90,in=0](0,-1)to[out=180,in=-90](-.7,0)to[out=90,in=-100](-.6,1);
			\draw[blue,thick,->-=.5,>=stealth](-1.2,.1)to[bend left=5](0,.1)to[out=-10,in=90](.3,-.2)to[out=-90,in=0](0,-.5)to[out=180,in=-90](-.3,-.2)to[out=90,in=-100](0,1);
			\draw[green](-1.15,.6)node{$[1]\gamma$};
			\draw[blue] (-1.2,-.1)node{$\gamma$};
			\draw[red] (-.8,.6)to (0,1)to(-.8,.8);
			\draw[ultra thick,bend right=10] (-.6,1)node{$\bullet$}to(0,1);
			\draw(0,1)node{$\bullet$};
		\end{tikzpicture}\qquad
		\begin{tikzpicture}[xscale=1,yscale=1]
			\clip(-1.5,-1.5) rectangle (1.4,1.5);
			\draw (.3,1)node[above]{$\delta$};
			\draw[red,thick](0,1)to[out=-135,in=80](-.45,.1)to[out=-100,in=180](0,-1)to[out=0,in=-80](.45,.1)to[out=100,in=-45](0,1);
			\draw[blue,thick,->-=.7,>=stealth](0,0) to (.6,1);
			\draw[green,thick,->-=.5,>=stealth](0,0)to (0,1);
			\draw[blue] (.3,.5)node[right]{$\gamma$};
			\draw[green](-.4,.5)node{$[1]\gamma$};
			\draw[ultra thick,bend right=10](0,1)to(.6,1)node{$\bullet$};
			\draw(0,1)node{$\bullet$};
			\draw[thick](0,0)node{$\bullet$};
		\end{tikzpicture}
		\qquad
		\begin{tikzpicture}[xscale=1,yscale=1]
			\clip(-1.5,-1.5) rectangle (1.4,1.5);
			\draw (-.3,1)node[above]{$\delta$};
			\draw[red,thick](0,1)to[out=-135,in=80](-.45,.1)to[out=-100,in=180](0,-1)to[out=0,in=-80](.45,.1)to[out=100,in=-45](0,1);
			\draw[green,thick,->-=.7,>=stealth](0,0)to(-.6,1);
			\draw[blue,thick,->-=.5,>=stealth](0,0)to(0,1);
			\draw[green](-.3,.5)node[left]{$[1]\gamma$};
			\draw[blue] (-.1,.5)node[right]{$\gamma$};
			\draw[ultra thick,bend right=10] (-.6,1)node{$\bullet$}to(0,1);
			\draw(0,1)node{$\bullet$};
			\draw(0,0)node{$\bullet$};
		\end{tikzpicture}
		\caption{Rotation at one end}
		\label{fig:forms}
	\end{figure}
	Then by Lemma~\ref{lem:order2}, we have  $\M(\gamma)>\M([1]\gamma)$. Moreover, since $\gamma$, $[1]\gamma$ and $\delta$ enclose a contractible triangle, by Lemma~\ref{lem:order2} again, there is no PPC $\gamma'$ with $\gamma'(0)=\gamma(0)=[1]\gamma(0)$ and such that $\M(\gamma)>\M(\gamma')>\M([1]\gamma)$. Therefore, the bijection in Proposition~\ref{prop:bis} between PPCs and inextensible words implies that $\M([1]\gamma)$ is the successor of $\M(\gamma)$, i.e. $\M([1]\gamma)=[1]\M(\gamma)$.
\end{proof}
Now we are ready to interpret the Auslander-Reiten translation via the tagged rotation.

\begin{theorem}\label{thm:tau}
	Let $(\gamma,\kappa)$ be a tagged permissible curve. Then $M(\gamma,\kappa)$ is the projective (resp. injective) module at an idempotent $e$ if and only if the tagged generalized permissible curve $\rho(\gamma,\kappa)$ (resp. $\rho^{-1}(\gamma,\kappa)$) is in $\TT^\times$ corresponding to $e$ (see Remark~\ref{rmk:times} for the correspondence). When $M(\gamma,\kappa)$ is not projective, we have $$\tau M(\gamma,\kappa)=M(\rho(\gamma,\kappa)).$$
	Moreover, if $\gamma(0),\gamma(1)\in\MM$, the Auslander-Reiten sequence ending at $M(\gamma,\kappa)$ is
	$$0\to M([1]\gamma[1],\emptyset)\to M(\gamma[1],\emptyset)\oplus M([1]\gamma,\emptyset)\to M(\gamma,\emptyset)\to 0;$$
	if $\gamma(0)\in \PP$ and $\gamma(1)\in\MM$,
	the Auslander-Reiten sequence ending at $M(\gamma,\kappa)$ is
	$$0\to M([1]\gamma,\kappa')\to M([1]\overline{\gamma},\emptyset)\to M(\gamma,\kappa)\to 0$$
	where $\kappa'(0)=1-\kappa(0)$ and $\overline{\gamma}$ is the completion of $\gamma$ (see Definition~\ref{def:completion}).
\end{theorem}

\begin{proof}
    Let $(\gamma',\kappa')$ be a tagged generalized permissible curve in $\TT$ and $e$ the corresponding idempotent. By Construction~\ref{cons:mod}, $M(\gamma,\kappa)=M(\M(\gamma),N_\kappa)$. By Remark~\ref{rmk:proj}, $M(\M(\gamma),N_\kappa)$ is projective at $e$ if and only if the following hold.
    \begin{enumerate}
        \item When $e=e_j$ for $j\in Q_0^{sp}\setminus Sp$, we have $\M(\gamma)=\x\y^{-1}$ with $\x$ and $\y$ the direct word in $\mathfrak{w}(j,+)$ and $\mathfrak{w}(j,-)$ respectively, whose length is maximal among those of direct words there, and $N_{\kappa}=\k$. Then we have $\gamma=\rho^{-1}(\gamma')$ and $\kappa=\kappa'=\emptyset$, see the left picture of Figure~\ref{fig:proj}. So $\rho(\gamma,\kappa)=(\gamma',\kappa')$.
        \item When $e=e_i-\epsilon_i$ (resp. $\epsilon_i$), we have $\M(\gamma)=\x z^{-1}_{i,\sigma(\epsilon_i)}$ with $\x$ the direct word in $\mathfrak{w}(i,-\sigma(\epsilon_i))$, whose length is maximal among those of direct words there, and $N_{\kappa}=\k_a$ with $a=0$ (resp. $a=1$). Note that $\kappa'=\kappa_a$ with $a=1$ (resp. $a=0$). Then $\gamma=\rho^{-1}(\gamma')$ and $\kappa\neq\kappa'$, see the right picture of Figure~\ref{fig:proj}. So $\rho(\gamma,\kappa)=(\gamma',\kappa')$.
    \end{enumerate}
    In all, in each case, we have $\rho(\gamma,\kappa)$ is in $\TT^\times$ corresponding to $e$. The assertion for injective modules can be proved dually.
    
    \begin{figure}[htpb]\centering
        \begin{tikzpicture}[xscale=2.8,yscale=1.3]
            \draw[ultra thick,bend right=10](0,1)node{$\bullet$}to(1,1)node{$\bullet$};
            \draw[ultra thick, bend left=10](0,-1)node{$\bullet$}to(1,-1)node{$\bullet$};
            \draw[red,thick,bend right=5](0,1)to(1,-1);
            \draw[red](.2,.2)node{$\gamma'$};
            \draw[red](.8,-.4)node{$j$};
            \draw[red,thick](1,-1)to(0,-.1) (1,-1)to(0,-.5) (0,1)to(1,.5) (0,1)to(1,.1);
            \draw[blue,thick](0,-1)to(1,1);
            \draw(0.65,0.05)node[blue]{$\gamma$};
		\end{tikzpicture}\qquad
		\begin{tikzpicture}[xscale=3,yscale=1.5]
            \draw(0,-.2)node{$\bullet$};
            \draw[red,thick](0,1)to[out=-115,in=80](-.3,.1)to[out=-100,in=180](0,-.8)to[out=0,in=-80](.3,.1)to[out=100,in=-65](0,1);
            \draw[red](0.4,-.3)node{$i$};
            \draw[blue,thick,bend right=10](0,-.2)to(.6,1);
            \draw[blue,thick](0,1)tonode[left]{$\gamma'$}(0,-.2);
            \draw[blue](.45,.3)node{$\gamma$}; 
            \draw[red] (.9,.4)to (0,1)to(.9,.7);
            \draw[ultra thick,bend right=10](0,1)to(.6,1)node{$\bullet$};
            \draw(0,1)node{$\bullet$};
		\end{tikzpicture}
		\caption{$M(\gamma,\kappa)$ is a projective module}
		\label{fig:proj}
	\end{figure}
	
	When $M(\gamma,\kappa)$ is not projective, denote $\rho(\gamma,\kappa)=(\rho(\gamma),\kappa')$. Then $\rho(\gamma)$ is not in $\TT$ and hence is not trivial by Lemma~\ref{lem:suc2}. There are the following three cases.
	\begin{itemize}
		\item If $\gamma(0),\gamma(1)\in\MM$, we have $\rho(\gamma)=[1]\gamma[1]$. We deduce that at least one of $\gamma[1]$ and $[1]\gamma$ is not trivial. This is because otherwise both $\gamma[1]$ and $[1]\gamma$ do not transversally  cross any curve in $\TT$ in the interior of $\surf$. Since $\gamma$ is permissible, there is a curve in $\TT$ crossing $\gamma$ in its interior. But this curve does not cross $\gamma[1]$ and $[1]\gamma$, so it has to be $\rho(\gamma)$, cf. Figure~\ref{fig:non-trivial}, a contradiction.
		
		\begin{figure}[htpb]\centering
			\begin{tikzpicture}[xscale=1,yscale=1]
				\draw[ultra thick] (-1,.2)to(-1,-1) (1,1)to(1,-.2);
				\draw[red, thick, bend left=30] (-1,.2)to (1,1) (1,-.2)to(-1,-1);
				\draw (1,1)node{$\bullet$} (-1,-1)node{$\bullet$} (-1,.8)node[red]{$\gamma[1]$} (1,-.8)node[red]{$[1]\gamma$}; 
				\draw[blue,thick] (-1,-1)to(1,1) (.2,.3)node[above]{$\rho(\gamma)$};
				\draw[green,thick] (-1,.2)node[black]{$\bullet$}to (1,-.2)node[black]{$\bullet$} (.4,-.1)node[below]{$\gamma$};
			\end{tikzpicture}
			\caption{Both $\gamma[1]$ and $[1]\gamma$ are trivial}
			\label{fig:non-trivial}
		\end{figure}
		Then by Lemma~\ref{lem:suc}, we have $\M(\gamma[1])=\M(\gamma)[1]$ (if $\gamma[1]$ is not trivial), $\M([1]\gamma)=[1]\M(\gamma)$ (if $[1]\gamma$ is not trivial) and  $\M(\rho(\gamma))=[1]\M(\gamma)[1]$. By Theorem~\ref{thm:ref}, we get the required Auslander-Reiten sequence.
		\item If $\gamma(0)\in \PP$ and $\gamma(1)\in \MM$, we have $\rho(\gamma)=[1]\gamma$ which is not trivial. So $[1]\overline{\gamma}$ is not trivial neither. By Lemma~\ref{lem:suc}, we have $\M([1]\gamma)=[1]\M(\gamma)$ and $\M([1]\overline{\gamma})=[1]\M(\overline\gamma)=[1]F(\M(\gamma))$, where the last equality is due to Proposition~\ref{prop:bi3}~(4). By Definition~\ref{def:ro} and Construction~\ref{cons:mod}, we have $N_{\kappa'}(x_0)=\kappa'(0)=1-\kappa(0)=1-N_{\kappa}(x_0)$. Hence by Theorem~\ref{thm:ref}, we get the required Auslander-Reiten sequence.
		\item If $\gamma(0),\gamma(1)\in\PP$, we have $\rho(\gamma)=\gamma$. By Definition~\ref{def:ro} and Construction~\ref{cons:mod}, we have  $N_{\kappa'}(x_t)=\kappa'(t)=1-\kappa(t)=1-N_{\kappa}(x_t)$ for any $t\in\{0,1\}$. Hence by Theorem~\ref{thm:ref}, we have $\tau M(\gamma,\kappa)=M(\rho(\gamma,\kappa))$.
	\end{itemize}
\end{proof}

\begin{remark}\label{rmk:AR}
	The module $M(\gamma[1],\emptyset)$ (resp. $M([1]\gamma,\emptyset)$) occurring in the middle term of the AR sequence in Theorem~\ref{thm:tau} is decomposable if and only if $\gamma[1]$ (resp. $[1]\gamma$) is non-trivial and its permissible version $[\gamma[1]]$ (resp. $[[1]\gamma]$) is the completion of a permissible curve $\eta$. In this case, we have that $M(\gamma[1],\emptyset)$ (resp. $M([1]\gamma,\emptyset)$) is $M(\eta,\kappa_0)\oplus M(\eta,\kappa_1)$, cf. Notations~\ref{not:per} and \ref{not:sum}. Hence the middle term of the AR sequence may contain more than two indecomposable summands.
\end{remark}

\section{Intersection-dimension formula}\label{sec:int-dim}

For any two tagged generalized permissible curves $(\gamma_1,\kappa_1),(\gamma_2,\kappa_2)$, whenever we consider their intersections, we always assume that $\gamma_1$ and $\gamma_2$ are in a minimal position, i.e., they are representatives in their homotopy classes such that their intersections are minimal. In what follows, denote by $\gamma_1\sim\gamma_2$ if $\gamma_1$ is homotopic to $\gamma_2$.

\begin{definition}[Intersection number]\label{def:tag int}
	Let $(\gamma_1,\kappa_1),(\gamma_2,\kappa_2)$ be 
	two tagged generalized permissible curves. The \emph{intersection number} between $(\gamma_1,\kappa_1)$ and $(\gamma_2,\kappa_2)$ is defined to be
	$$\Int((\gamma_1,\kappa_1),(\gamma_2,\kappa_2))=|\gamma_1\cap^\circ\gamma_2|+|\mathfrak{T}((\gamma_1,\kappa_1),(\gamma_2,\kappa_2))|,$$
	where
	$$\gamma_1\cap^\circ\gamma_2:=\{(t_1,t_2)\mid 0< t_1,t_2< 1,\ \gamma_1(t_1)=\gamma_2(t_2) \},$$
	is the set of \emph{interior} intersections between $\gamma_1$ and $\gamma_2$, and $\mathfrak{T}((\gamma_1,\kappa_1),(\gamma_2,\kappa_2))$ is the set of \emph{tagged intersections} between $(\gamma_1,\kappa_1)$ and $(\gamma_2,\kappa_2)$, i.e., $(t_1,t_2)\in\{0,1\}\times\{0,1\}$ satisfying $\gamma_1(t_1)=\gamma_2(t_2)\in\PP$ and the following conditions.
	\begin{enumerate}
		\item[(TI1)] $\kappa_1(t_1)\neq\kappa_2(t_2)$.
		\item[(TI2)] If $\gamma_1|_{t_1\rightarrow(1-t_1)}\sim\gamma_2|_{t_2\rightarrow(1-t_2)}$, where $\gamma_{0\to 1}=\gamma$ and $\gamma_{1\to 0}=\gamma^{-1}$, then $\gamma_1(1-t_1)=\gamma_2(1-t_2)\in \PP$ and $\kappa_1(1-t_1)\neq\kappa_2(1-t_2)$.
	\end{enumerate}
\end{definition}

\begin{example}
	Let $(\gamma_1,\kappa_1)$ and $(\gamma_2,\kappa_2)$ be two tagged generalized permissible curves such that $\gamma_1\nsim\gamma_2$, $\gamma_1\sim\gamma_2^{-1}$, $\gamma_1(0)=\gamma_2(0)=\gamma_1(1)=\gamma_2(1)\in\PP$ and $\kappa_1(0)=1,\kappa_1(1)=\kappa_2(0)=\kappa_2(1)$=0, see Figure~\ref{fig:exmint}. Then $(0,0)$ and $(0,1)$ satisfy the condition (TI1). Since $\gamma_1\nsim\gamma_2$, we have that $(0,0)$ also satisfies the condition (TI2). Since $\gamma_1\sim\gamma_2^{-1}$ and $\kappa_1(1)=\kappa_2(0)$, we have that $(0,1)$ does not satisfy the condition (TI2).
	
	\begin{figure}[htpb]\centering
		\begin{tikzpicture}[rotate=45]
			\draw[blue,->-=.45,>=stealth, thick](0,0).. controls +(45:2) and +(90:1) .. (2.4,0)
			.. controls +(-90:1) and +(-30:2) .. (0,0);
			\draw[blue](0.17,0.17)node{$\times$};
			\draw[](3,.7)node{$\gamma_2$}  (2.5,.5)node{$\gamma_1$};
			\draw[blue,-<-=.47,>=stealth, thick](0,0).. controls +(75:3) and +(90:1) .. (3,0)
			.. controls +(-90:1) and +(-40:3) .. (0,0);
			\draw[ultra thick,fill=gray!20] (1.4,.15) circle (.4);
			\draw(0,0)node{$\bullet$};
		\end{tikzpicture}
		\caption{Examples of intersections at a puncture}
		\label{fig:exmint}
	\end{figure}
\end{example}

The main result in this section is the following.

\begin{theorem}\label{thm:int-dim}
	Let $(\gamma_1,\kappa_1),(\gamma_2,\kappa_2)$ be two tagged permissible curves. Then we have
	\begin{equation}\label{eq:int-dim}
		\Int((\gamma_1,\kappa_1),(\gamma_2,\kappa_2))=\dim_\k\Hom(M_1,\tau M_2)+\dim_\k\Hom(M_2,\tau M_1)
	\end{equation}
	where $M_i=M(\gamma_i,\kappa_i)$, $i=1,2$.
\end{theorem}

\begin{remark}
	When $\surf$ has no unmarked boundary components and $\TT$ is a maximal admissible partial triangulation, the algebra $A$ is the endomorphism algebra of a cluster tilting object in the cluster category of $\surf$. In this case, the formula \eqref{eq:int-dim} is a direct consequence of \cite[Theorem~1.1~(iii)]{QZ} and \cite[Proposition~4.3~(c)]{AIR}.
\end{remark}

When $\surf$ has no punctures, i.e. $\PP=\emptyset$, $A$ is a gentle algebra. In this case, the permissible curves are the curves satisfying \ref{itm:t1} and \ref{itm:t2} in Definition~\ref{def:permissible}, since \ref{itm:p1} and \ref{itm:p2} hold automatically. For any permissible curve $\gamma$, we define $M(\gamma)=M(\gamma,\emptyset)$. Then we have the following int-dim formula for the model in \cite{BS}. 

\begin{corollary}\label{cor:gentle}
	When $A$ is a gentle algebra, for any two permissible curves $\gamma_1,\gamma_2$, we have
	\begin{equation}
		\Int(\gamma_1,\gamma_2)=\dim_\k\Hom(M(\gamma_1),\tau M(\gamma_2))+\dim_\k\Hom(M(\gamma_2),\tau M(\gamma_1))
	\end{equation}
	where $\Int(\gamma_1,\gamma_2)=|\{(t_1,t_2)\mid \gamma_1(t_1)=\gamma_2(t_2)\notin \MM \}|$.
\end{corollary}

\subsection{Homomorphisms}\label{subsec:hom}

Let $\m=\omega_m\cdots\omega_1$ be a word. Recall from Notation~\ref{not:()} that for any integers $i,j$ with $0\leq i<j\leq m+1$, the subword of $\m$ between $i$ and $j$ is
$$\m_{(i,j)}=\begin{cases}
	\omega_{j-1}\cdots\omega_{i+1}&\text{if $i<j-1$,}\\
	1_{t(\omega_i)}&\text{if $i=j-1$.}
\end{cases}$$

\begin{definition}\label{def:common pair}
	Let $\m=\omega_m\cdots\omega_1$ and $\n=\nu_n\cdots\nu_1$ be two admissible words. A \emph{common pair} from $\m$ to $\n$ is a pair $((i,j),(h,l))$ of pairs of integers with $0\leq i< j\leq m+1$ and $0\leq h< l\leq n+1$ such that one of the following conditions holds:
	\begin{itemize}
		\item[(1)] $\m_{(i,j)}=\n_{(h,l)}$, $\omega_i^{-1}<\nu_h^{-1}$ and $\omega_{j}<\nu_{l}$;
		\item[(2)] $\m_{(i,j)}=(\n_{(h,l)})^{-1}$, $\omega^{-1}_i<\nu_{l}$ and $\omega_{j}<\nu_h^{-1}$,
	\end{itemize}
	where if an inequality contains at least one of $\omega_0$, $\omega_{m+1}$, $\nu_0$ and $\nu_{n+1}$ then it is assumed to hold automatically. Let $H^{\m,\n}$ be the set of common pairs from $\m$ to $\n$.
\end{definition}

\begin{remark}\label{rmk:factor-sub}
    Let $\m=\omega_m\cdots\omega_1$ be an admissible word. A subword $\m_{(i,j)}, 0\leq i<j\leq m+1$ of $\m$ is called 
    \begin{itemize}
        \item a \emph{factor string} if $\omega_j$ (resp. $\omega_i^{-1}$) is in $Q^{sp}_1\cup\{z_{k,\theta}|k\in Q^{sp}_0,\theta\in\{+,-\}\}$ when $j\leq m$ (resp. $1\leq i$), and
        \item a \emph{substring} if $\omega^{-1}_j$ (resp. $\omega_i$) is in $Q^{sp}_1\cup\{z_{k,\theta}|k\in Q^{sp}_0,\theta\in\{+,-\}\}$ when $j\leq m$ (resp. $1\leq i$).
    \end{itemize}
    Then a common pair $((i,j),(h,l))$ from $\m$ to $\n$ is nothing but a pair of a factor string $\m_{(i,j)}$ of $\m$ and a substring $\n_{(h,l)}$ of $\n$ such that $\m_{(i,j)}=\n_{(h,l)}$ or $\m_{(i,j)}=\n_{(h,l)}^{-1}$. See Figure~\ref{fig:common pair}.
	\begin{figure}
		$$\xymatrix@R=1em@C=5em{
			\m:\quad\cdot&\cdot\ar@{~}[l]_{\omega_m\cdots\omega_{j+1}}&&&\cdot&\cdot\ar@{~}[l]_{\omega_{i-1}\cdots\omega_1}\\
			&&\cdot\ar@{=}[d]\ar[ul]_{\omega_j}&\cdot\ar@{=}[d]\ar[ur]^{\omega_i}\ar@{~>}[l]_{\omega_{j-1}\cdots\omega_{i+1}}&&\\
			&&\cdot&\cdot\ar@{~}[l]^{\nu_{l-1}\cdots\nu_{h+1}}&&\\
			\n:\quad\cdot&\cdot\ar[ur]_{\nu_l}\ar@{~}[l]^{\nu_n\cdots\nu_{l+1}}&&&\cdot\ar[ul]^{\nu_h}&\cdot\ar@{~>}[l]^{\nu_{h-1}\cdots\nu_1}
		}$$
		\caption{Common pairs from $\m$ to $\n$}\label{fig:common pair}
	\end{figure}
\end{remark}

For each common pair $J=((i,j),(h,l))\in H^{\m,\n}$, denote by $A_J$ the $\k$-algebra generated by the indeterminates associated to the special end letters in $\m_{(i,j)}$ (or equivalently in $\n_{(h,l)}$). Then $A_J$ is a subalgebra of $A_\m$ and $A_\n$, which implies that any $A_\m$-module or $A_\n$-module can be regarded as an $A_J$-module. The following result is useful in this section.

\begin{theorem}[\cite{G}]\label{thm:dim formular}
	Let $\m, \n$ be two admissible words and $V, W$ be one-dimensional $A_\m$-module and $A_\n$-module, respectively. Then we have
	\[\dim_{\k}\emph{Hom}_{A}(M(\m, V), M(\n, W))=\sum_{J\in H^{\m,\n}}\dim_{\k}\emph{Hom}_{A_J}(V, W)\]
\end{theorem}

\subsection{Proof of Theorem~\ref{thm:int-dim}}\label{subsec:4.2}

By taking suitable representatives in their homotopy classes, we may assume that any intersection in $\gamma_1\cap^\circ\gamma_2$ is not in a curve in $\TT$.

Let
$$\mathfrak{P}(\gamma_1,\gamma_2)=\{(t_1,t_2)|\gamma_1(t_1)=\gamma_2(t_2)\in \PP\}\subseteq \{0,1\}\times\{0,1\}$$
be the set of \emph{punctured intersections} between $\gamma_1$ and $\gamma_2$.
Note that $\mathfrak{T}((\gamma_1,\kappa_1),(\gamma_2,\kappa_2))$ is a subset of $\mathfrak{P}(\gamma_1,\gamma_2)$.

We denote by $\gamma_1^1,\cdots,\gamma_1^m$ (resp. $\gamma_2^1,\cdots,\gamma_2^n$) the arc segments of $\gamma_1$ (resp. $\gamma_2$) divided by $\TT$ in the order with respect to the orientation of $\gamma_1$ (resp. $\gamma_2$). Let $c$ be an intersection of $\gamma_1$ and $\gamma_2$. By our setting, it is an intersection of $\gamma_1^{c_1}$ and $\gamma_2^{c_2}$ in their interior, for some $1\leq c_1\leq m$ and $1\leq c_2\leq n$. Note that for each pair of arc segments $\gamma_1^u$ and $\gamma_2^v$ of $\gamma_1$ and $\gamma_2$, they have at most one intersection. 

\begin{lemma}\label{lem:types}
	Let $c$ be an intersection in $\gamma_1\cap^\circ\gamma_2$. Then it has one of the forms in Figure~\ref{fig:I}, \ref{fig:II} or \ref{fig:III}, where $\{a,b\}=\{1,2\}$.
\end{lemma}

\begin{figure}
	\begin{tikzpicture}[scale=1.5]
        \begin{scope}[xscale=.8]
		\draw (0,0)node{$\bullet$};
		\draw[red,thick] (1,-1)to (0,0)to(-1,-1);
		\draw[blue,thick] (-1,-.2)node[left]{$\gamma_b$}to(1,-.8); \draw[blue,thick] (1,-.2)to(-1,-.8)node[left]{$\gamma_a$}  (0,-.5)node[below]{$c$} (0,-.5)node{$\bullet$};
        \end{scope}
	\begin{scope}[shift={(2.5,-.5)},rotate=-90,xscale=.7,yscale=1.3]
		\draw (-.7,0)node{$\bullet$} (.7,0)node{$\bullet$};
		\draw[red,thick] (.7,-1)to (.7,0) to (-.7,0)to(-.7,-1);
		\draw[blue,thick] (-.7,.2)to(1,-.8)node[left]{$\gamma_a$}; \draw[blue,thick](.7,.2)to(-1,-.8)node[left]{$\gamma_b$} (0,-.23)node[below]{$c$} (0,-.23)node{$\bullet$};
	\end{scope}
        \begin{scope}[shift={(4.5,0)},xscale=.8]
    	\draw (0,0)node{$\bullet$};
    	\draw[red,thick] (1,-1)to (0,0)to(-1,-1);
    	\draw[blue,thick] (-1,-.2)node[left]{$\gamma_b$}to(1,-1)node[black]{$\bullet$}; 
    	\draw[blue,thick] (1,-.2)to(-1,-.8)node[left]{$\gamma_a$}  (-.14,-.56)node[below]{$c$} (-.14,-.56)node{$\bullet$};
    \end{scope}
    \begin{scope}[shift={(7,-.5)},rotate=-90,xscale=.7,yscale=1.3]
    	\draw (-.7,0)node{$\bullet$} (.7,0)node{$\bullet$};
    	\draw[red,thick] (.7,-1)to (.7,0) to (-.7,0)to(-.7,-1);
    	\draw[blue,thick] (-.7,.2)to(.7,-1)node[black]{$\bullet$} (.3,-.7)node[left]{$\gamma_a$}; \draw[blue,thick](.7,.2)to(-1,-.8)node[left]{$\gamma_b$} (-.12,-.3)node[below]{$c$} (-.12,-.3)node{$\bullet$};
    \end{scope}
    \end{tikzpicture}
	\caption{Intersections of type I}\label{fig:I}
\end{figure}

\begin{figure}
	\begin{tikzpicture}[xscale=1.5,yscale=1.5]
		\clip(-1.3,-.7) rectangle (1,1.1);
		\draw[red, thick, bend left=10](0,1)to(-.5,-.1);
		\draw[red, thick, bend left=10](-.5,-.1) to (.5,-.5);
		\draw[red, thick](0,1)to(.9,.4);
		\draw[red,thick,dashed,bend right=10](.5,-.5)to(.9,.4);
		\draw[blue,thick,bend right=10](-.8,.8)node[left]{$\gamma_a$}to(.6,.8);
		\draw[blue, thick, bend left=30] (0,1) to (0,-.4)node[below]{$\gamma_b$};
		\draw[blue] (.25,.65)node{$c$} (.12,.73)node{$\bullet$};
		\draw[thick](0,1)node{$\bullet$}(-.5,-.1)node{$\bullet$}(.9,.4)node{$\bullet$}(.5,-.5)node{$\bullet$};
	\end{tikzpicture}\qquad
	\begin{tikzpicture}[xscale=1.5,yscale=1.5]
		\clip(-1.3,-.7) rectangle (1,1.1);
		\draw[ultra thick, fill=gray!20](0,.2) circle (.1);
		\draw[red,thick](0,1)to[out=-145,in=90](-.6,.2)to[out=-90,in=145](0,-.6);
		\draw[red,thick](0,1)to[out=-35,in=90](.6,.2)to[out=-90,in=35](0,-.6);
		\draw[blue,thick,smooth](-1,.6)node[below]{$\gamma_b$}to[out=-5,in=170](0,.5)to[out=-10,in=90](.3,.2)to[out=-90,in=70](0,-.6);
		\draw[blue,thick,bend left=10](-.8,-.4)node[left]{$\gamma_a$}to(.8,-.4);
		\draw[blue] (0.05,-.2)node{$c$} (.12,-.33)node{$\bullet$};
		\draw[thick](0,1)node{$\bullet$}(0,-.6)node{$\bullet$};
	\end{tikzpicture}\qquad
	\begin{tikzpicture}[xscale=1.5,yscale=1.5]
		\clip(-1.3,-.7) rectangle (1,1.1);
		\draw[ultra thick, fill=gray!20](0,0) circle (.1);
		\draw[red,thick](0,1)to[out=-135,in=60](-.7,0.3)to[out=-120,in=180](0,-.6)to[out=0,in=-60](.7,0.3)to[out=120,in=-45](0,1);
		\draw[blue,thick](-1,.3)node[below]{$\gamma_b$}to[bend left=5](0,.3)to[out=-10,in=90](.3,0)to[out=-90,in=0](0,-.3)to[out=180,in=-90](-.3,0)to[out=90,in=-110](0,1);
		\draw[blue,thick,bend right=10](-.8,.8)node[left]{$\gamma_a$}to(.6,.8);
		\draw[blue] (0,.6)node{$c$} (-.11,.72)node{$\bullet$};
		\draw(0,1)node{$\bullet$};
	\end{tikzpicture}
	\caption{Intersections of type II}\label{fig:II}
\end{figure}

\begin{figure}
	\begin{tikzpicture}[xscale=1.5,yscale=1.5]
		\clip(-1.3,-.7) rectangle (1,1.1);
		\draw[red, thick, bend left=10](0,1)to(-.5,-.1);
		\draw[red, thick, bend left=10](-.5,-.1) to (.5,-.5);
		\draw[red, thick](0,1)to(.9,.4);
		\draw[red,thick,dashed,bend right=10](.5,-.5)to(.9,.4);
		\draw[blue, thick, bend right=10](-.8,.4)node[left]{$\gamma_a$}to(.9,.4);
		\draw[blue, thick, bend left=30] (0,1) to (0,-.4)node[below]{$\gamma_b$};
		\draw[blue] (.3,.2)node{$c$} (.2,.3)node{$\bullet$};
		\draw[thick](0,1)node{$\bullet$}(-.5,-.1)node{$\bullet$}(.9,.4)node{$\bullet$}(.5,-.5)node{$\bullet$};
	\end{tikzpicture}\qquad
	\begin{tikzpicture}[xscale=1.5,yscale=1.5]
		\clip(-1.3,-.7) rectangle (1,1.1);
		\draw[ultra thick, fill=gray!20](0,.2) circle (.1);
		\draw[red,thick](0,1)to[out=-145,in=90](-.6,.2)to[out=-90,in=145](0,-.6);
		\draw[red,thick](0,1)to[out=-35,in=90](.6,.2)to[out=-90,in=35](0,-.6);
		\draw[blue,thick,smooth](-1,.6)node[below]{$\gamma_b$}to[out=-5,in=170](0,.5)to[out=-10,in=90](.3,.2)to[out=-90,in=70](0,-.6);
		\draw[blue,thick,smooth](.9,-.2)node[above]{$\gamma_a$}to[out=180-5,in=180+170](0,-.1)to[out=180-10,in=180+90](-.3,.2)to[out=180-90,in=180+70](0,1);
		\draw[blue] (.1,-.2)node{$c$} (.23,-.15)node{$\bullet$};
		\draw[thick](0,1)node{$\bullet$}(0,-.6)node{$\bullet$};
	\end{tikzpicture}\qquad
	\begin{tikzpicture}[xscale=1.5,yscale=1.5]
		\clip(-1.3,-.7) rectangle (1,1.1);
		\draw[ultra thick, fill=gray!20](0,0) circle (.1);
		\draw[red,thick](0,1)to[out=-135,in=60](-.7,0.3)to[out=-120,in=180](0,-.6)to[out=0,in=-60](.7,0.3)to[out=120,in=-45](0,1);
		\draw[blue,thick](-1,.3)node[above]{$\gamma_b$}to[bend left=5](0,.3)to[out=-10,in=90](.3,0)to[out=-90,in=0](0,-.3)to[out=180,in=-90](-.3,0)to[out=90,in=-110](0,1);
		\draw[blue,thick](-1,.4)node[below]{$\gamma_a$}to[bend left=5](0,.4)to[out=-10,in=90](.4,0)to[out=-90,in=0](0,-.4)to[out=180,in=-90](-.4,0)to[out=90,in=-120](0,1);
		\draw[blue] (-.47,0.23)node{$c$} (-.36,.16)node[above]{$\bullet$};
		\draw(0,1)node{$\bullet$};
	\end{tikzpicture}
	\caption{Intersections of type III}\label{fig:III}
\end{figure}

\begin{proof}
	By our assumption, the intersection $c$ lives in the interior of a region $\Delta$ of $\TT$. By \ref{itm:t2} in Definition~\ref{def:permissible}, $\gamma_1^{c_1}$ and $\gamma_2^{c_2}$ cut out angles from $\Delta$. There are the following cases.
	\begin{itemize}
		\item Both $\gamma_1^{c_1}$ and $\gamma_2^{c_2}$ have no self-intersections, and they share a common curve in $\TT$. By \ref{itm:t1} in Definition~\ref{def:permissible}, two end segments in the same region of $\TT$ that cross the same side are homotopic to each other, so at least one of $\gamma_1^{c_1}$ and $\gamma_2^{c_2}$ is not an end segment. Hence depending on whether $\gamma_1^{c_1}$ and $\gamma_2^{c_2}$ cut the same angle and whether one of $\gamma_1^{c_1}$ and $\gamma_2^{c_2}$ is an end segment, there are the four subcases shown in Figure~\ref{fig:I}, where $\{a,b\}=\{1,2\}$. 
		\item If both $\gamma_1^{c_1}$ and $\gamma_2^{c_2}$ have no self-intersections, after the intersection $c$, then they do not share a common curve in $\TT$. If both $\gamma_1$ and $\gamma_2$ are not end segments, then we are in one of the first two situations in Figure~\ref{fig:I}, a contradiction. Hence we have the following two subcases.
        \begin{itemize}
            \item  If exactly one of $\gamma_1^{c_1}$ and $\gamma_2^{c_2}$ is an end segment, denote by $\gamma_b$ the end one and $\gamma_a$ the other one. Then $\gamma_b$ has one of the forms in Figure~\ref{fig:ro}. Since $\gamma_a$ need to cross $\gamma_b$, the only possible relative positions between $\gamma_a$ and $\gamma_b$ are shown in the cases in Figure~\ref{fig:II}, respectively.
            \item If both of $\gamma_1^{c_1}$ and $\gamma_2^{c_2}$ are end segments, then they have the same form in one of the first three cases in Figure~\ref{fig:ro} (because in the fourth case, they will not cross in their interiors). So we are in one of the situations in Figure~\ref{fig:III}.
        \end{itemize}
        \end{itemize}
\end{proof}

The intersections $c$ in the cases in Figures~\ref{fig:I}, \ref{fig:II} and \ref{fig:III} are called of types I, II and III, respectively. To distinguish the maps from $M(\gamma_1,\kappa_1)$ to $\tau M(\gamma_2,\kappa_2)$ and those from $M(\gamma_2,\kappa_2)$ to $\tau M(\gamma_1,\kappa_1)$, we introduce the following notion.

\begin{definition}[Black intersection]\label{def:black}
	An interior intersection $c$ between $\gamma_1$ and $\gamma_2$ is called a \emph{black} interior intersection from $\gamma_1$ to $\gamma_2$ if it has one of the forms in Figures~\ref{fig:I}, \ref{fig:II} and \ref{fig:III}, where $a=1$ and $b=2$. We denote by $\Int^\bullet(\gamma_1,\gamma_2)$ the number of black interior intersections from $\gamma_1$ to $\gamma_2$.
	
	A punctured intersection $(t_1,t_2)$ between $\gamma_1$ and $\gamma_2$ is called a \emph{black} punctured intersection from $\gamma_1$ to $\gamma_2$ if $\gamma_1|_{t_1\rightarrow(1-t_1)}$ is left to $\gamma_2|_{t_2\rightarrow(1-t_2)}$ (see Figure~\ref{fig:black}). We denote by $\mathfrak{P}^\bullet(\gamma_1,\gamma_2)$ the set of black punctured intersections from $\gamma_1$ to $\gamma_2$.
	\begin{figure}[htpb]
		\begin{tikzpicture}[xscale=1.5,yscale=1.5,rotate=90]
			\draw(0,1)node{$\bullet$};
			\draw(0,0)node{$\bullet$};
			\draw[red,thick](0,1)to[out=-125,in=60](-.4,0.3)to[out=-120,in=180](0,-.5)to[out=0,in=-60](.4,0.3)to[out=120,in=-55](0,1);
			\draw[blue,thick] (-.5,-.5)node[right]{$\gamma_2$}to (0,0)to(.5,-.5)node[right]{$\gamma_1$};
		\end{tikzpicture}
		\caption{A black punctured intersection from $\gamma_1$ to $\gamma_2$}\label{fig:black}
	\end{figure}
	
	A tagged intersection between $(\gamma_1,\kappa_1)$ and $(\gamma_2,\kappa_2)$ is called a \emph{black} tagged intersection from $(\gamma_1,\kappa_1)$ to $(\gamma_2,\kappa_2)$ if it is a black punctured intersection from $\gamma_1$ to $\gamma_2$. We denote by $\mathfrak{T}^\bullet((\gamma_1,\kappa_1),(\gamma_2,\kappa_2))$ the set of black tagged intersections from $(\gamma_1,\kappa_1)$ to $(\gamma_2,\kappa_2)$.
	
	We call the number $$\Int^\bullet((\gamma_1,\kappa_1),(\gamma_2,\kappa_2)):=\Int^\bullet(\gamma_1,\gamma_2)+|\mathfrak{T}^\bullet((\gamma_1,\kappa_1),(\gamma_2,\kappa_2))|$$
	the \emph{black intersection number} from $(\gamma_1,\kappa_1)$ to $(\gamma_2,\kappa_2)$.
\end{definition}

\begin{remark}
    By the assumption that $\gamma_1$ and $\gamma_2$ are in a minimal position 
 and the assumption that any intersection in $\gamma_1\cap^\circ\gamma_2$ is not in a curve in $\TT$, an intersection that has one of the forms in Figure~\ref{fig:I} can move to other regions. But its type is always I and it is always a black intersection from $\gamma_a$ to $\gamma_b$. So the notion of types and black intersections is well-defined. The other types of intersections cannot be moved since the two curves do not cross the same curve in $\TT$ after that intersection.
\end{remark}

\begin{remark}
	In each of the last two cases in Figure~\ref{fig:III}, there is one more intersection between $\gamma_a$ and $\gamma_b$, different from $c$, which is a black intersection from $\gamma_b$ to $\gamma_a$ by definition.
\end{remark}

By Lemma~\ref{lem:types}, an interior intersection between $\gamma_1$ and $\gamma_2$ is either a black intersection from $\gamma_1$ to $\gamma_2$ or a black intersection from $\gamma_2$ to $\gamma_1$. So we have
$$|\gamma_1\cap^\circ\gamma_2|=\Int^\bullet(\gamma_1,\gamma_2)+\Int^\bullet(\gamma_2,\gamma_1).$$
By the definitions of black punctured intersections and black tagged intersections, we also have the disjoint unions
$$\mathfrak{P}(\gamma_1,\gamma_2)=\mathfrak{P}^\bullet(\gamma_1,\gamma_2)\cup\mathfrak{P}^\bullet(\gamma_2,\gamma_1)$$
and
$$\mathfrak{T}((\gamma_1,\kappa_1),(\gamma_2,\kappa_2))=\mathfrak{T}^\bullet((\gamma_1,\kappa_1),(\gamma_2,\kappa_2))\cup\mathfrak{T}^\bullet((\gamma_2,\kappa_2),(\gamma_1,\kappa_1)).$$
So we have
$$\Int((\gamma_1,\kappa_1),(\gamma_2,\kappa_2))=\Int^\bullet((\gamma_1,\kappa_1),(\gamma_2,\kappa_2))+\Int^\bullet((\gamma_2,\kappa_2),(\gamma_1,\kappa_1)).$$ Hence to show Theorem~\ref{thm:int-dim}, it suffices to show the following result.

\begin{theorem}\label{thm:bint-dim}
	Let $(\gamma_1,\kappa_1),(\gamma_2,\kappa_2)\in\PTS$. Then we have	
	\begin{equation}\label{eq:b-tau}
		\Int^\bullet((\gamma_1,\kappa_1),(\gamma_2,\kappa_2))=\dim_\k\Hom(M(\gamma_1,\kappa_1),\tau M(\gamma_2,\kappa_2)).
	\end{equation}
\end{theorem}

\begin{remark}\label{rmk:inj}
    The formula \eqref{eq:b-tau} interprets the dimension of Hom between two indecomposable modules in $\mathcal{S}$ via the intersection number between the corresponding tagged curves, except the second module is injective. However, this exceptional case can be included in the following way. By Theorem~\ref{thm:tau}, all indecomposable injective modules are    
    $M(\rho(\gamma,\kappa))$ for $(\gamma,\kappa)\in \TT^\times$. We claim
    \begin{equation}\label{eq:inj}
    \dim_\k\Hom(M(\gamma_1,\kappa_1), M(\rho(\gamma,\kappa)))=\Int((\gamma_1,\kappa_1),(\gamma,\kappa))
    \end{equation}
    holds for any $(\gamma_1,\kappa_1)\in\PTS$ and any $(\gamma,\kappa)\in \TT^\times$. 

    To show \eqref{eq:inj}, let $e$ be the idempotent associated to $(\gamma,\kappa)$ (see Remark~\ref{rmk:times}) and $S$ the corresponding simple module. Let $s$ be the number of $S$ appearing in the composition series of $M(\gamma_1,\kappa_1)$. Then $s=\dim_\k\Hom(M(\gamma_1,\kappa_1), M(\rho(\gamma,\kappa)))$. Recall from Construction~\ref{cons:mod} that we have $M(\gamma_1,\kappa_1)=M(\M(\gamma_1),N_{\kappa_1})$. 
    \begin{enumerate}
        \item If $e=e_j$ for some $j\in Q_0^{sp}\setminus Sp$, by Construction~\ref{construct:M(m,N)}, $s$ equals the number of times of the word $\M(\gamma_1)$ passing through the vertex $j$. By the construction of $\M(\gamma_1)$ (see Section~\ref{subsec:preper}), this number is equal to $|\gamma_1\cap^\circ\gamma|=\Int((\gamma_1,\kappa_1),(\gamma,\kappa))$, where the last equality is due to that in this case, both of the endpoints of $\gamma$ are in $\MM$.
        \item If $e=e_i-\epsilon_i$ (resp. $\epsilon_i$) for some $i\in Sp$, by Construction~\ref{construct:M(m,N)}, $s=s_1+s_2$, where $s_1$ is the number of $\epsilon_i$ and $\epsilon_i^{-1}$ appearing in $\M(\gamma_1)$, and $s_2$ is the number of special end letters of $\M(\gamma_1)$ whose value under $N_{\kappa_1}$ is $0$ (resp. $1$). By the construction of $\M(\gamma_1)$ (see Section~\ref{subsec:preper}), $s_1=|\gamma_1\cap^\circ\gamma|$ and $s_2=|\mathfrak{T}((\gamma_1,\kappa_1),(\gamma,\kappa))|$. Hence $s=\Int((\gamma_1,\kappa_1),(\gamma,\kappa))$ as required.
    \end{enumerate}
\end{remark}

\subsection{Proof of Theorem~\ref{thm:bint-dim}}\label{subsec:4.3}

Let $\m=\omega_m\cdots\omega_1$ and $\n=\nu_n\cdots\nu_1$ be the words $\M(\gamma_1)$ and $\M(\gamma_2)$, respectively. By our setting, $\rho(\gamma_2,\kappa_2)=(\rho(\gamma_2),\kappa_2')$, where $\kappa_2'(t)=1-\kappa_2(t)$ for any $t$ with $\gamma_2(t)\in\PP$. If there is at least one black interior or punctured intersection from $\gamma_1$ to $\gamma_2$, then it follows from the cases in Figures~\ref{fig:I}, \ref{fig:II}, \ref{fig:III} and \ref{fig:black} that $\rho(\gamma_2)$ is not trivial and $\M(\gamma_2)$ and $\M(\rho(\gamma_2))$ share a common subword. So we can write $\M(\rho(\gamma_2))=\n'=\nu'_s\cdots\nu'_r$ with integers $s>r$ (where $r$ is probably non-positive or bigger than 1) such that $\n'_{(p,q)}=\n_{(p,q)}$, $\nu_p\neq\nu'_p$ and $\nu_q\neq\nu'_q$ for some integers $p<q$. 

Let $H^{\m,\n'}_l$, $l=0,1,2$, be the set of common pairs from $\m$ to $\n'$ which contains exactly $l$ special end letters. Then $H^{\m,\n'}=H^{\m,\n'}_0\cup H^{\m,\n'}_1\cup H^{\m,\n'}_2$.

\begin{lemma}\label{lem:h0}
	There is a bijection from the set of black interior intersections from $\gamma_1$ to $\gamma_2$ to the set $H^{\m,\n'}_0$.
\end{lemma}

\begin{proof}
    For any black intersection $c$ from $\gamma_1$ to $\gamma_2$, we have  $t(\omega_{c_1})=t(\nu'_{c_2})$. This is because for the case that $c$ is of type I (see Figure~\ref{fig:I} where $a=1$ and $b=2$), we have $t(\omega_{c_1})=t(\nu_{c_2})$ and  $t(\nu_{c_2})=t(\nu'_{c_2})$, where the second equality is due to that $\rho(\gamma_2)$ is between $\gamma_1$ and $\gamma_2$; for the case that $c$ is of type II or III (see Figures~\ref{fig:II} and \ref{fig:III} where $a=1$ and $b=2$), we have $c_2=n$ (i.e. $\gamma_2^{c_2}$ is the ending arc segment of $\gamma_2$) and $t(\nu'_n)=t(\omega_{c_1})$. Then $\gamma_1$ and $\rho(\gamma_2)$ share a common part containing the ending points of $\gamma_1^{c_1}$ and ${\rho(\gamma_2)}^{c_2}$, and separate (on both sides) that $\gamma_1$ is to the left of $\rho(\gamma_2)$, see Figures~\ref{fig:II'} and \ref{fig:III'}. So by Lemma~\ref{lem:order}, we get a common pair from $\m$ to $\n'$. Note that, because in each case, the common part does not contain punctures, the induced common pair does not contain special end letters.
	
    Conversely, for any common pair $((i,j),(h,l))$ from $\m$ to $\n'$ which does not contain any special end letters, it gives a common part of $\gamma_1$ and $\rho(\gamma_2)$ which does not contain punctures, with separating (on both sides) that $\gamma_1$ is same with or to the left of $\rho(\gamma_2)$. There are two cases.
    \begin{itemize}
        \item $(h,l)$ intersects $(p,q)$, then this gives a common part between $\gamma_1$ and $\gamma_2$ with separating  $\gamma_1$ is to the left of $\gamma_2$ because $\gamma_2$ is to the right of $\rho(\gamma_2)$, which induces a black intersection from $\gamma_1$ to $\gamma_2$ of type I.
        \item $(h,l)$ does not intersect $(p,q)$. Without loss of generality, assume $(h,l)$ is included in  $[q,s+1)$, which implies that $h=q$ by checking the cases in Figure~\ref{fig:forms}. Then we are in one of the situations in Figures~\ref{fig:II} and \ref{fig:III} where $a=1$ and $b=2$, that is, we get a black intersection from $\gamma_1$ to $\gamma_2$ of type II or III.
    \end{itemize}
    By construction, the above two maps are mutually inverse. Thus, we finish the proof.
\end{proof}

For any $J\in H_0^{\m,\n'}$, since it does not contain any special end letters, we have $A_J= \k$. Then $N_{\kappa_1}=\k=N_{\kappa_2'}$ as $A_J$-modules. So $\dim_{\k}\Hom_{A_J}(N_{\kappa_1},N_{\kappa_2'})=1$. Hence, by Lemma~\ref{lem:h0}, we have
\begin{equation}\label{eq:h0}
    \Int^\bullet(\gamma_1,\gamma_2)=\sum_{J\in H_0^{\m,\n'}}\dim_{\k}\Hom_{A_J}(N_{\kappa_1},N_{\kappa_2'}).
\end{equation}

\begin{figure}
	\begin{tikzpicture}[xscale=1.5,yscale=1.5]
		\clip(-1.3,-.7) rectangle (1,1.1);
		\draw[red, thick, bend left=10](0,1)to(-.5,-.1);
		\draw[red, thick, bend left=10](-.5,-.1) to (.5,-.5);
		\draw[red, thick](0,1)to(.9,.4);
		\draw[red,thick,dashed,bend right=10](.5,-.5)to(.9,.4);
		\draw[blue,thick,bend right=10](-.8,.8)node[left]{$\gamma_1$}to(.6,.8);
		\draw[blue, thick, bend left=30] (0,1) to (0,-.4)node[below]{$\gamma_2$};
		\draw[green, thick] (-.6,1)to[out=-45,in=90](0,.2)to[out=-90,in=60](-.2,-.4)node[left]{$\rho(\gamma_2)$};
		\draw[blue] (.25,.65)node{$c$} (.12,.73)node{$\bullet$};
		\draw[ultra thick, bend left=10] (0,1)to(-.6,1)node{$\bullet$};
		\draw[thick](0,1)node{$\bullet$}(-.5,-.1)node{$\bullet$}(.9,.4)node{$\bullet$}(.5,-.5)node{$\bullet$};
	\end{tikzpicture}\qquad
	\begin{tikzpicture}[xscale=1.5,yscale=1.5]
		\clip(-1.3,-.7) rectangle (1,1.1);
		\draw[ultra thick, fill=gray!20](0,.2) circle (.1);
		\draw[red,thick](0,1)to[out=-145,in=90](-.6,.2)to[out=-90,in=145](0,-.6);
		\draw[red,thick](0,1)to[out=-35,in=90](.6,.2)to[out=-90,in=35](0,-.6);
		\draw[blue,thick,smooth](-1,.6)node[below]{$\gamma_2$}to[out=-5,in=170](0,.5)to[out=-10,in=90](.3,.2)to[out=-90,in=70](0,-.6);
		\draw[blue,thick,bend left=10](-.8,-.4)node[left]{$\gamma_1$}to(.8,-.4);
		\draw[blue] (0.05,-.2)node{$c$} (.12,-.33)node{$\bullet$};
		\draw[green,thick](-1,.6)node[above]{$\rho(\gamma_2)$};
		\draw[green,thick](-1,.7)to[out=0,in=160](.2,.6)to[out=-20,in=100](.6,-.6);
		\draw[ultra thick,bend left=10](0,-.6)to(.6,-.6)node{$\bullet$};
		\draw[thick](0,1)node{$\bullet$}(0,-.6)node{$\bullet$};
	\end{tikzpicture}\qquad
	\begin{tikzpicture}[xscale=1.5,yscale=1.5]
		\clip(-1.3,-.7) rectangle (1,1.1);
		\draw[ultra thick, fill=gray!20](0,0) circle (.1);
		\draw[red,thick](0,1)to[out=-135,in=60](-.7,0.3)to[out=-120,in=180](0,-.6)to[out=0,in=-60](.7,0.3)to[out=120,in=-45](0,1);
		\draw[blue,thick](-1,.3)node[below]{$\gamma_2$}to[bend left=5](0,.3)to[out=-10,in=90](.3,0)to[out=-90,in=0](0,-.3)to[out=180,in=-90](-.3,0)to[out=90,in=-110](0,1);
		\draw[blue,thick,bend right=10](-.8,.8)node[left]{$\gamma_1$}to(.6,.8);
		\draw[blue] (0,.6)node{$c$} (-.11,.72)node{$\bullet$};
		\draw[green](-1,.3)node[above]{$\rho(\gamma_2)$};
		\draw[green,thick](-1,.4)to[bend left=5](0,.4)to[out=-10,in=90](.4,0)to[out=-90,in=0](0,-.4)to[out=180,in=-90](-.4,0)to[out=90,in=-70](-.6,1);
		\draw[ultra thick,bend right=10] (-.6,1)node{$\bullet$}to(0,1);
		\draw(0,1)node{$\bullet$};
	\end{tikzpicture}
	\caption{Common parts from $\gamma_1$ to $\rho(\gamma_2)$ via intersections $c$ of type II}\label{fig:II'}
\end{figure}

\begin{figure}
	\begin{tikzpicture}[xscale=1.5,yscale=1.5]
		\clip(-1.3,-.7) rectangle (1,1.1);
		\draw[red, thick, bend left=10](0,1)to(-.5,-.1);
		\draw[red, thick, bend left=10](-.5,-.1) to (.5,-.5);
		\draw[red, thick](0,1)to(.9,.4);
		\draw[red,thick,dashed,bend right=10](.5,-.5)to(.9,.4);
		\draw[blue, thick, bend right=10](-.8,.4)node[left]{$\gamma_1$}to(.9,.4);
		\draw[blue, thick, bend left=30] (0,1) to (0,-.4)node[below]{$\gamma_2$};
		\draw[green, thick] (-.6,1)to[out=-45,in=90](0,.2)to[out=-90,in=60](-.2,-.4)node[left]{$\rho(\gamma_2)$};
		\draw[blue] (.3,.2)node{$c$} (.2,.3)node{$\bullet$};
		\draw[ultra thick, bend left=10] (0,1)to(-.6,1)node{$\bullet$};
		\draw[thick](0,1)node{$\bullet$}(-.5,-.1)node{$\bullet$}(.9,.4)node{$\bullet$}(.5,-.5)node{$\bullet$};
	\end{tikzpicture}\qquad
	\begin{tikzpicture}[xscale=1.5,yscale=1.5]
		\clip(-1.3,-.7) rectangle (1,1.1);
		\draw[ultra thick, fill=gray!20](0,.2) circle (.1);
		\draw[red,thick](0,1)to[out=-145,in=90](-.6,.2)to[out=-90,in=145](0,-.6);
		\draw[red,thick](0,1)to[out=-35,in=90](.6,.2)to[out=-90,in=35](0,-.6);
		\draw[blue,thick,smooth](-1,.6)node[below]{$\gamma_2$}to[out=-5,in=170](0,.5)to[out=-10,in=90](.3,.2)to[out=-90,in=70](0,-.6);
		\draw[blue,thick,smooth](.9,-.2)node[above]{$\gamma_1$}to[out=180-5,in=180+170](0,-.1)to[out=180-10,in=180+90](-.3,.2)to[out=180-90,in=180+70](0,1);
		\draw[blue] (.1,-.2)node{$c$} (.23,-.15)node{$\bullet$};
		\draw[green,thick](-1,.6)node[above]{$\rho(\gamma_2)$};
		\draw[green,thick](-1,.7)to[out=0,in=160](.2,.6)to[out=-20,in=100](.6,-.6);
		\draw[ultra thick,bend left=10](0,-.6)to(.6,-.6)node{$\bullet$};
		\draw[thick](0,1)node{$\bullet$}(0,-.6)node{$\bullet$};
	\end{tikzpicture}\qquad
	\begin{tikzpicture}[xscale=1.5,yscale=1.5]
		\clip(-1.3,-.7) rectangle (1,1.1);
		\draw[ultra thick, fill=gray!20](0,0) circle (.1);
		\draw[red,thick](0,1)to[out=-135,in=60](-.7,0.3)to[out=-120,in=180](0,-.6)to[out=0,in=-60](.7,0.3)to[out=120,in=-45](0,1);
		\draw[blue,thick](-1,.3)node[above]{$\gamma_2$}to[bend left=5](0,.3)to[out=-10,in=90](.3,0)to[out=-90,in=0](0,-.3)to[out=180,in=-90](-.3,0)to[out=90,in=-110](0,1);
		\draw[blue,thick](-1,.4)node[below]{$\gamma_1$}to[bend left=5](0,.4)to[out=-10,in=90](.4,0)to[out=-90,in=0](0,-.4)to[out=180,in=-90](-.4,0)to[out=90,in=-120](0,1);
		\draw[blue] (-.47,0.23)node{$c$} (-.36,.16)node[above]{$\bullet$};
		\draw[green](-1,.6)node[above]{$\rho(\gamma_2)$};
		\draw[green,thick](-1.2,.6)to[bend left=5](0,.6)to[out=-10,in=90](.6,0)to[out=-90,in=0](0,-.5)to[out=180,in=-90](-.6,0)to[out=90,in=-100](-.6,1);
		\draw[ultra thick,bend right=10] (-.6,1)node{$\bullet$}to(0,1);
		\draw(0,1)node{$\bullet$};
	\end{tikzpicture}
	\caption{Common parts from $\gamma_1$ to $\rho(\gamma_2)$ via intersections $c$ of type III}\label{fig:III'}
\end{figure}

Let
$$\mathfrak{P}^\bullet_1(\gamma_1,\gamma_2)=\{(t_1,t_2)\in\mathfrak{P}^\bullet(\gamma_1,\gamma_2)\mid \gamma_1|_{t_1\rightarrow(1-t_1)}\nsim\gamma_2|_{t_2\rightarrow(1-t_2)}\},$$
and
$$\mathfrak{T}^\bullet_1((\gamma_1,\kappa_1),(\gamma_2,\kappa_2))=\{(t_1,t_2)\in\mathfrak{P}^\bullet_1(\gamma_1,\gamma_2)\mid \kappa_1(t_1)\neq\kappa_2(t_2) \}.$$

\begin{lemma}\label{lem:h1}
    There is a bijection between $\mathfrak{P}^\bullet_1(\gamma_1,\gamma_2)$ and $H_1^{\m,\n'}$.
\end{lemma}

\begin{proof}
    This directly follows from Lemma~\ref{lem:order}.
\end{proof}

For any $(t_1,t_2)\in \mathfrak{P}^\bullet_1(\gamma_1,\gamma_2)$, denote by $J$ the corresponding common pair in $H_1^{\m,\n'}$. Then $A_{J}=\k[x]/(x^2-x)$ and $N_{\kappa_1}=\k_{\kappa_1(t_1)}$ and $N_{\kappa_2'}=\k_{1-\kappa_2(t_2)}$ as $A_J$-modules. Using the formula \eqref{eq:1}, we have
$$\dim_\k\Hom_{A_J}(N_{\kappa_1},N_{\kappa_2'})=\begin{cases}
	1&\text{if }(t_1,t_2)\in\mathfrak{T}_1^\bullet((\gamma_1,\kappa_1),(\gamma_2,\kappa_2)),\\
	0&\text{otherwise.}
\end{cases}$$
Hence, by Lemma~\ref{lem:h1}, we have
\begin{equation}\label{eq:h1}
	|\mathfrak{T}_1^\bullet((\gamma_1,\kappa_1),(\gamma_2,\kappa_2))|=\sum_{J\in H_1^{\m,\n'}}\dim_\k\Hom_{A_J}(N_{\kappa_1},N_{\kappa_2'}).
\end{equation}

Let
$$\mathfrak{P}_2(\gamma_1,\gamma_2)=\{(t_1,t_2)\in\mathfrak{P}(\gamma_1,\gamma_2)\mid \gamma_1|_{t_1\rightarrow(1-t_1)}\sim\gamma_2|_{t_2\rightarrow(1-t_2)},\gamma_1(1-t_1)=\gamma_2(1-t_2)\in \PP\},$$
and
$$\mathfrak{P}^\bullet_2(\gamma_1,\gamma_2)=\mathfrak{P}_2(\gamma_1,\gamma_2)\cap\mathfrak{P}^\bullet(\gamma_1,\gamma_2).$$
Note that for each $(t_1,t_2)\in\mathfrak{P}_2$, $(1-t_1,1-t_2)$ is also in $\mathfrak{P}_2$, with one of which being black from $\gamma_1$ to $\gamma_2$ while the other one being black from $\gamma_2$ to $\gamma_1$. So if the set $\mathfrak{P}_2$ is not empty, it contains exactly one black punctured intersection, say $(t_1,t_2)$, from $\gamma_1$ to $\gamma_2$. In this case, we have only one common pair $J$ in $H_2^{\m,\n'}$ (which is the one standing for the identity between $\m$ or its inverse and $\n=\n'$). So we have $A_J=\k\<x,y\>/(x^2-x,y^2-y)$ and $N_{\kappa_1}=\k_{\kappa_1(t_1),\kappa_1(1-t_1)}$ and $N_{\kappa_2'}=\k_{1-\kappa_2(t_2),1-\kappa_2(1-t_2)}$ as $A_J$-modules. Using formula~\eqref{eq:2}, we have 
$$\dim_\k\Hom_{A_J}(N_{\kappa_1},N_{\kappa_2'})=\begin{cases}
	1&\text{if }(t_1,t_2)\in\mathfrak{T}_2^\bullet((\gamma_1,\kappa_1),(\gamma_2,\kappa_2)),\\
	0&\text{otherwise,}
\end{cases}$$
where 
$$\mathfrak{T}_2^\bullet((\gamma_1,\kappa_1),(\gamma_2,\kappa_2)):=\{(t_1,t_2)\in\mathfrak{P}_2^\bullet(\gamma_1,\gamma_2)\mid \kappa_1(t_1)\neq\kappa_2(t_2),\kappa_1(1-t_1)\neq\kappa_2(1-t_2) \}$$
is the set of the black tagged intersections from $(\gamma_1,\kappa_1)$ to $(\gamma_2,\kappa_2)$. Hence, we have
\begin{equation}\label{eq:h2}
	|\mathfrak{T}_2^\bullet((\gamma_1,\kappa_1),(\gamma_2,\kappa_2))|=\sum_{J\in H_2^{\m,\n'}}\dim_\k\Hom_{A_J}(N_{\kappa_1},N_{\kappa_2'})
\end{equation}
In the case that $\mathfrak{P}_2$ is empty, there are no common pairs in $H_2^{\m,\n'}$. So we also have equation~\eqref{eq:h2}.

By definition, we have
$$|\mathfrak{T}^\bullet((\gamma_1,\kappa_1),(\gamma_2,\kappa_2))|=|\mathfrak{T}^\bullet_1((\gamma_1,\kappa_1),(\gamma_2,\kappa_2))|+|\mathfrak{T}^\bullet_2((\gamma_1,\kappa_1),(\gamma_2,\kappa_2))|;$$
by Theorem~\ref{thm:dim formular}, we have
$$\dim_\k\Hom_A(M(\m,N_{\kappa_1}),M(\n',N_{\kappa_2'}))=\sum\limits_{J\in H^{\m,\n'}}\dim_\k\Hom_{A_J}(N_{\kappa_1},N_{\kappa_2'});$$
and by Theorem~\ref{thm:tau}, we have
$$M(\rho(\gamma_2,\kappa_2))=\tau M(\gamma_2,\kappa_2).$$
Combining these with \eqref{eq:h0}, \eqref{eq:h1} and \eqref{eq:h2}, we have
$$\begin{array}{rcl}
	\Int^\bullet((\gamma_1,\kappa_1),(\gamma_2,\kappa_2))&=&\Int^\bullet(\gamma_1,\gamma_2)+|\mathfrak{T}^\bullet((\gamma_1,\kappa_1),(\gamma_2,\kappa_2))|\\
	&=&\Int^\bullet(\gamma_1,\gamma_2)+|\mathfrak{T}^\bullet_1((\gamma_1,\kappa_1),(\gamma_2,\kappa_2))|+|\mathfrak{T}^\bullet_2((\gamma_1,\kappa_1),(\gamma_2,\kappa_2))|\\
	&=&\sum\limits_{J\in H^{\m,\n'}}\dim_\k\Hom_{A_J}(N_{\kappa_1},N_{\kappa_2'})\\
	&=&\dim_\k\Hom_A(M(\m,N_{\kappa_1}),M(\n',N_{\kappa_2'}))\\
	&=&\dim_\k\Hom_A(M(\gamma_1,\kappa_1),M(\rho(\gamma_2,\kappa_2)))\\
	&=&\dim_\k\Hom_A(M(\gamma_1,\kappa_1),\tau M(\gamma_2,\kappa_2))
\end{array}$$
This finishes the proof of Theorem~\ref{thm:bint-dim} and hence finishes the proof of Theorem~\ref{thm:int-dim}.

\section{Tau-tilting theory for skew-gentle algebras}

We recall some definitions and notations of $\tau$-tilting theory from \cite{AIR}. For a finite dimensional basic $\k$-algebra $\Lambda$, a $\Lambda$-module $M$ is called \emph{$\tau$-rigid} provided
$$\Hom_\Lambda(M,\tau M)=0.$$
A $\tau$-rigid module $M$ is called \emph{$\tau$-tilting} if $M$ is $\tau$-rigid and $|M|=|\Lambda|$, where $|M|$ denotes the number of non-isomorphic indecomposable direct summands of $M$. A $\tau$-rigid module $M$ is called \emph{support $\tau$-tilting} if there exists an idempotent $e$ of $\Lambda$ such that $M$ is a $\tau$-tilting $(\Lambda/\<e\>)$-module. A $\Lambda$-module $M$ is called \emph{rigid} if $\Ext^1_\Lambda(M,M)=0$. By the Auslander-Reiten formula, any $\tau$-rigid module is rigid.

A pair $(M,P)$ of a $\Lambda$-module $M$ and a projective $\Lambda$-module $P$ is called a \emph{$\tau$-rigid pair} if $M$ is $\tau$-rigid and $\Hom_\Lambda(P,M)=0$. A $\tau$-rigid pair $(M,P)$ is called a \emph{support $\tau$-tilting pair} if $(M,P)$ is $\tau$-rigid and $|M|+|P|=|\Lambda|$. As usual, we say that $(M,P)$ is basic if $M$ and $P$ are basic. We shall use the following two results.

\begin{lemma}[{\cite[Corollary~2.13]{AIR}}]\label{lem:air1}
    A $\tau$-rigid pair is support $\tau$-tilting if and only if it is maximal with respect to the $\tau$-rigid pair property.
\end{lemma}

\begin{lemma}[{\cite[Proposition~2.3]{AIR}}]\label{lem:air2}
    A $\Lambda$-module $M$ is support $\tau$-tilting if and only if there exists a projective $\Lambda$-module $P$ such that $(M,P)$ is a support $\tau$-tilting pair.
\end{lemma}

Now we apply our results in the previous sections to study the $\tau$-tilting theory for the skew-gentle algebra $A$. We denote by
\begin{itemize}
    \item $\ir A$: the set of isoclasses of indecomposable $\tau$-rigid $A$-modules;
    \item $\tr A$: the set of isoclasses of basic $\tau$-rigid $A$-modules;
    \item $\st A$: the set of isoclasses of basic support $\tau$-tilting $A$-modules;
    \item $\trp A$: the set of isoclasses of basic $\tau$-rigid pairs;
    \item $\stp A$: the set of isoclasses of basic support $\tau$-tilting pairs.
\end{itemize}

Recall that $\PTS$ denotes the set of tagged permissible curves (up to inverse) and $\GPTS$ denotes the set of tagged generalized permissible curves (up to inverse). The set $\TT^\times=\GPTS\setminus\PTS$ is the tagged version of the admissible partial triangulation $\TT$.

\begin{proposition}\label{prop:tau-rigid}
    There is a bijection 
    $$\begin{array}{ccc}
        \{(\gamma,\kappa)\in \PTS\mid \Int((\gamma,\kappa),(\gamma,\kappa))=0 \}&\to&\ir A\\(\gamma,\kappa)&\mapsto&M(\gamma,\kappa)
    \end{array}$$
\end{proposition}

\begin{proof}
    By Lemma~\ref{lem:band}, the set $\mathcal{S}$ consists of all of string modules and the band modules which are on the bottom of tubes of rank 2, and any indecomposable module which is not in $\mathcal{S}$ is either in a homogeneous tube or in the $r$-th level, $r\geq 2$, of a tube of rank 2. In particular, any indecomposable $\tau$-rigid $A$-module is in $\mathcal{S}$. Hence we get the required bijection by the bijection in Theorem~\ref{thm:curve and mod} and the formula in Theorem~\ref{thm:int-dim}.
\end{proof}

Denote by $\operatorname{indproj}(A)$ the set of isoclasses of indecomposable projective $A$-modules. By Theorem~\ref{thm:tau}, we have the following.

\begin{corollary}\label{cor:proj}
    The map $M$ gives a bijection
    $$\begin{array}{ccc}
        \{(\gamma,\kappa)\in \PTS\mid \rho(\gamma,\kappa)\in\TT^\times \}&\to&\operatorname{indproj}(A).
        \end{array}$$
\end{corollary}

\begin{definition}\label{def:t.p.d}
    A collection $R$ of tagged generalized permissible curves in $\GPTS$ is called a \emph{partial generalized dissection} of $\surf$ (with respect to $\TT$) if 
    $$\Int((\gamma_1,\kappa_1),(\gamma_2,\kappa_2))=0$$
    for any $(\gamma_1,\kappa_1)$ and $(\gamma_2,\kappa_2)$ in $R$. Denote by $\PTGPD(\surf)$ the set of partial generalized dissections of $\surf$. Let $R\in\PTGPD(\surf)$.
    \begin{enumerate}
        \item We call $R$ a \emph{generalized dissection} of $\surf$ if it is maximal in $\PTGPD(\surf)$. Denote by $\TGPD(\surf)$ the set of generalized dissections of $\surf$.
        \item We call $R$ a \emph{partial dissection} of $\surf$ if $R\subseteq\PTS$. Denote by $\PTPD(\surf)$ the set of partial dissections of $\surf$.
        \item We call $R$ a \emph{dissection} of $\surf$ if $R\subseteq\PTS$ and $|R|=|\TT^\times|$, where $|R|$ denotes the number of elements in $R$. Denote by $\TPD(\surf)$ the set of dissections of $\surf$.
    \end{enumerate}
\end{definition}

\begin{proposition}\label{prop:tau-rigid2}
    There is a bijection
    $$\begin{array}{ccc}
        \PTPD(\surf)&\to&\tr A\\
        R&\mapsto& \bigoplus\limits_{(\gamma,\kappa)\in R}M(\gamma,\kappa)
    \end{array}$$
    which restricts to a bijection from $\TPD(\surf)$ to $\ttt A$.
\end{proposition}

\begin{proof}
    The first bijection follows from Proposition~\ref{prop:tau-rigid} and the formula in Theorem~\ref{thm:int-dim}. The second bijection follows from the first one and $|\TT^\times|=|Q_0|+|Sp|=|A|$.
\end{proof}

\begin{theorem}\label{thm:tau-rigid}
	There is a bijection
	$$\begin{array}{ccc}
		\PTGPD(\surf)&\to&\trp A\\
		R&\mapsto& \left(\bigoplus\limits_{(\gamma,\kappa)\in R\setminus\TT^\times}M(\gamma,\kappa),\bigoplus\limits_{\rho(\gamma,\kappa)\in R\cap\TT^\times}M(\gamma,\kappa)\right)
	\end{array}$$
	which restricts to a bijection from $\TGPD(\surf)$ to $\stp A$.
\end{theorem}

\begin{proof}
    Let $(\gamma,\kappa), (\gamma',\kappa')\in \PTS$ be such that $\rho(\gamma',\kappa')\in\TT^\times$. By the construction of $M(\gamma,\kappa)$ (Construction~\ref{cons:mod}), $\Int((\gamma,\kappa),\rho(\gamma',\kappa'))=0$ if and only if the simple top of the projective module $M(\gamma',\kappa')$ does not occur in the composition series of $M(\gamma,\kappa)$. This is equivalent to $\Hom_A(M(\gamma',\kappa'),M(\gamma,\kappa))=0$. So by Corollary~\ref{cor:proj} and Proposition~\ref{prop:tau-rigid2}, we get the first bijection. The second bijection follows from the first one and Lemma~\ref{lem:air1}.
\end{proof}

We have the following immediate consequences.

\begin{corollary}\label{cor:1}
    Let $R\in\PTGPD(\surf)$. Then $R\in\TGPD(\surf)$ if and only if $|R|=|\TT^\times|$. In particular, we have $\TT^\times\in\TGPD(\surf)$ and $\TPD(\surf)\subset\TGPD(\surf)$.
\end{corollary}

\begin{proof}
    This follows from Theorem~\ref{thm:tau-rigid} and  $|\TT^\times|=|A|$. 
\end{proof}

\begin{corollary}\label{cor:supp}
	There is a bijection
	$$\begin{array}{ccc}
		\TGPD(\surf)&\to&\st A\\
		R&\mapsto&\bigoplus\limits_{(\gamma,\kappa)\in R}M([\gamma],\kappa)
	\end{array}$$
\end{corollary}

\begin{proof}
	This follows from Theorem~\ref{thm:tau-rigid} and Lemma~\ref{lem:air2}. 
\end{proof}

\section{An example}\label{sec:example}

Let $\surf$ be a punctured marked surface as shown in the left picture of Figure~\ref{fig:exm}, where an admissible partial triangulation $\TT$ is given. The four types $\aaa$-$\dd$ of regions all appear in $\TT$. The associated quiver $Q^\TT$ is shown in the right picture of Figure~\ref{fig:exm} with the relation set $R^\TT=\{\epsilon_1^2-\epsilon_1,\epsilon_5^2,ab,ba,ed,hc\}$. The corresponding skew-tiling algebra $A=\Lambda^\TT$ is a skew-gentle algebra from the skew-gentle triple $(Q,Sp,I)$ with $Q=Q^\TT\setminus\{\varepsilon_1\}$, $Sp=\{1\}$ and $I=R^\TT\setminus\{\epsilon_1^2-\epsilon_1\}$.

\begin{figure}[htpb]
	\begin{tikzpicture}[scale=.9]
		\draw[ultra thick] (0,0) circle (3);
		\draw[ultra thick,fill=gray!20] (0,0) circle (.5);
		\draw[ultra thick,fill=gray!20](-70:1.7) circle (.2);
		\draw[ultra thick,fill=gray!20](90:1.5) circle (.2);
		\draw[red,thick] (0,-.5)to[out=-170,in=140](60:-2)to[out=-40,in=-100](0,-.5); 
		\draw[red,thick]
		(0,-.5)to[out=-170,in=90](-140:2)to[out=-90,in=-180](-90:2.5)to[out=0,in=-90](-40:2)to[out=90,in=-10](0,-.5);
		\draw[red,thick] (0,-.5)to[out=180,in=-90](160:2.5)to[out=90,in=180](90:2.5)to[out=0,in=140](30:3);
		\draw[red,thick] (0,-.5)to[bend right=10](3,0);
		\draw[red,thick] (0,.5)to[out=30,in=-90](60:1.5)to[out=90,in=0](90:2)to[out=180,in=90](120:1.5)to[out=-90,in=150](0,.5);
		\draw[red,thick,bend left=15] (30:3)to(0,.5);
		\draw[red,thick] (3,0)to[out=160,in=-10](0,.5);
		\draw[red] (0,-2.4)node[below]{$2$} (-.3,-1.7)node{$1$} (-2.2,-.1)node{$3$} (25:2.5)node{$4$} (-1,1.3)node{$5$} (2,.5)node{$6$}
		(2,-.5)node{$7$};
		\draw[blue,->-=.7,>=stealth] (33:-1.5)to[bend right=10](49:-1.5)node[left]{$b$};
		\draw[blue,->-=.7,>=stealth] (50:-1)to[bend right=10](-98:1) (-98:1.05)node[left]{$\epsilon_1$};
		\draw[blue,->-=.7,>=stealth] (-98:1)to[bend right=10](-39:1) (-70:1.1)node{$a$};
		\draw[blue,->-=.7,>=stealth] (10:-1.5)to[bend right=10](33:-1.5) (24:-1.5)node[left]{$c$};
		\draw[blue,->-=.7,>=stealth] (41:2.84)to[bend right=10](29:2.6) (37:2.4)node{$d$};
		\draw[blue,->-=.7,>=stealth] (31:1.5)to[bend right=10](60:1.5) (39:1.7)node{$e$};
		\draw[blue,->-=.7,>=stealth] (57:1)to[bend right=10](123:1) (90:1.05)node{$\epsilon_5$};
		\draw[blue,->-=.7,>=stealth] (12:1.5)to[bend right=10](31:1.5) (21:1.5)node[right]{$f$};
		\draw[blue,->-=.7,>=stealth] (6:2.3)to[bend right=10](-6:2.3) (0:2.3)node[left]{$g$};
		\draw[blue,->-=.7,>=stealth] (-39:1.95)to[bend right=10](-13:1.7) (-26:1.8)node[right]{$h$};
		\draw (0,-.5)node{$\bullet$} (0,.5)node{$\bullet$} (30:3)node{$\bullet$} (3,0)node{$\bullet$}
		(65:-1.4)node{$\bullet$};
	\end{tikzpicture}
	\begin{tikzpicture}[scale=.4]
		\draw(-7,4)node{$Q^\TT:$};
		\draw(0,0)node{$\xymatrix{
				\ar@(lu,ld)[]_{\epsilon_1}1\ar@/^.5pc/[r]^{a}&2\ar@/^.5pc/[l]^{b}\ar[rd]_h&3\ar[l]_c\ar[r]^d&4\ar[r]^e&5\ar@(ru,rd)[]^{\epsilon_5}\\
				&&7&6\ar[u]_f\ar[l]^g
			}$};
		\draw(0,-5)node{$R^\TT=\{\epsilon_1^2-\epsilon_1,\epsilon_5^2,ab,ba,ed,hc\}$};
		\draw(0,-6)node{$ $};
	\end{tikzpicture}
	\caption{An example of an admissible partial triangulation of a punctured marked surface and the associated quiver with relation}\label{fig:exm}
\end{figure}

To show a tagged permissible curve $(\gamma,\kappa)$ in a picture, as in \cite{FST,QZ},  we use the symbol $\times$ on one end of $\gamma$ to denote the value of $\kappa$ on this end is $1$. The (solid) tagged curves $\gamma^\times_i=(\gamma_i,\kappa_i), i=1,2,3$ shown in the left picture of Figure~\ref{fig:exm2} are all tagged permissible curves, whose tagged rotations $\rho(\gamma^\times_i), i=1,2,3$ are shown in the right picture of Figure~\ref{fig:exm2}.

\begin{figure}[htpb]
	\begin{tikzpicture}[scale=.9]
		\draw[ultra thick] (0,0) circle (3);
		\draw[ultra thick,fill=gray!20] (0,0) circle (.5);
		\draw[ultra thick,fill=gray!20](-70:1.7) circle (.2);
		\draw[ultra thick,fill=gray!20](90:1.5) circle (.2);
		\draw[red,thick,dashed] (0,-.5)to[out=-170,in=140](60:-2)to[out=-40,in=-100](0,-.5); 
		\draw[red,thick,dashed](0,-.5)to[out=-170,in=90](-140:2)to[out=-90,in=-180](-90:2.5)to[out=0,in=-90](-40:2)to[out=90,in=-10](0,-.5);
		\draw[red,thick,dashed] (0,-.5)to[out=180,in=-90](160:2.5)to[out=90,in=180](90:2.5)to[out=0,in=150](30:3);
		\draw[red,thick,dashed] (0,-.5)to[bend right=10](3,0);
		\draw[red,thick,dashed] (0,.5)to[out=30,in=-90](60:1.5)to[out=90,in=0](90:2)to[out=180,in=90](120:1.5)to[out=-90,in=150](0,.5);
		\draw[red,thick,dashed,bend left=15] (30:3)to(0,.5);
		\draw[red,thick,dashed] (3,0)to[out=160,in=-10](0,.5);
		\draw[blue, thick] (0,-.5)to[out=0,in=-90](0.7,0)to[out=90,in=-30](0,1)to[out=150,in=-90](-0.4,1.5)to[out=90,in=180](0,1.9)to[out=0,in=140](.8,1)to[out=-40,in=150](3,0);
		\draw[blue] (2.2,.7)node{$\gamma^\times_1$};
		\draw[blue,thick] (3,0)[out=-100,in=20]to(-60:2.5)[out=-160,in=-60]to(65:-1.4);
		\draw[blue] (2.2,-1.1)node{$\gamma^\times_2$} (69:-1.47)node[rotate=-10]{\bf $+$};
		\draw[blue,thick] (65:-1.4)[out=120,in=210]to (110:2.2)[out=30,in=135]to(60:2.1)[out=-45,in=55]to(1.1,0)[out=-125,in=20]to(65:-1.4);
		\draw[blue] (-1.6,1)node{$\gamma^\times_3$};
		\draw (0,-.5)node{$\bullet$} (0,.5)node{$\bullet$} (30:3)node{$\bullet$} (3,0)node{$\bullet$}
		(65:-1.4)node{$\bullet$};
	\end{tikzpicture}\qquad
	\begin{tikzpicture}[scale=.9]
		\draw[ultra thick] (0,0) circle (3);
		\draw[ultra thick,fill=gray!20] (0,0) circle (.5);
		\draw[ultra thick,fill=gray!20](-70:1.7) circle (.2);
		\draw[ultra thick,fill=gray!20](90:1.5) circle (.2);
		\draw[red,thick,dashed] (0,-.5)to[out=-170,in=140](60:-2)to[out=-40,in=-100](0,-.5); 
		\draw[red,thick,dashed](0,-.5)to[out=-170,in=90](-140:2)to[out=-90,in=-180](-90:2.5)to[out=0,in=-90](-40:2)to[out=90,in=-10](0,-.5);
		\draw[red,thick,dashed] (0,-.5)to[out=180,in=-90](160:2.5)to[out=90,in=180](90:2.5)to[out=0,in=150](30:3);
		\draw[red,thick,dashed] (0,-.5)to[bend right=10](3,0);
		\draw[red,thick,dashed] (0,.5)to[out=30,in=-90](60:1.5)to[out=90,in=0](90:2)to[out=180,in=90](120:1.5)to[out=-90,in=150](0,.5);
		\draw[red,thick,dashed,bend left=15] (30:3)to(0,.5);
		\draw[red,thick,dashed] (3,0)to[out=160,in=-10](0,.5);
		\draw[blue, thick]
		(0,.5)to[out=170,in=85](-.8,0)to[out=-95,in=185](0,-.75)to[out=5,in=-90](0.7,0)to[out=90,in=-30](0,1)to[out=150,in=-90](-0.4,1.5)to[out=90,in=180](0,1.9)to[out=0,in=170](30:3);
		\draw[blue] (2,1.3)node{$\rho(\gamma^\times_1)$};
		\draw[blue,thick] (30:3)[out=-80,in=20]to(-60:2.5)[out=-160,in=-60]to(65:-1.4);
		\draw[blue] (2.1,-.7)node{$\rho(\gamma^\times_2)$};
		\draw[blue,thick] (65:-1.4)[out=120,in=210]to (110:2.3)[out=30,in=135]to(60:2.2)[out=-45,in=55]to(1.1,0)[out=-125,in=20]to(65:-1.4);
		\draw[blue] (-1.82,1)node{$\rho(\gamma^\times_3)$} (57:-1.28)node[rotate=20]{\bf $\times$} (70:-1.27)node[rotate=20]{\bf $\times$};
		\draw (0,-.5)node{$\bullet$} (0,.5)node{$\bullet$} (30:3)node{$\bullet$} (3,0)node{$\bullet$}
		(65:-1.4)node{$\bullet$};
	\end{tikzpicture}
	\caption{Examples of tagged permissible curves and their tagged rotations}\label{fig:exm2}
\end{figure}

By Construction~\ref{cons:mod}, the indecomposable representations $M(\gamma^\times_1)$, $M(\gamma^\times_2)$,  $M(\gamma^\times_3)$, $M(\rho(\gamma^\times_1))$, $M(\rho(\gamma^\times_2))$ and $M(\rho(\gamma^\times_3))$ are, respectively, the following.

$$\begin{array}{rl}
	M(\gamma^\times_1):&\qquad \xymatrix@R=1em{
		\ar@(lu,ld)[]_{0}0\ar@/^.5pc/[r]^{0}&0\ar@/^.5pc/[l]^{0}\ar[rd]_0&0\ar[l]_0\ar[r]^0&\k^2\ar[r]^{\left(\begin{smallmatrix}
				1&0\\0&1
			\end{smallmatrix}\right)}&\k^2\ar@(ru,rd)[]^{\left(\begin{smallmatrix}
				0&0\\1&0
			\end{smallmatrix}\right)}\\
		&&0&\k\ar[u]_{\left(\begin{smallmatrix}
				1\\0
			\end{smallmatrix}\right)}\ar[l]^0
	}\\
	M(\gamma^\times_2):&\qquad \xymatrix@R=1em{
		\ar@(lu,ld)[]_{1}\k\ar@/^.5pc/[r]^{0}&\k\ar@/^.5pc/[l]^{1}\ar[rd]_0&0\ar[l]_0\ar[r]^0&0\ar[r]^0&0\ar@(ru,rd)[]^{0}\\
		&&0&0\ar[u]_0\ar[l]^0
	}\\
	M(\gamma^\times_3):&\xymatrix@R=1em{
		\ar@(lu,ld)[]_{\left(\begin{smallmatrix}
				0&0\\0&0
			\end{smallmatrix}\right)}\k^2\ar@/^.5pc/[r]^{\left(\begin{smallmatrix}
				0&0\\0&1
			\end{smallmatrix}\right)}&\k^2\ar@/^.5pc/[l]^{\left(\begin{smallmatrix}
				1&0\\0&0
			\end{smallmatrix}\right)}\ar[rd]_{\left(\begin{smallmatrix}
				0&1
			\end{smallmatrix}\right)}&\k\ar[l]_{\left(\begin{smallmatrix}
				1\\0
			\end{smallmatrix}\right)}\ar[r]^1&\k\ar[r]^0&0\ar@(ru,rd)[]^{0}\\
		&&\k&\k\ar[u]_1\ar[l]^1
	}\\
	M(\rho(\gamma^\times_1)):&\xymatrix@R=1em{
		\ar@(lu,ld)[]_{\left(\begin{smallmatrix}
				0&0\\1&1
			\end{smallmatrix}\right)}\k^2\ar@/^.5pc/[r]^{\left(\begin{smallmatrix}
				0&0\\0&1
			\end{smallmatrix}\right)}&\k^2\ar@/^.5pc/[l]^{\left(\begin{smallmatrix}
				1&0\\0&0
			\end{smallmatrix}\right)}\ar[rd]_{\left(\begin{smallmatrix}
				0&1
			\end{smallmatrix}\right)}&\k\ar[l]_{\left(\begin{smallmatrix}
				1\\0
			\end{smallmatrix}\right)}\ar[r]^0&\k\ar[r]^{\left(\begin{smallmatrix}
				1\\0
			\end{smallmatrix}\right)}&\k^2\ar@(ru,rd)[]^{\left(\begin{smallmatrix}
				0&0\\1&0
			\end{smallmatrix}\right)}\\
		&&\k&\k\ar[u]_1\ar[l]^1
	}\\
	M(\rho(\gamma^\times_2)):&\qquad \xymatrix@R=1em{
		\ar@(lu,ld)[]_{0}\k\ar@/^.5pc/[r]^{0}&\k\ar@/^.5pc/[l]^{1}\ar[rd]_1&0\ar[l]_0\ar[r]^0&0\ar[r]^0&0\ar@(ru,rd)[]^{0}\\
		&&\k&\k\ar[u]_0\ar[l]^1
	}\\
	M(\rho(\gamma^\times_3)):&\xymatrix@R=1em{
		\ar@(lu,ld)[]_{\left(\begin{smallmatrix}
				1&0\\0&1
			\end{smallmatrix}\right)}\k^2\ar@/^.5pc/[r]^{\left(\begin{smallmatrix}
				0&0\\0&1
			\end{smallmatrix}\right)}&\k^2\ar@/^.5pc/[l]^{\left(\begin{smallmatrix}
				1&0\\0&0
			\end{smallmatrix}\right)}\ar[rd]_{\left(\begin{smallmatrix}
				0&1
			\end{smallmatrix}\right)}&\k\ar[l]_{\left(\begin{smallmatrix}
				1\\0
			\end{smallmatrix}\right)}\ar[r]^1&\k\ar[r]^0&0\ar@(ru,rd)[]^{0}\\
		&&\k&\k\ar[u]_1\ar[l]^1
	}
\end{array}$$
By Theorem~\ref{thm:tau}, we have $M(\rho(\gamma_i^\times))=\tau M(\gamma_i^\times)$, $i=1,2,3$, and the Auslander-Reiten sequence ending at $M(\gamma_2^\times)$ is
$$0\to M(\rho(\gamma_2^\times))\to M(\eta^\times)\to M(\gamma_2^\times)\to 0$$
where $\eta$ is shown in Figure~\ref{fig:exm4}.
\begin{figure}[htpb]
	\begin{tikzpicture}[scale=.9]
		\draw[ultra thick] (0,0) circle (3);
		\draw[ultra thick,fill=gray!20] (0,0) circle (.5);
		\draw[ultra thick,fill=gray!20](-70:1.7) circle (.2);
		\draw[ultra thick,fill=gray!20](90:1.5) circle (.2);
		\draw[red,thick,dashed] (0,-.5)to[out=-170,in=140](60:-2)to[out=-40,in=-100](0,-.5); 
		\draw[red,thick,dashed](0,-.5)to[out=-170,in=90](-140:2)to[out=-90,in=-180](-90:2.5)to[out=0,in=-90](-40:2)to[out=90,in=-10](0,-.5);
		\draw[red,thick,dashed] (0,-.5)to[out=180,in=-90](160:2.5)to[out=90,in=180](90:2.5)to[out=0,in=150](30:3);
		\draw[red,thick,dashed] (0,-.5)to[bend right=10](3,0);
		\draw[red,thick,dashed] (0,.5)to[out=30,in=-90](60:1.5)to[out=90,in=0](90:2)to[out=180,in=90](120:1.5)to[out=-90,in=150](0,.5);
		\draw[red,thick,dashed,bend left=15] (30:3)to(0,.5);
		\draw[red,thick,dashed] (3,0)to[out=160,in=-10](0,.5);
		\draw[blue, thick] (3,0)to[out=-100,in=0](.5,-2.3)to[out=180,in=-90](-.9,-1.4)to[out=90,in=180](-.6,-1)to[out=0,in=120](.2,-1.7)to[out=-60,in=-110](30:3);
		\draw[blue] (2.4,-1)node{$\eta^\times$};
		\draw (0,-.5)node{$\bullet$} (0,.5)node{$\bullet$} (30:3)node{$\bullet$} (3,0)node{$\bullet$}
		(65:-1.4)node{$\bullet$};
	\end{tikzpicture}
	\caption{An example of the middle term of an Auslander-Reiten sequence}\label{fig:exm4}
\end{figure}

The intersection numbers $\Int(\gamma^\times_1,\gamma^\times_3)=1$, $\Int(\gamma^\times_2,\gamma^\times_3)=2$, and $\Int(\gamma^\times_3,\gamma^\times_3)=0$. So by Theorem~\ref{thm:int-dim}, we have 
$$\dim_\k\Hom_A(M(\gamma_1^\times),\tau M(\gamma_3^\times))+\dim_\k\Hom_A(M(\gamma_3^\times),\tau M(\gamma_1^\times))=1,$$ $$\dim_\k\Hom_A(M(\gamma_2^\times),\tau M(\gamma_3^\times))+\dim_\k\Hom_A(M(\gamma_3^\times),\tau M(\gamma_2^\times))=2,$$
$$\dim_\k\Hom_A(M(\gamma_3^\times),\tau M(\gamma_3^\times))+\dim_\k\Hom_A(M(\gamma_3^\times),\tau M(\gamma_3^\times))=0.$$
In particular, $M(\gamma_3^\times)$ is an indecomposable $\tau$-rigid $A$-module.

The black intersection numbers are $\Int^\bullet(\gamma^\times_1,\gamma^\times_3)=1$, $\Int^\bullet(\gamma^\times_2,\gamma^\times_3)=1$ and $\Int^\bullet(\gamma^\times_3,\gamma^\times_2)=1$. So by Theorem~\ref{thm:bint-dim}, we have
$$\dim_\k\Hom_A(M(\gamma_1^\times),\tau M(\gamma_3^\times))=1,$$
$$\dim_\k\Hom_A(M(\gamma_2^\times),\tau M(\gamma_3^\times))=1,$$
$$\dim_\k\Hom_A(M(\gamma_3^\times),\tau M(\gamma_2^\times))=1.$$

Finally, consider the collection $R$ of tagged generalized permissible curves shown as solid curves in Figure~\ref{fig:exm3}. Because for any two general tagged permissible curves, their intersection number is $0$, we have that $R$ is a partial generalized dissection of $\surf$. Note that $|R|=8=|Q^\TT_0|+|Sp|=|A|=|\TT^\times|$. So by Corollary~\ref{cor:1}, $R$ is a generalized dissection of $\surf$. Hence by Corollary~\ref{cor:supp}, we have that $\bigoplus_{\gamma^\times\in R}M(\gamma^\times)$ is a basic support $\tau$-tilting $A$-module, but not a basic $\tau$-tilting module (because $R\cap\TT^\times\neq\emptyset$).

\begin{figure}[htpb]
	\begin{tikzpicture}[scale=.9]
		\draw[ultra thick] (0,0) circle (3);
		\draw[ultra thick,fill=gray!20] (0,0) circle (.5);
		\draw[ultra thick,fill=gray!20](-70:1.7) circle (.2);
		\draw[ultra thick,fill=gray!20](90:1.5) circle (.2);
		\draw[red,thick,dashed] (0,-.5)to[out=-170,in=140](60:-2)to[out=-40,in=-100](0,-.5); 
		\draw[red,thick,dashed](0,-.5)to[out=-170,in=90](-140:2)to[out=-90,in=-180](-90:2.5)to[out=0,in=-90](-40:2)to[out=90,in=-10](0,-.5);
		\draw[red,thick,dashed] (0,-.5)to[out=180,in=-90](160:2.5)to[out=90,in=180](90:2.5)to[out=0,in=150](30:3);
		\draw[red,thick,dashed] (0,-.5)to[bend right=10](3,0);
		\draw[red,thick,dashed] (0,.5)to[out=30,in=-90](60:1.5)to[out=90,in=0](90:2)to[out=180,in=90](120:1.5)to[out=-90,in=150](0,.5);
		\draw[red,thick,dashed,bend left=15] (30:3)to(0,.5);
		\draw[red,thick,dashed] (3,0)to[out=160,in=-10](0,.5);
		\draw[blue,thick] (3,0)[out=-100,in=20]to(-60:2.5)[out=-160,in=-60]to(65:-1.4);
		\draw[blue,thick] (65:-1.4)[out=120,in=210]to (110:2.2)[out=30,in=135]to(60:2.1)[out=-45,in=55]to(1.1,0)[out=-125,in=20]to(65:-1.4);
		\draw[blue,thick] (3,0)[bend left=20]to (65:-1.4);
		\draw[blue,thick] (65:-1.4)to[out=10,in=-90](30:3);
		\draw[blue,thick] (0,.5)to[out=30,in=0](0,1.85)to[out=180,in=150](0,.5);
		\draw[blue,thick] (0,-.5)to(65:-1.4);
		\draw[blue,thick] (30:3)to[out=150,in=0](90:2.4)to[out=180,in=90](-1.8,.8)to[out=-90,in=150](65:-1.4);
		\draw[blue,thick] (0,.5)to[out=-10,in=90] (.7,0)to[out=-90,in=35](65:-1.4);
		\draw (0,-.5)node{$\bullet$} (0,.5)node{$\bullet$} (30:3)node{$\bullet$} (3,0)node{$\bullet$}
		(65:-1.4)node{$\bullet$};
	\end{tikzpicture}
	\caption{An example of a generalized dissection}\label{fig:exm3}
\end{figure}

\end{document}